\documentclass[american]{amsart}

\usepackage[french,english]{babel}
\usepackage{a4wide}
\usepackage[utf8]{inputenc}
\usepackage{graphicx}
\usepackage{wrapfig}
\usepackage{subfig}
\usepackage[hyperref]{hyperref}

\usepackage{amssymb}
\usepackage{color}

\renewcommand{\phi}{\varphi}
\newcommand{\C}{{\mathbb{C}}}

\newcommand{\N}{{\mathbb{N}}}
\newcommand{\R}{{\mathbb{R}}}
\newcommand{\Q}{{\mathbb{Q}}}
\newcommand{\Z}{{\mathbb{Z}}}
\newcommand{\1}{{\mathbf{1}}}
\renewcommand{\epsilon}{\varepsilon}
\renewcommand{\theta}{\vartheta}
\renewcommand{\SS}{{\mathbb{S}}}
\newcommand{\T}{{\mathbb{T}}}

\newcommand{\complex}{{\mathcal{J}}}
\newcommand{\Disk}[1][{}]{{\mathbb{D}_{#1}}}

\newcommand{\moduli}{{\mathcal{M}}}

\newcommand{\lie}[1]{{\mathcal{L}_{#1}}}
\newcommand{\pairing}[2]{{\langle{#1},{#2}\rangle}}

\newcommand{\abs}[1]{{\left\lvert #1\right\rvert}}

\newcommand{\restricted}[2]{{\left.{#1}\right|_{#2}}}
\newcommand{\overtwisted}{{\mathbb{D}_\mathrm{OT}}}

\newcommand{\lcan}{{\lambda_{\mathrm{can}}}}
\newcommand{\acan}{{\alpha_{\mathrm{can}}}}

\newcommand{\dR}{{\operatorname{dR}}}

\newcommand{\p}{\partial}
\newcommand{\PD}{{\operatorname{PD}}}

\newcommand{\uU}{{\mathcal U}}
\newcommand{\secref}[1]{{\S\ref{#1}}}

\newcommand{\nbhd}{\mathcal{N}}
\newcommand{\interface}{\mathcal{I}}
\newcommand{\handle}{{H}} 
\newcommand{\Ring}{\mathcal{R}}
\newcommand{\HC}[1]{{\operatorname{HC}_*}\big(#1\big)}
\newcommand{\CC}[1]{{\operatorname{CC}_*}\big(#1\big)}
\newcommand{\annulus}{\mathbb{A}}
\newcommand{\Hom}{\operatorname{Hom}}
\newcommand{\dbar}{\bar{\p}}

\definecolor{Chris}{rgb}{0.5,0,1}

\DeclareMathOperator{\CZ}{CZ}

\DeclareMathOperator{\interior}{int}

\DeclareMathOperator{\Torsion}{Tor}

\numberwithin{equation}{section}

\theoremstyle{plain}
\newtheorem{theorem}{Theorem}
\newtheorem{corollary}{Corollary}
\newtheorem{question}{Question}
\newtheorem*{thmu}{Theorem}
\newtheorem{proposition}{Proposition}[section]
\newtheorem{thm}[proposition]{Theorem}

\newtheorem{lemma}[proposition]{Lemma}

\theoremstyle{remark}
\newtheorem{remark}[proposition]{Remark}
\newtheorem{example}[proposition]{Example}
\newtheorem*{remarku}{Remark}

\theoremstyle{definition}
\newtheorem{definition}{Definition}
\newtheorem{defn}[proposition]{Definition}

\begin{document}

\bibliographystyle{amsalpha}

\title{Weak Symplectic Fillings and Holomorphic Curves}

\author[K.\ Niederkrüger]{Klaus Niederkrüger}

\email[K.\ Niederkrüger]{niederkr@math.univ-toulouse.fr}

\address[K.\ Niederkrüger]{
  Institut de mathématiques de Toulouse\\
  Université Paul Sabatier -- Toulouse III\\
  118 route de Narbonne\\
  F-31062 Toulouse Cedex 9\\
  FRANCE}

\author[C.\ Wendl]{Chris Wendl}

\email[C.\ Wendl]{wendl@math.hu-berlin.de}

\address[C.\ Wendl]{Institut für Mathematik \\
  Humboldt-Universität zu Berlin \\
  10099 Berlin \\
  Germany}

\begin{abstract}
  \textbf{English:} We prove several results on weak symplectic
  fillings of contact $3$--manifolds, including: (1) Every weak
  filling of any planar contact manifold can be deformed to a blow up
  of a Stein filling.  (2) Contact manifolds that have fully
  separating planar torsion are not weakly fillable---this gives many
  new examples of contact manifolds without Giroux torsion that have
  no weak fillings.  (3) Weak fillability is preserved under splicing
  of contact manifolds along symplectic pre-Lagrangian tori---this
  gives many new examples of contact manifolds without Giroux torsion
  that are weakly but not strongly fillable.

  We establish the obstructions to weak fillings via two parallel
  approaches using holomorphic curves.  In the first approach, we
  generalize the original Gromov-Eliashberg ``Bishop disk'' argument
  to study the special case of Giroux torsion via a Bishop family of
  holomorphic annuli with boundary on an ``anchored overtwisted
  annulus''.  The second approach uses punctured holomorphic curves,
  and is based on the observation that every weak filling can be
  deformed in a collar neighborhood so as to induce a stable
  Hamiltonian structure on the boundary.  This also makes it possible
  to apply the techniques of Symplectic Field Theory, which we
  demonstrate in a test case by showing that the distinction between
  weakly and strongly fillable translates into contact homology as the
  distinction between twisted and untwisted coefficients.

  \begin{otherlanguage*}{francais}
    \textbf{Français:} On montre plusieurs résultats concernant les
    remplissages faibles de variétés de contact de dimension $3$,
    notamment : (1) Les remplissages faibles des variétés de contact
    planaires sont à déformation près des éclatements de remplissages
    de Stein.  (2) Les variétés de contact ayant de la torsion
    planaire et satisfaisant une certaine condition homologique
    n'admettent pas de remplissages faibles -- de cette manière on
    obtient des nouveaux exemples de variétés de contact qui ne sont
    pas faiblement remplissables.  (3) La remplissabilité faible est
    préservée par l'opération de somme connexe le long de tores
    pré-Lagrangiens --- ce qui nous donne beaucoup de nouveaux
    exemples de variétés de contact sans torsion de Giroux qui sont
    faiblement, mais pas fortement remplissables.

    On établit une obstruction à la remplissabilité faible avec deux
    approches qui utilisent des courbes holomorphes.  La première
    méthode se base sur l'argument original de Gromov-Eliashberg des
    \og disques de Bishop \fg{}.  On utilise une famille d'anneaux
    holomorphes s'appuyant sur un \og anneau vrillé ancré \fg{} pour
    étudier le cas spécial de la torsion de Giroux.  La deuxième
    méthode utilise des courbes holomorphes à pointes, et elle se base
    sur l'observation que dans un remplissage faible, la structure
    symplectique peut être déformée au voisinage du bord, en une
    structure Hamiltonienne stable.  Cette observation permet aussi
    d'appliquer les méthodes à la théorie symplectique de champs, et
    on montre dans un cas simple que la distinction entre les
    remplissabilités faible et forte se traduit en homologie de
    contact par une distinction entre coefficients tordus et non
    tordus.
  \end{otherlanguage*}
\end{abstract}

\maketitle

\setcounter{tocdepth}{2}
\tableofcontents

\setcounter{section}{-1}
\section{Introduction}

The study of symplectic fillings via $J$--holomorphic curves goes back
to the foundational result of Gromov \cite{Gromov_Kurven} and
Eliashberg \cite{Eliashberg_HoloDiscs}, which states that a closed
contact $3$--manifold that is overtwisted cannot admit a weak
symplectic filling.  Let us recall some important definitions: in the
following, we always assume that $(W,\omega)$ is a symplectic
$4$--manifold, and $(M,\xi)$ is an oriented $3$--manifold with a
positive and cooriented contact structure.  Whenever a contact form
for~$\xi$ is mentioned, we assume it is compatible with the given
coorientation.

\begin{definition}
  A contact $3$--manifold $(M,\xi)$ embedded in a symplectic
  $4$--manifold $(W,\omega)$ is called a \textbf{contact hypersurface}
  if there is a contact form $\alpha$ for~$\xi$ such that $d\alpha =
  \restricted{\omega}{TM}$.  In the case where $M = \p W$ and its
  orientation matches the natural boundary orientation, we say that
  $(W,\omega)$ has \textbf{contact type boundary} $(M,\xi)$, and if
  $W$ is also compact, we call $(W,\omega)$ a \textbf{strong
    symplectic filling} of $(M,\xi)$.
\end{definition}

\begin{definition}
  A contact $3$--manifold $(M,\xi)$ embedded in a symplectic
  $4$--manifold $(W,\omega)$ is called a \textbf{weakly contact
    hypersurface} if $\restricted{\omega}{\xi} > 0$, and in the
  special case where $M = \p W$ with the natural boundary orientation,
  we say that $(W,\omega)$ has \textbf{weakly contact boundary}
  $(M,\xi)$.  If $W$ is also compact, we call $(W,\omega)$ a
  \textbf{weak symplectic filling} of $(M,\xi)$.
\end{definition}

It is easy to see that a strong filling is also a weak filling.  In
general, a strong filling can also be characterized by the existence
in a neighborhood of $\p W$ of a transverse, outward pointing
\emph{Liouville vector field}, i.e.~a vector field $Y$ such that
$\lie{Y}\omega = \omega$.  The latter condition makes it possible to
identify a neighborhood of $\p W$ with a piece of the symplectization
of $(M,\xi)$; in particular, one can then enlarge $(W,\omega)$ by
symplectically attaching to $\p W$ a cylindrical end.

The Gromov-Eliashberg result was proved using a so-called \emph{Bishop
  family} of pseudoholomorphic disks: the idea was to show that in any
weak filling $(W,\omega)$ whose boundary contains an overtwisted disk,
a certain \emph{noncompact} $1$--parameter family of $J$--holomorphic
disks with boundary on $\p W$ must exist, but yields a contradiction
to Gromov compactness.  In \cite{Eliashberg_HoloDiscs}, Eliashberg
also used these techniques to show that all weak fillings of the tight
$3$--sphere are diffeomorphic to blow-ups of a ball.  More recently,
the Bishop family argument has been generalized by the first author
\cite{NiederkruegerPlastikstufe} to define the \emph{plastikstufe},
the first known obstruction to symplectic filling in higher
dimensions.

In the mean time, several finer obstructions to symplectic filling in
dimension three have been discovered, including some which obstruct
strong filling but not weak filling.  Eliashberg
\cite{Eliashberg_torus} used some of Gromov's classification results
for symplectic $4$--manifolds \cite{Gromov_Kurven} to show that on the
$3$--torus, the standard contact structure is the only one that is
strongly fillable, though Giroux had shown \cite{Giroux_plusOuMoins}
that it has infinitely many distinct weakly fillable contact
structures.  The first examples of tight contact structures without
weak fillings were later constructed by Etnyre and Honda
\cite{EtnyreHonda_weakly}, using an obstruction due to Paolo Lisca
\cite{Lisca_fillings} based on Seiberg-Witten theory.

The simplest filling obstruction beyond overtwisted disks is the
following.  Define for each $n\in \N$ the following contact
$3$--manifolds with boundary:
\begin{equation*}
  T_n :=  \bigl(\T^2\times [0,n],\,
  \sin (2\pi z)\, d\phi + \cos (2\pi z)\,d\theta\bigr)\; ,
\end{equation*}
where $(\phi, \theta)$ are the coordinates on $\T^2 = \SS^1 \times
\SS^1$, and $z$ is the coordinate on $[0,n]$.  We will refer to $T_n$
as a \textbf{Giroux torsion domain}.

\begin{figure}[htbp]
  \centering
  \includegraphics[width=5cm,keepaspectratio]{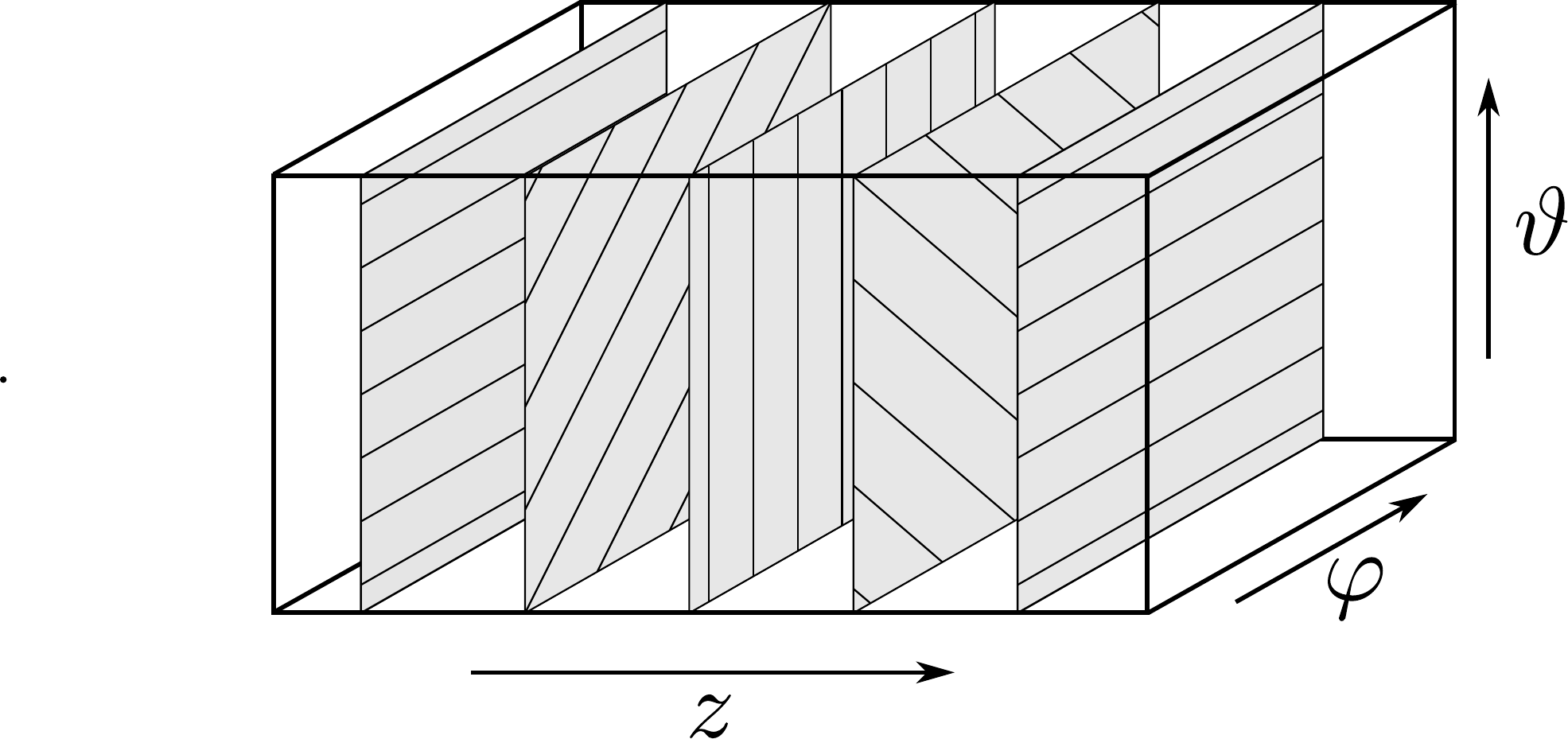}
  \caption{The region between the grey planes on either side
    represents half a Giroux torsion domain.  The grey planes are
    pre-Lagrangian tori with their characteristic foliations, which
    show the contact structure turning along the $z$--axis as we move
    from left to right.  Domains with higher Giroux torsion can be
    constructed by gluing together several half-torsion
    domains.}\label{fig: half-torsion domain}
\end{figure}

\begin{definition}
  Let $(M,\xi)$ be a $3$--dimensional contact manifold.  The
  \textbf{Giroux torsion} $\Torsion(M, \xi) \in \Z\cup \{\infty\}$ is
  the largest number $n \ge 0$ for which we can find a contact
  embedding of the Giroux torsion domain $T_n \hookrightarrow M$.  If
  this is true for arbitrarily large $n$, then we define $\Torsion(M,
  \xi) = \infty$.
\end{definition}

\begin{remarku}
  Due to the classification result of Eliashberg
  \cite{Eliashberg_Overtwisted}, overtwisted contact manifolds have
  infinite Giroux torsion, and moreover, one can assume in this case
  that the torsion domain $T_n \subset M$ separates~$M$.  It is not
  known whether a contact manifold with infinite Giroux torsion must
  be overtwisted in general.
\end{remarku}

The present paper was motivated partly by the following fairly recent
result.

\begin{thmu}[Gay~\cite{Gay_GirouxTorsion} and
  Ghiggini-Honda~\cite{GhigginiHonda_twisted}]
  A closed contact $3$--manifold $(M,\xi)$ with positive Giroux
  torsion does not have a strong symplectic filling.  Moreover, if it
  contains a Giroux torsion domain $T_n$ that splits $M$ into separate
  path components, then $(M,\xi)$ does not even admit a \emph{weak}
  filling.
\end{thmu}

The first part of this statement was proved originally by David Gay
with a gauge theoretic argument, and the refinement for the separating
case follows from a computation of the Ozsváth-Szabó contact invariant
due to Paolo Ghiggini and Ko Honda.  Observe that due to the remark
above on overtwistedness and Giroux torsion, the result implies the
Eliashberg-Gromov theorem.

As this brief sampling of history indicates, holomorphic curves have
not been one of the favorite tools for defining filling obstructions
in recent years.  One might argue that this is unfortunate, because
holomorphic curve arguments have a tendency to seem more geometrically
natural and intuitive than those involving the substantial machinery
of Seiberg-Witten theory or Heegaard Floer homology---and in higher
dimensions, of course, they are still the only tool available.  A
recent exception was the paper \cite{ChrisGirouxTorsion}, where the
second author used families of holomorphic cylinders to provide a new
proof of Gay's result on Giroux torsion and strong fillings.  By
similar methods, the second author has recently defined a more general
obstruction to strong fillings \cite{ChrisOpenBook2}, called
\emph{planar torsion}, which provides many new examples of contact
manifolds $(M,\xi)$ with $\Torsion(M,\xi) = 0$ that are nevertheless
not strongly fillable.  The reason these results apply primarily to
\emph{strong} fillings is that they depend on moduli spaces of
\emph{punctured} holomorphic curves, which live naturally in the
noncompact symplectic manifold obtained by attaching a cylindrical end
to a strong filling.  By contrast, the Eliashberg-Gromov argument
works also for weak fillings because it uses compact holomorphic
curves with boundary, which live naturally in a compact almost complex
manifold with boundary that is pseudoconvex, but not necessarily
convex in the \emph{symplectic} sense.  The Bishop family argument
however has never been extended for any compact holomorphic curves
more general than disks, because these tend to live in moduli spaces
of nonpositive virtual dimension.

In this paper, we will demonstrate that both approaches, via compact
holomorphic curves with boundary as well as punctured holomorphic
curves, can be used to prove much more general results involving
\emph{weak} symplectic fillings.  As an illustrative example of the
compact approach, we shall begin in \secref{sec:annulus} by presenting
a new proof of the above result on Giroux torsion, as a consequence of
the following.

\begin{theorem}\label{thm:BishopGirouxTorsion}
  Let $(M, \xi)$ be a closed $3$--dimensional contact manifold
  embedded into a closed symplectic $4$--manifold $(W,\omega)$ as a
  weakly contact hypersurface.  If $(M,\xi)$ contains a Giroux torsion
  domain $T_n \subset M$, then the restriction of the symplectic form
  $\omega$ to $T_n$ cannot be exact.
\end{theorem}

By a theorem of Eliashberg \cite{Eliashberg_capping} and Etnyre
\cite{Etnyre_capping}, every weak filling can be capped to produce a
closed symplectic $4$--manifold.  The above statement thus implies a
criterion for $(M,\xi)$ to be not weakly fillable---our proof will in
fact demonstrate this directly, without any need for the capping
result.  We will use the fact that every Giroux torsion domain
contains an object that we call an \emph{anchored overtwisted
  annulus}, which we will show serves as a filling obstruction
analogous to an overtwisted disk.  Note that for a torsion domain $T_n
\subset M$, the condition that $\omega$ is exact on~$T_n$ is
equivalent to the vanishing of the integral
\begin{equation*}
  \int_{\T^2\times\{c\}} \omega
\end{equation*}
on any slice $T^2 \times \{c\} \subset T_n$.  For a strong filling
this is \emph{always} satisfied since $\omega$ is exact on the
boundary, and it is also always satisfied if $T_n$ separates~$M$.

The proof of Theorem~\ref{thm:BishopGirouxTorsion} is of some interest
in itself for being comparatively low-tech, which is to say that it
relies only on technology that was already available as of 1985.  As
such, it demonstrates new potential for well established techniques,
in particular the Gromov-Eliashberg Bishop family argument, which we
shall generalize by considering a ``Bishop family of holomorphic
annuli'' with boundaries lying on a $1$--parameter family of so-called
\emph{half-twisted annuli}.  Unlike overtwisted disks, a single
overtwisted annulus does not suffice to prove anything: the boundaries
of the Bishop annuli must be allowed to vary in a nontrivial family,
called an \emph{anchor}, so as to produce a moduli space with positive
dimension.  One consequence of this extra degree of freedom is that
the required energy bounds are no longer automatic, but in fact are
only satisfied when $\omega$ satisfies an extra cohomological
condition.  This is one way to understand the geometric reason why
Giroux torsion always obstructs strong fillings, but only obstructs
weak fillings in the presence of extra topological conditions.  This
method also provides some hope of being generalizable to higher
dimensions, where the known examples of filling obstructions are still
very few.

In \secref{sec:punctured}, we will initiate the study of weak fillings
via punctured holomorphic curves in order to obtain more general
results.  The linchpin of this approach is Theorem~\ref{theorem:
  stableHypersurface} in \secref{sec: collar neighborhood weak
  boundary}, which says essentially that any weak filling can be
deformed so that its boundary carries a stable Hamiltonian structure.
This is almost as good as a strong filling, as one can then
symplectically attach a cylindrical end---but extra cohomological
conditions are usually needed in order to do this without losing the
ability to construct nice holomorphic curves in the cylindrical end.
It turns out that the required conditions are \emph{always} satisfied
for planar contact manifolds, and we obtain the following surprising
generalization of a result proved for strong fillings in
\cite{ChrisGirouxTorsion}.

\begin{theorem}\label{thm:planar}
  If $(M,\xi)$ is a planar contact $3$--manifold, then every weak
  filling of $(W,\omega)$ is symplectically deformation equivalent to
  a blow up of a Stein filling of $(M,\xi)$.
\end{theorem}

\begin{corollary}\label{cor:planar}
  If $(M,\xi)$ is weakly fillable but not Stein fillable, then it is
  not planar.
\end{corollary}

\begin{corollary}\label{cor:Dehntwists}
  Given any planar open book supporting a contact manifold $(M,\xi)$,
  the manifold is weakly fillable if and only if the monodromy of the
  open book can be factored into a product of positive Dehn twists.
\end{corollary}

The second corollary follows easily from the result proved in
\cite{ChrisGirouxTorsion}, that every planar open book on a strongly
fillable contact manifold can be extended to a Lefschetz fibration of
the filling over the disk.  This fact was used in recent work of Olga
Plamenevskaya and Jeremy Van Horn-Morris \cite{PlamenevskayaVanHorn}
to find new examples of planar contact manifolds that have either
unique fillings or no fillings at all.  Theorem~\ref{thm:planar} in
fact reduces the classification question for weak fillings of planar
contact manifolds to the classification of Stein fillings, and as
shown in \cite{ChrisFiberSums} using the results in
\cite{ChrisGirouxTorsion}, the latter reduces to an essentially
combinatorial question involving factorizations of monodromy maps into
products of positive Dehn twists.  Note that most previous
classification results for weak fillings
(e.g.~\cite{Eliashberg_HoloDiscs, Lisca_Lens, PlamenevskayaVanHorn})
have applied to rational homology spheres, as it can be shown
homologically in such settings that weak fillings are always
deformable to strong ones.  Theorem~\ref{thm:planar} makes no such
assumption about the topology of~$M$.

\begin{remarku}
  It is easy to see that nothing like Theorem~\ref{thm:planar} holds
  for non-planar contact manifolds in general.  There are of course
  many examples of weakly but not strongly fillable contact manifolds;
  still more will appear in the results stated below.  There are also
  Stein fillable contact manifolds with weak fillings that cannot be
  deformed into blown up Stein fillings: for instance, Giroux shows in
  \cite{Giroux_plusOuMoins} that the standard contact $3$--torus
  $(\T^3,\xi_1)$ admits weak fillings diffeomorphic to $\Sigma \times
  \T^2$ for any compact oriented surface $\Sigma$ with connected
  boundary.  As shown in \cite{ChrisGirouxTorsion} however,
  $(\T^3,\xi_1)$ has only one Stein filling, diffeomorphic to $\Disk
  \times \T^2$, and if $\Sigma \ne \Disk$ then $\Sigma\times \T^2$ is
  not homeomorphic to any blow-up of $\Disk \times \T^2$, since
  $\pi_2(\Sigma \times \T^2) = 0$.
\end{remarku}

Using similar methods, \secref{sec:punctured} will also generalize
Theorem~\ref{thm:BishopGirouxTorsion} to establish a new obstruction
to weak symplectic fillings in dimension three.  We will recall in
\S\ref{subsec:review} the definition of a planar torsion domain, which
is a generalization of a Giroux torsion domain that furnishes an
obstruction to strong filling by a result in \cite{ChrisOpenBook2}.
The same will not be true for weak fillings, but becomes true after
imposing an extra homological condition: for any closed $2$--form
$\Omega$ on~$M$, one says that~$M$ has \emph{$\Omega$--separating}
planar torsion if
\begin{equation*}
  \int_L \Omega = 0
\end{equation*}
for every torus~$L$ in a certain special set of disjoint tori in the
torsion domain.

\begin{theorem}\label{theorem: planarTorsion}
  Suppose $(M,\xi)$ is a closed contact $3$--manifold with
  $\Omega$--separating planar torsion for some closed $2$--form
  $\Omega$ on~$M$.  Then $(M,\xi)$ admits no weakly contact type
  embedding into a closed symplectic $4$--manifold $(W,\omega)$ with
  $\restricted{\omega}{TM}$ cohomologous to~$\Omega$.  In particular,
  $(M,\xi)$ has no weak filling $(W,\omega)$ with
  $[\restricted{\omega}{TM}] = [\Omega]$.
\end{theorem}

As is shown in \cite{ChrisOpenBook2}, any Giroux torsion domain
embedded in a closed contact manifold has a neighborhood that contains
a planar torsion domain, thus Theorem~\ref{theorem: planarTorsion}
implies another proof of Theorem~\ref{thm:BishopGirouxTorsion}.  If
each of the relevant tori $L \subset M$ separates~$M$, then $\int_L
\Omega = 0$ for all~$\Omega$ and we say that $(M,\xi)$ has \emph{fully
  separating} planar torsion.

\begin{corollary}\label{cor:obstruction}
  If $(M,\xi)$ is a closed contact $3$--manifold with fully separating
  planar torsion, then it admits no weakly contact type embedding into
  any closed symplectic $4$--manifold.  In particular, $(M,\xi)$ is
  not weakly fillable.
\end{corollary}

\begin{remarku}
  The statement about non-fillability in
  Corollary~\ref{cor:obstruction} also follows from a recent
  computation of the twisted ECH contact invariant that has been
  carried out in parallel work of the second author
  \cite{ChrisOpenBook2}.  The proof via ECH is however extremely
  indirect, as according to the present state of technology it
  requires the isomorphism established by Taubes \cite{Taubes_ECH5}
  from ECH to monopole Floer homology, together with results of
  Kronheimer and Mrowka \cite{KronheimerMrowka_contact} that relate
  the monopole invariants to weak fillings.  Our proof on the other
  hand will require no technology other than holomorphic curves.
\end{remarku}

We now show that there are many contact manifolds without Giroux
torsion that satisfy the above hypotheses.  Consider a closed oriented
surface
\begin{equation*}
  \Sigma = \Sigma_+ \cup_\Gamma \Sigma_-
\end{equation*}
obtained as the union of two (not necessarily connected) surfaces
$\Sigma_\pm$ with boundary along a multicurve $\Gamma \ne \emptyset$.
By results of Lutz \cite{Lutz_CircleActions}, the $3$--manifold $\SS^1
\times \Sigma$ admits a unique (up to isotopy) $\SS^1$--invariant
contact structure $\xi_\Gamma$ such that the surfaces $\{*\} \times
\Sigma$ are all convex and have $\Gamma$ as the dividing set.  If
$\Gamma$ has no component that bounds a disk, then the manifold
$(\SS^1 \times \Sigma, \xi_\Gamma)$ is tight
\cite[Proposition~4.1]{Giroux_cercles}, and
if $\Gamma$ also has no two connected components that are isotopic
in~$\Sigma$, then it follows from arguments due to Giroux (see
\cite{Massot_vanishing}) that $(\SS^1 \times \Sigma, \xi_\Gamma)$ does
not even have Giroux torsion.  But as we will review in
\secref{subsec:review}, it is easy to construct examples that satisfy
these conditions and have planar torsion.
  
\begin{corollary}\label{cor:notWeakly}
  For the $\SS^1$--invariant contact manifold $(\SS^1 \times
  \Sigma,\xi_\Gamma)$ described above, suppose the following
  conditions are satisfied (see Figure~\ref{fig:notWeakly}):
  \begin{enumerate}
  \item $\Gamma$ has no contractible components and no pair of
    components that are isotopic in~$\Sigma$.
  \item $\Sigma_+$ contains a connected component $\Sigma_P \subset
    \Sigma_+$ of genus zero, whose boundary components each
    separate~$\Sigma$.
  \end{enumerate}
  Then $(\SS^1 \times \Sigma,\xi_\Gamma)$ has no Giroux torsion and is
  not weakly fillable.
\end{corollary}

\begin{wrapfigure}{r}{0.4\textwidth}
  \vspace{-10pt}
  \begin{center}
    \includegraphics[width=0.38\textwidth,
    keepaspectratio]{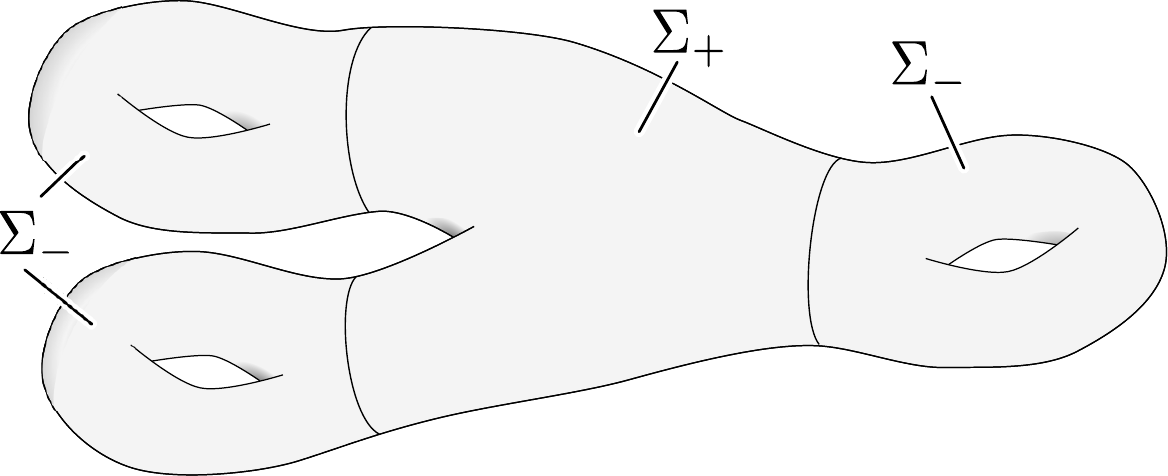}
  \end{center}
  \vspace{-10pt}
  \caption{An example of a surface $\Sigma$ and multicurve $\Gamma
    \subset \Sigma$ satisfying the conditions of
    Corollary~\ref{cor:notWeakly}.}\label{fig:notWeakly}
  \vspace{-10pt}
\end{wrapfigure}

The example of the tight $3$--tori shows that the homological
condition in the Giroux torsion case cannot be relaxed, and indeed,
the first historical examples of weakly but not strongly fillable
contact structures can in hindsight be understood via the distinction
between separating and non-separating Giroux torsion.  In
\secref{sec:handles}, we will introduce a new symplectic handle
attachment technique that produces much more general examples of weak
fillings:

\begin{theorem}\label{thm:weakConstruction}
  Suppose $(W,\omega)$ is a (not necessarily connected) weak filling
  of a contact $3$--manifold $(M,\xi)$, and $T \subset M$ is an
  embedded oriented torus which is pre-Lagrangian in $(M,\xi)$ and
  symplectic in $(W,\omega)$.  Then:
  \begin{enumerate}
  \item $(W,\omega)$ is also a weak filling of every contact manifold
    obtained from $(M,\xi)$ by performing finitely many Lutz twists
    along~$T$.
  \item If $T' \subset M$ is another torus satisfying the stated
    conditions, disjoint from~$T$, such that $\int_T \omega =
    \int_{T'} \omega$, then the contact manifold obtained from
    $(M,\xi)$ by splicing along $T$ and $T'$ is also weakly fillable.
  \end{enumerate}
\end{theorem}

See \secref{sec:handles} for precise definitions of the Lutz twist and
splicing operations, as well as more precise versions of
Theorem~\ref{thm:weakConstruction}.  We will use the theorem to
explicitly construct new examples of contact manifolds that are weakly
but not strongly fillable, including some that have planar torsion but
no Giroux torsion.  Let
\begin{equation*}
  \Sigma = \Sigma_+ \cup_\Gamma \Sigma_-
\end{equation*}
be a surface divided by a multicurve $\Gamma$ into two parts as
described above.  The principal circle bundles $P_{\Sigma,e}$ over
$\Sigma$ are distinguished by their Euler number $e = e(P) \in \Z$
which can be easily determined by removing a solid torus around a
fiber of $P_{\Sigma,e}$, choosing a section outside this neighborhood,
and computing the intersection number of the section with a meridian
on the torus.  The Euler number thus measures how far the bundle is
from being trivial.  Lutz \cite{Lutz_CircleActions} also showed that
every nontrivial $\SS^1$--principal bundle $P_{\Sigma,e}$ with Euler
number $e$ over $\Sigma$ admits a unique (up to isotopy)
$\SS^1$--invariant contact structure $\xi_{\Gamma, e}$ that is tangent
to fibers over the multicurve $\Gamma$ and is everywhere else
transverse.  For simplicity, we will continue to write $\xi_\Gamma$
for the corresponding contact structure $\xi_{\Gamma,0}$ on the
trivial bundle $P_{\Sigma,0} = \SS^1 \times \Sigma$.

\begin{theorem}\label{thm:weakFillings}
  Suppose $\bigl(P_{\Sigma,e},\xi_{\Gamma,e}\bigr)$ is the
  $\SS^1$--invariant contact manifold described above, for some
  multicurve $\Gamma \subset \Sigma$ whose connected components are
  all non-separating.  Then $\bigl(P_{\Sigma,e},\xi_{\Gamma,e}\bigr)$
  is weakly fillable.
\end{theorem}

\begin{corollary}\label{cor:weakNotStrong}
  There exist contact $3$--manifolds without Giroux torsion that are
  weakly but not strongly fillable.  In particular, this is true for
  the $\SS^1$--invariant contact manifold $(\SS^1 \times
  \Sigma,\xi_\Gamma)$ whenever all of the following conditions are
  met:
  \begin{enumerate}
  \item $\Gamma$ has no connected components that separate~$\Sigma$,
    and no pair of connected components that are isotopic in~$\Sigma$,
  \item $\Sigma_+$ has a connected component of genus zero,
  \item Either of the following is true:
    \begin{enumerate}
    \item $\Sigma_+$ or $\Sigma_-$ is disconnected,
    \item $\Sigma_+$ and $\Sigma_-$ are not diffeomorphic to each
      other.
    \end{enumerate}
  \end{enumerate}
\end{corollary}

\begin{remarku}
  Our proof of Theorem~\ref{thm:weakFillings} will actually produce
  not just a weak filling of $\bigl(P_{\Sigma,e},\xi_{\Gamma,e}\bigr)$
  but also a connected weak filling of a disjoint union of this with
  another contact $3$--manifold.  By Etnyre's obstruction
  \cite{Etnyre_planar} (or by Theorem~\ref{thm:planar}), it follows
  that $\bigl(P_{\Sigma,e}, \xi_{\Gamma,e}\bigr)$ is not planar
  whenever $\Gamma \subset \Sigma$ has no separating component.
\end{remarku}

\begin{figure}
  \centering
  \subfloat[]{\includegraphics[width=0.4\textwidth,
    keepaspectratio]{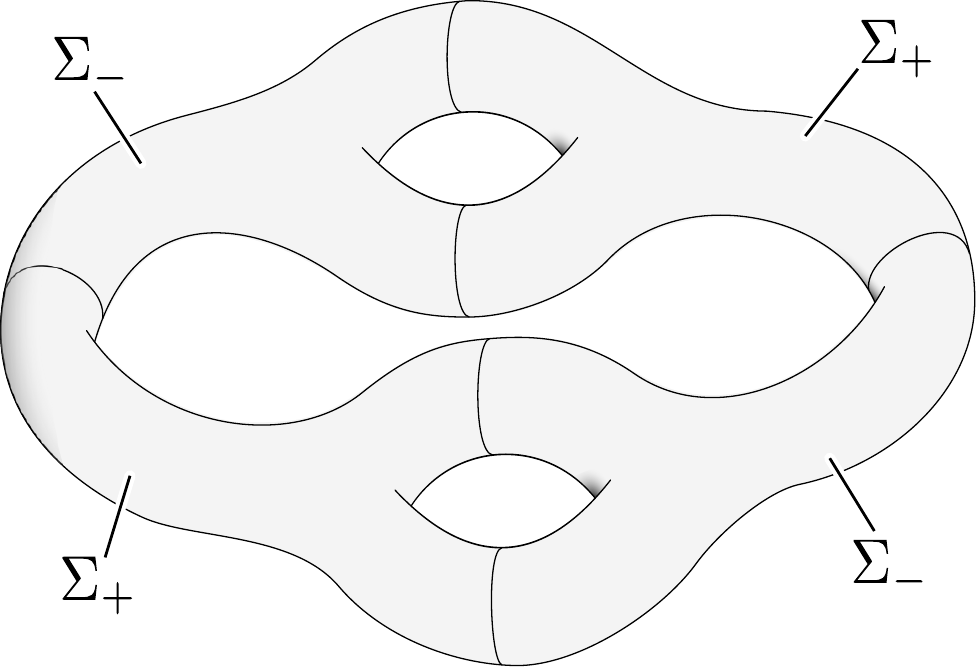}}
  \parbox{0.05\textwidth}{\quad}
  \subfloat[]{\includegraphics[width=0.4\textwidth,
    keepaspectratio]{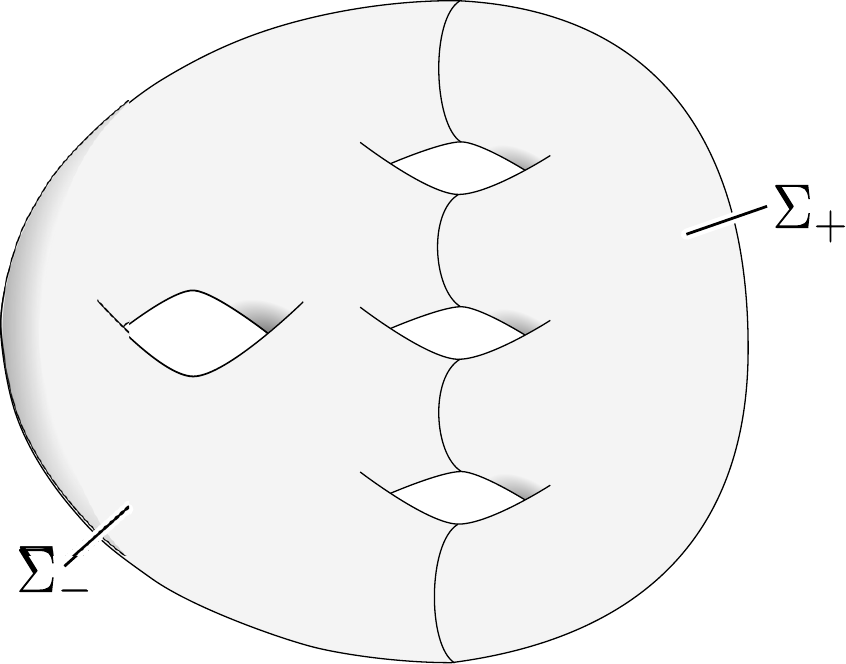}}
  \caption{Surfaces $\Sigma = \Sigma_+ \cup_\Gamma \Sigma_-$ which
    yield $\SS^1$--invariant contact manifolds $(\SS^1 \times
    \Sigma,\xi_\Gamma)$ that are weakly but not strongly fillable due
    to Corollary~\ref{cor:weakNotStrong}.}\label{fig:WeakNotStrong}
\end{figure}

One further implication of the techniques introduced in
\secref{sec:punctured} is that weak fillings can now be studied using
the technology of Symplectic Field Theory.  The latter is a general
framework introduced by Eliashberg, Givental and Hofer
\cite{SymplecticFieldTheory} for defining contact invariants by
counting $J$--holomorphic curves in symplectizations and in noncompact
symplectic cobordisms with cylindrical ends.  In joint work of the
second author with Janko Latschev \cite{LatschevWendl}, it is shown
that SFT contains an algebraic variant of planar torsion, which gives
an infinite hierarchy of obstructions to the existence of strong
fillings and exact symplectic cobordisms in all
dimensions.\footnote{Examples are as yet only known in dimension
  three, with the exception of \emph{algebraic overtwistedness}, see
  \cite{PSOvertwistedIsAlgebraicallyOvertwisted} and
  \cite{ContactHomologyLeftHanded}.}  Stable Hamiltonian structures
can be used to incorporate weak fillings into this picture as well:
analogously to the situation in Heegaard Floer homology, the
distinction between strong and weak is then seen algebraically via
twisted (i.e.~group ring) coefficients in SFT.

We will explain a special case of this statement in
\secref{subsec:SFT}, focusing on the simplest and most widely known
invariant defined within the SFT framework: contact homology.  Given a
contact manifold $(M,\xi)$, the contact homology $\HC{M,\xi}$ can be
defined as a $\Z_2$--graded supercommutative algebra with unit: it is
the homology of a differential graded algebra generated by Reeb orbits
of a nondegenerate contact form, where the differential counts rigid
$J$--holomorphic spheres with exactly one positive end and arbitrarily
many negative ends.  (See \secref{subsec:SFT} for more precise
definitions.)  We say that the homology \emph{vanishes} if it
satisfies the relation $\1 = 0$, which implies that it contains only
one element.  In defining this algebra, one can make various choices
of coefficients, and in particular for any linear subspace $\Ring
\subset H_2(M ; \R)$, one can define contact homology as a module over
the group ring\footnote{In the standard presentation of contact
  homology, one usually requires the subspace $\Ring \subset
  H_2(M;\R)$ to lie in the kernel of $c_1(\xi)$, however this is only
  needed if one wants to lift the canonical $\Z_2$--grading to a
  $\Z$--grading, which is unnecessary for our purposes.}
\begin{equation*}
  \Q[H_2(M;\R) / \Ring] = \left\{ \sum_{i=1}^N c_i e^{A_i}
    \ \Bigm|\ c_i \in \Q,\ A_i \in H_2(M;\R) / \Ring \right\} \;,
\end{equation*}
with the differential ``twisted'' by inserting factors of $e^A$ to
keep track of the homology classes of holomorphic curves.  We will
denote the contact homology algebra defined in this way for a given
subspace $\Ring \subset H_2(M;\R)$ by
\begin{equation*}
  \HC{M,\xi ;\, \Q[H_2(M;\R) / \Ring]} \;.
\end{equation*}
There are two obvious special cases that must be singled out: if
$\Ring = H_2(M;\R)$, then the coefficients reduce to~$\Q$, and we
obtain the \textbf{untwisted} contact homology $\HC{M,\xi ;\, \Q}$, in
which the group ring does not appear.  If we instead set $\Ring =
\{0\}$, the result is the \textbf{fully twisted} contact homology
$\HC{M,\xi ;\, \Q[H_2(M;\R)]}$, which is a module over
$\Q[H_2(M;\R)]$.  There is also an intermediately twisted version
associated to any cohomology class $\beta \in H^2_\dR(M)$, namely
$\HC{M,\xi ;\, \Q[H_2(M;\R) / \ker\beta]}$, where we identify $\beta$
with the induced linear map $H_2(M;\R) \to \R, \, A \mapsto
\pairing{\beta}{A}$.  Observe that the canonical projections
$\Q[H_2(M;\R)] \to \Q[H_2(M;\R) / \ker\beta] \to \Q$ yield algebra
homomorphisms
\begin{equation*}
  \HC{M,\xi;\, \Q[H_2(M;\R)]} \to \HC{M,\xi;\,
    \Q[H_2(M;\R) / \ker\beta]} \to \HC{M,\xi;\, \Q} \;,
\end{equation*}
implying in particular that whenever the fully twisted version
vanishes, so do all the others.  The choice of twisted coefficients
then has the following relevance for the question of fillability.

\begin{theorem}\footnote{While the fundamental concepts of Symplectic
    Field Theory are now a decade old, its analytical foundations
    remain work in progress (cf.~\cite{Hofer_polyfoldSurvey}), and it
    has meanwhile become customary to gloss over this fact while using
    the conceptual framework of SFT to state and ``prove'' theorems.
    We do not entirely mean to endorse this custom, but at the same
    time we have followed it in the discussion surrounding
    Theorem~\ref{thm:SFT}, which really should be regarded as a
    \emph{conjecture} for which we will provide the essential elements
    of the proof, with the expectation that it will become fully
    rigorous as soon as the definition of the theory is complete.}
\label{thm:SFT}
Suppose $(M,\xi)$ is a closed contact $3$--manifold with a cohomology
class $\beta \in H^2_\dR(M)$ for which $\HC{M,\xi ;\, \Q[H_2(M;\R) /
  \ker\beta]}$ vanishes.  Then $(M,\xi)$ does not admit any weak
symplectic filling $(W,\omega)$ with $[\restricted{\omega}{TM}] =
\beta$.
\end{theorem}

Since weak fillings that are exact near the boundary are equivalent to
strong fillings up to symplectic deformation (cf.~Proposition~3.1 in
\cite{EliashbergContactProperties}), the special case $\beta=0$ means
that the \emph{untwisted} contact homology gives an obstruction to
strong filling, and we similarly obtain an obstruction to weak filling
from the \emph{fully twisted} contact homology:

\begin{corollary}\label{cor:twisted}
  For any closed contact $3$--manifold $(M,\xi)$:
  \begin{enumerate}
  \item If $\HC{M,\xi ;\, \Q}$ vanishes, then $(M,\xi)$ is not
    strongly fillable.
  \item If $\HC{M,\xi ;\, \Q[H_2(M;\R)]}$ vanishes, then $(M,\xi)$ is
    not weakly fillable.
  \end{enumerate}
\end{corollary}

This result does not immediately yield any new knowledge about contact
topology, as so far the overtwisted contact manifolds are the only
examples in dimension~$3$ for which any version (in particular the
twisted version) of contact homology is known to vanish,
cf.~\cite{YauContactHomologyVanishes} and \cite{ChrisOpenBook2}.
We've included it here merely as a ``proof of concept'' for the use of
SFT with twisted coefficients to study weak fillings.  For the higher
order algebraic filling obstructions defined in \cite{LatschevWendl},
there are indeed examples where the twisted and untwisted theories
differ, corresponding to tight contact manifolds that are weakly but
not strongly fillable.

We conclude this introduction with a brief discussion of open
questions.  

Insofar as planar torsion provides an obstruction to weak filling, it
is natural to wonder how sharp the homological condition in
Theorem~\ref{theorem: planarTorsion} is.  The most obvious test cases
are the $\SS^1$--invariant product manifolds $(\SS^1 \times \Sigma,
\xi_\Gamma)$, under the assumption that $\Sigma \setminus \Gamma$
contains a connected component of genus zero, as for these the
question of strong fillability is completely understood by results in
\cite{ChrisOpenBook2} and \cite{ChrisFiberSums}.
Theorems~\ref{theorem: planarTorsion} and~\ref{thm:weakFillings} give
criteria when such manifolds either are or are not weakly fillable,
but there is still a grey area in which neither result applies,
e.g.~neither is able to settle the following:

\begin{question}
  Suppose $\Sigma = \Sigma_+ \cup_\Gamma \Sigma_-$, where $\Sigma
  \setminus \Gamma$ contains a connected component of genus zero and
  some connected components of $\Gamma$ separate $\Sigma$, while
  others do not.  Is $(\SS^1 \times \Sigma, \xi_\Gamma)$ weakly
  fillable?
\end{question}

Another question concerns the classification of weak fillings: on
rational homology spheres this reduces to a question about strong
fillings, and Theorem~\ref{thm:planar} reduces it to the Stein case
for all planar contact manifolds, which makes general classification
results seem quite realistic.  But already in the simple case of the
tight $3$--tori, one can combine explicit examples such as $\Sigma
\times \T^2$ with our splicing technique to produce a seemingly
unclassifiable zoo of inequivalent weak fillings.  Note that the
splicing technique can be applied in general for contact manifolds
that admit fillings with homologically nontrivial pre-Lagrangian tori,
and these are \emph{never} planar, because due to an obstruction of
Etnyre \cite{Etnyre_planar} fillings of planar contact manifolds must
have trivial~$b_2^0$.

\begin{question}
  Other than rational homology spheres, are there any non-planar
  weakly fillable contact $3$--manifolds for which weak fillings can
  reasonably be classified?
\end{question}

On the algebraic side, it would be interesting to know whether
Theorem~\ref{thm:SFT} actually implies any contact topological results
that are not known; this relates to the rather important open question
of whether there exist tight contact $3$--manifolds with vanishing
contact homology.  In light of the role played by twisted coefficients
in the distinction between strong and weak fillings, this question can
be refined as follows:

\begin{question}
  Does there exist a tight contact $3$--manifold with vanishing
  (twisted or untwisted) contact homology?  In particular, is there a
  weakly fillable contact $3$--manifold with vanishing untwisted
  contact homology?
\end{question}

The generalization of overtwistedness furnished by planar torsion
gives some evidence that the answer to this last question may be no.
In particular, planar torsion as defined in \cite{ChrisOpenBook2}
comes with an integer-valued \emph{order} $k \ge 0$, and for every $k
\ge 1$, our results give examples of contact manifolds with planar
$k$--torsion that are weakly but not strongly fillable.  This
phenomenon is also detected algebraically both by Embedded Contact
Homology \cite{ChrisOpenBook2} and by Symplectic Field Theory
\cite{LatschevWendl}, where in each case the untwisted version
vanishes and the twisted version does not.  Planar $0$--torsion,
however, is fully \emph{equivalent} to overtwistedness, and thus
always causes the twisted theories to vanish.  Thus on the $k=0$
level, there is a conspicuous lack of candidates that could answer the
above question in the affirmative.

Relatedly, the distinction between twisted and untwisted contact
homology makes just as much sense in higher dimensions, yet the
distinction between weak and strong fillings apparently does not.  The
simplest possible definition of a weak filling in higher dimensions,
that $\p W = M$ with $\restricted{\omega}{\xi}$ symplectic, is not
very natural and probably cannot be used to prove anything.  A better
definition takes account of the fact that $\xi$ carries a natural
conformal symplectic structure, and $\omega$ should be required to
define the same conformal symplectic structure on~$\xi$: in this case
we say that $(M,\xi)$ is \textbf{dominated} by $(W,\omega)$.  In
dimension three this notion is equivalent to that of a weak filling,
but surprisingly, in higher dimensions it is equivalent to
\emph{strong} filling, by a result of McDuff
\cite{McDuff_contactType}.  It is thus extremely unclear whether any
sensible distinct notion of weak fillability exists in higher
dimensions, except algebraically:

\begin{question}
  In dimensions five and higher, are there contact manifolds with
  vanishing untwisted but nonvanishing twisted contact homology (or
  similarly, algebraic torsion as in \cite{LatschevWendl})?  If so,
  what does this mean about their symplectic fillings?
\end{question}

Another natural question in higher dimensions concerns the variety of
possible filling obstructions, of which very few are yet known.  There
are obstructions arising from the \emph{plastikstufe}
\cite{NiederkruegerPlastikstufe}, designed as a higher dimensional
analog of the overtwisted disk, as well as from left handed
stabilizations of open books \cite{ContactHomologyLeftHanded}.  Both
of these cause contact homology to vanish, and there is as yet no
known example of a ``higher order'' filling obstruction in higher
dimensions, i.e.~something analogous to Giroux torsion or planar
torsion, which might obstruct symplectic filling without killing
contact homology.  One promising avenue to explore in this area would
be to produce a higher dimensional generalization of the anchored
overtwisted annulus, though once an example is constructed, it may be
far from trivial to show that it has nonvanishing contact homology.

\begin{question}
  Is there any higher dimensional analog of the anchored overtwisted
  annulus, and can it be used to produce examples of nonfillable
  contact manifolds with nonvanishing contact homology?
\end{question}

\subsubsection*{Acknowledgments}

We are grateful to Emmanuel Giroux, Michael Hutchings and Patrick
Massot for enlightening conversations.

During the initial phase of this research, K.~Niederkrüger was working
at the \emph{ENS de Lyon} funded by the project \emph{Symplexe}
06-BLAN-0030-01 of the \emph{Agence Nationale de la Recherche} (ANR).
Currently he is employed at the \emph{Université Paul Sabatier --
  Toulouse~III}.

C.~Wendl is supported by an Alexander von Humboldt Foundation research
fellowship.

\section{Giroux torsion and the overtwisted annulus}
\label{sec:annulus}

In this section, which can be read independently of the remainder of
the paper, we adapt the techniques used in the non-fillability proof
for overtwisted manifolds due to Eliashberg and Gromov to prove
Theorem~\ref{thm:BishopGirouxTorsion}.

We begin by briefly sketching the original proof for overtwisted
contact structures.  Assume $(M,\xi)$ is a closed overtwisted contact
manifold with a weak symplectic filling $(W,\omega)$.  The condition
$\restricted{\omega}{\xi} > 0$ implies that we can choose an almost
complex structure~$J$ on~$W$ which is tamed by~$\omega$ and makes the
boundary $J$--convex.  The elliptic singularity in the center of the
overtwisted disk $\overtwisted \subset M$ is the source of a
$1$--dimensional connected moduli space $\moduli$ of $J$--holomorphic
disks
\begin{equation*}
  u:\, \bigl(\Disk, \p\Disk\bigr) \to \bigl(W,\overtwisted\bigr)
\end{equation*}
that represent homotopically trivial elements in $\pi_2\bigl(W,
\overtwisted\bigr)$, and whose boundaries encircle the singularity of
$\overtwisted$ once.  The space $\moduli$ is diffeomorphic to an open
interval, and as we approach one limit of this interval the
holomorphic curves collapse to the singular point in the center of the
overtwisted disk $\overtwisted$.

We can add to any holomorphic disk in $\moduli$ a capping disk in
$\overtwisted$, such that we obtain a sphere that bounds a ball, and
hence the $\omega$--energy of any disk in $\moduli$ is equal to the
symplectic area of the capping disk.  This implies that the energy of
any holomorphic disk in $\moduli$ is bounded by the integral of
$\abs{\omega}$ over $\overtwisted$, so that we can apply Gromov
compactness to understand the limit at the other end of $\moduli$.  By
a careful study, bubbling and other phenomena can be excluded, and the
result is a limit curve that must have a boundary point tangent to the
characteristic foliation at $\p\overtwisted$; but this implies that it
touches $\p W$ tangentially, which is impossible due to
$J$--convexity.

Below we will work out an analogous proof for the situation where $(M,
\xi)$ is a closed $3$--dimensional contact manifold that contains a
different object, called an anchored overtwisted annulus.  Assuming
$(M, \xi)$ has a weak symplectic filling or is a weakly contact
hypersurface in a closed symplectic $4$--manifold, we will choose an
adapted almost complex structure and instead of using holomorphic
disks, consider holomorphic annuli with boundaries varying along a
$1$--dimensional family of surfaces.  The extra degree of freedom in
the boundary condition produces a moduli space of positive dimension.
If $\omega$ is also exact on the region foliated by the family of
boundary conditions, then we obtain an energy bound, allowing us to
apply Gromov compactness and derive a contradiction.

\subsection{The overtwisted annulus}

We begin by introducing a geometric object that will play the role of
an overtwisted disk.  Recall that for any oriented surface $S
\hookrightarrow M$ embedded in a contact $3$--manifold $(M,\xi)$, the
intersection $TS \cap \xi$ defines an oriented singular
foliation~$S_\xi$ on~$S$, called the \emph{characteristic foliation}.
Its leaves are oriented $1$--dimensional submanifolds, and every point
where $\xi$ is tangent to~$S$ yields a singularity, which can be given
a sign by comparing the orientations of $\xi$ and $TS$.

\begin{defn}
  Let $(M,\xi)$ be a $3$--dimensional contact manifold.  A submanifold
  $\annulus\cong [0,1] \times \SS^1 \hookrightarrow M$ is called a
  \textbf{half-twisted annulus} if the characteristic foliation
  $\annulus_\xi$ has the following properties:
  \begin{enumerate}
  \item $\annulus_\xi$ is singular along~$\{0\} \times \SS^1$ and
    regular on $(0,1] \times \SS^1$.
  \item $\{1\} \times \SS^1$ is a closed leaf.
  \item $(0,1) \times \SS^1$ is foliated by an $\SS^1$--invariant
    family of characteristic leaves that each meet $\{0\} \times
    \SS^1$ transversely and approach $\p\annulus$ asymptotically.
  \end{enumerate}
  We will refer to the two boundary components $\p_L\annulus := \{1\}
  \times \SS^1$ and $\p_S\annulus := \{0\} \times \SS^1$ as the
  \textbf{Legendrian} and \textbf{singular} boundaries respectively.
  An \textbf{overtwisted annulus} is then a smoothly embedded annulus
  $\annulus \subset M$ which is the union of two half-twisted annuli
  \begin{equation*}
    \annulus = \annulus^- \cup \annulus^+
  \end{equation*}
  along their singular boundaries (see Figure~\ref{fig: characteristic
    foliation annuli}).
\end{defn}

\begin{wrapfigure}{r}{0.35\textwidth}
  \vspace{-10pt}
  \begin{center}
    \includegraphics[width=0.3\textwidth,
    keepaspectratio]{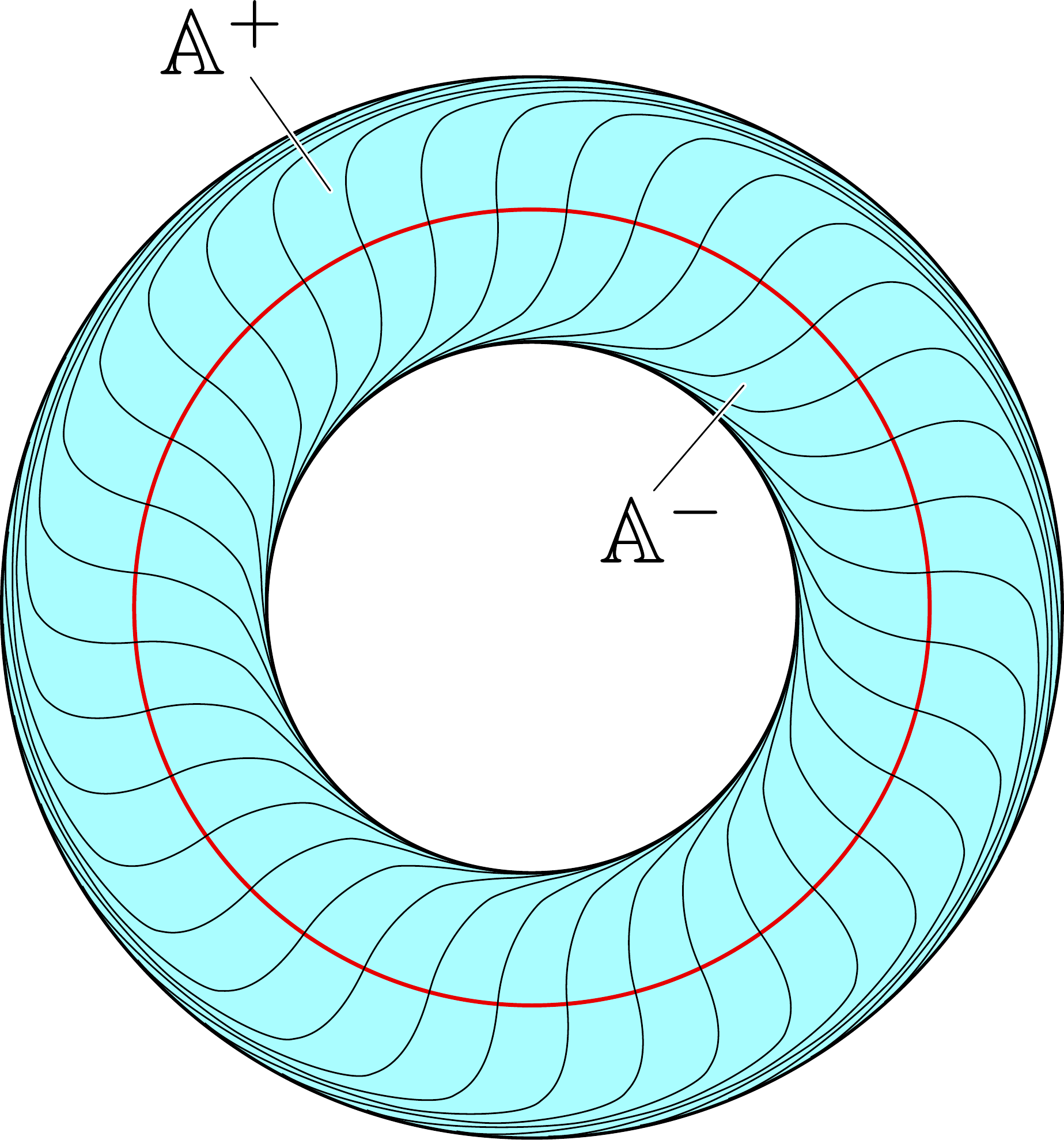}
  \end{center}
  \vspace{-10pt}
  \caption{An overtwisted annulus $\annulus = \annulus^- \cup
    \annulus^+$ with its singular characteristic
    foliation.}  \label{fig: characteristic foliation annuli}
  \vspace{-10pt}
\end{wrapfigure}

\begin{remark}
  As pointed out to us by Giroux, every neighborhood of a point in a
  contact manifold contains an overtwisted annulus.  Indeed, any knot
  admits a $C^0$--small perturbation to a Legendrian knot, which then
  has a neighborhood contactomorphic to the solid torus $\SS^1 \times
  \Disk \ni (\theta ; x,y)$ with contact structure $\ker\left( dy -
    x\,d\theta \right)$.  A small torus $\T^2 \cong \SS^1 \times
  \bigl\{(x,y)\bigm|\, x^2 + y^2 = \epsilon \bigr\}$ is composed of
  two annuli glued to each other along their boundaries, and the
  characteristic foliation on each of these is linear on the interior
  but singular at the boundary.  By pushing one of these annuli
  slightly inward along one boundary component and the other slightly
  outward along the corresponding boundary component, we obtain an
  overtwisted annulus.
\end{remark}

The above remark demonstrates that a single overtwisted annulus can
never give any contact topological information.  We will show however
that the following much more restrictive notion carries highly
nontrivial consequences.

\begin{defn}
  We will say that an overtwisted annulus $\annulus = \annulus^- \cup
  \annulus^+ \subset (M,\xi)$ is \textbf{anchored} if $(M,\xi)$
  contains a smooth $\SS^1$--parametrized family of half-twisted
  annuli $\bigl\{ \annulus^-_\theta \bigr\}_{\theta\in\SS^1}$ which
  are disjoint from each other and from $\annulus^+$, such that
  $\annulus^-_0 = \annulus^-$.  The region foliated by $\bigl\{
  \annulus^-_\theta \bigr\}_{\theta\in\SS^1}$ is then called the
  \textbf{anchor}.
\end{defn}

\begin{example}\label{ex: overtwisted annuli in torsion}
  Recall that we defined a Giroux torsion domain $T_n$ as the
  thickened torus $\T^2\times [0,n] = \bigl\{(\phi,\theta; z)\bigr\}$
  with contact structure given as the kernel of
  \begin{equation*}
    \sin (2\pi z)\, d\phi + \cos (2\pi z)\,d\theta  \;.
  \end{equation*}
  For every $\theta \in S^1$, such a torsion domain contains an
  overtwisted annulus $\annulus_\theta$ which we obtain by bending the
  image of
  \begin{equation*}
    [0,1] \times \SS^1  \hookrightarrow T_n,\,
    \bigl(z,\phi \bigr) \mapsto \bigl(\phi,
    \theta; z \bigr)
  \end{equation*}
  slightly downward along the edges $\{0,1\} \times \SS^1$ so that
  they become regular leaves of the foliation.  This can be done in
  such a way that $\T^2 \times [0,1]$ is foliated by an
  $\SS^1$--family of overtwisted annuli,
  \begin{equation*}
    \T^2 \times [0,1] = \bigcup_{\theta \in \SS^1} \annulus_\theta \;,
  \end{equation*}
  all of which are therefore anchored.
  \begin{figure}
    \begin{center}
      \includegraphics[width=0.3\textwidth,
      keepaspectratio]{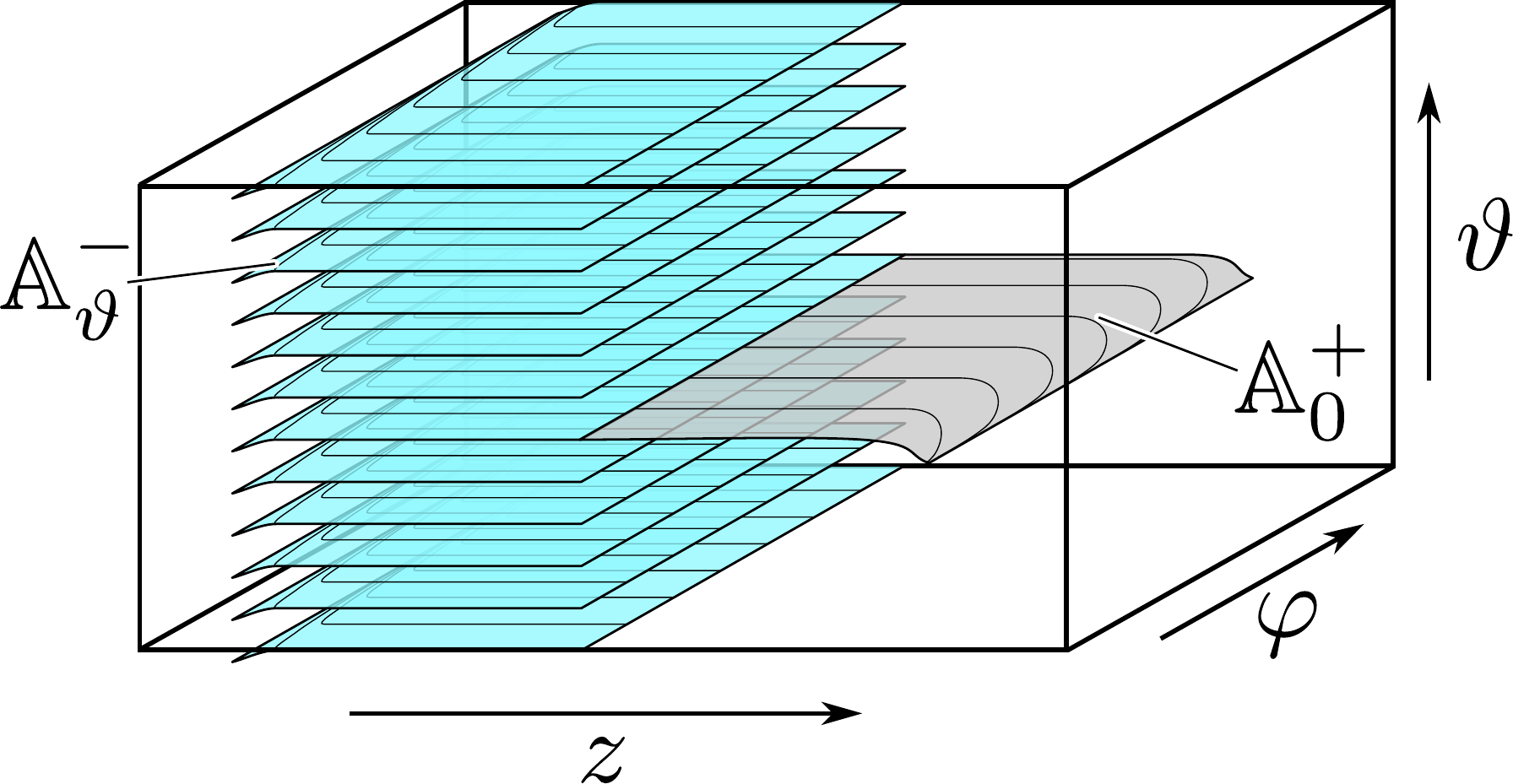}
    \end{center}
    \caption{An anchored overtwisted annulus $\annulus = \annulus_0^-
      \cup \annulus_0^+$ in a Giroux torsion domain $T_1$.}
    \label{fig: annuli in torsion domain}
  \end{figure}
\end{example}

The example shows that every contact manifold with positive Giroux
torsion contains an anchored overtwisted annulus, but in fact, as John
Etnyre and Patrick Massot have pointed out to us, the converse is also
true: it follows from deep results concerning the classification of
tight contact structures on thickened tori \cite{GirouxBifurcations}
that a contact manifold \emph{must} have positive Giroux torsion if it
contains an anchored overtwisted annulus.

We will use an anchored overtwisted annulus as a boundary condition
for holomorphic annuli.  By studying the moduli space of such
holomorphic curves, we find certain topological conditions that have
to be satisfied by a weak symplectic filling, and which will imply
Theorem~\ref{thm:BishopGirouxTorsion}.

\subsection{The Bishop family of holomorphic annuli}

In the non-fillability proof for overtwisted manifolds, the source of
the Bishop family is an elliptic singularity at the center of the
overtwisted disk.  For an anchored overtwisted annulus, holomorphic
curves will similarly emerge out of singularities of the
characteristic foliation, in this case the singular boundaries of the
half-twisted annuli in the anchor, which all together trace out a
pre-Lagrangian torus.  We shall first define a boundary value problem
for pseudoholomorphic annuli with boundary in an anchored overtwisted
annulus, and then choose a special almost complex structure near the
singularities for which solutions to this problem can be constructed
explicitly.  If $\omega$ is exact on the anchor, then the resulting
energy bound and compactness theorem for the moduli space will lead to
a contradiction.

For the remainder of \secref{sec:annulus}, suppose $(W,\omega)$ is a
weak filling of $(M,\xi)$, and the latter contains an anchored
overtwisted annulus $\annulus = \annulus^- \cup \annulus^+$ with
anchor $\{\annulus^-_\theta \}_{\theta \in \SS^1}$ such that
$\annulus^-_0 = \annulus^-$.  The argument will require only minor
modifications for the case where $(W,\omega)$ is closed and contains
$(M,\xi)$ as a weakly contact hypersurface; see
Remark~\ref{remark:hypersurface}.

\subsubsection{A boundary value problem for anchored overtwisted
  annuli}
\label{moduliSpace}

We will say that an almost complex structure $J$ on $W$ is
\textbf{adapted to the filling} if it is tamed by $\omega$ and
preserves~$\xi$.  The fact that $\xi$ is a positive contact structure
implies that any~$J$ adapted to the filling makes the boundary $\p W$
pseudoconvex, with the following standard consequences:

\begin{lemma}[cf.~\cite{Zehmisch_Diplomarbeit}, Theorem~4.2.3]
  \label{lemma:Jconvex}
  If $J$ is adapted to the filling $(W,\omega)$ of $(M,\xi)$, then:
  \begin{enumerate}
  \item Any embedded surface $S \subset M = \p W$ on which the
    characteristic foliation is regular is a totally real submanifold
    of $(W,J)$.
  \item Any connected $J$--holomorphic curve whose interior
    intersects~$\p W$ must be constant.
  \item If $S \subset \p W$ is a totally real surface as described
    above and $u :\, \Sigma \to W$ is a $J$--holomorphic curve
    satisfying the boundary condition $u(\p\Sigma) \subset S$, then
    $u|_{\p\Sigma}$ is immersed and positively transverse to the
    characteristic foliation on~$S$.
  \end{enumerate}
\end{lemma}

Given any adapted almost complex structure~$J$ on $(W,\omega)$, the
above lemma implies that the interiors $\interior{\annulus^+} \subset
\annulus^+$ and $\interior{\annulus^-_\theta} \subset
\annulus^-_\theta$ are all totally real submanifolds of $(W,J)$.  We
shall then consider a moduli space of $J$--holomorphic annuli defined
as follows.  Denote by $A_r$ the complex annulus
\begin{equation*}
  A_r = \bigl\{z\in\C\bigm|\, 1\le \abs{z} \le 1+r\bigr\}\subset \C
\end{equation*}
of modulus $r > 0$, and write its boundary components as $\p_r^- :=
\bigl\{z\in\C\bigm|\, \abs{z} = 1 \bigr\}$ and $\p_r^+ :=
\bigl\{z\in\C\bigm|\, \abs{z} = 1+r\bigr\}$.  We then define the space
\begin{equation*}
  \begin{split}
    \moduli(J) = \bigcup_{r > 0} \bigl\{ u : A_r \to W \ \bigm|\ & Tu
    \circ i = J \circ Tu,
    \text{ $u(\p_r^+) \subset \interior{\annulus^+}$,} \\
    &\text{$u(\p_r^-) \subset \interior{\annulus^-_\theta}$ for any
      $\theta \in \SS^1$} \bigl\} \Big/ \SS^1,
  \end{split}
\end{equation*}
where $\tau \in \SS^1$ acts on maps $u :\, A_r \to W$ by $\tau \cdot
u(z) := u(e^{2\pi i\tau} z)$.  This space can be given a natural
topology by fixing a smooth family of diffeomorphisms from a standard
annulus to the domains~$A_r$,
\begin{equation}
  \label{eqn:psir}
  \psi_r :\, [0,1] \times \SS^1 \to A_r : (s,t) \mapsto e^{s\log(1+r) + 2\pi it} \;,
\end{equation}
and then saying that a sequence $u_k : A_{r_k} \to W$ converges to
$u:\, A_r \to W$ in $\moduli(J)$ if $r_k \to r$ and
\begin{equation*}
  u_k \circ \psi_{r_k}(s,t + \tau_k) \to u \circ \psi_r(s,t)
\end{equation*}
for some sequence $\tau_k \in \SS^1$, with $C^\infty$--convergence on
$[0,1] \times \SS^1$.

We will show below that $J$ can be chosen to make $\moduli(J)$ a
nonempty smooth manifold of dimension one.  This explains why the
``anchoring'' condition is necessary: it introduces an extra degree of
freedom in the boundary condition, without which the moduli space
would generically be zero-dimensional and the Bishop family could
never expand to reach the edge of the half-twisted annuli.

\subsubsection{Special almost complex structures near the boundary}

Suppose $\alpha$ is a contact form for $(M,\xi)$.  The standard way to
construct compatible almost complex structures on the symplectization
$\bigl(\R\times M, d(e^t \alpha)\bigr)$ involves choosing a compatible
complex structure $J_\xi$ on the symplectic vector bundle
$\bigl(\restricted{\xi}{\{0\}\times M}, d\alpha\bigr)$, extending it
to a complex structure on $\bigl( \restricted{T(\R\times
  M)}{\{0\}\times M}, d(e^t\alpha)\bigr)$ such that
\begin{equation*}
  J X_\alpha = - \p_t \text{ and } J \p_t = X_\alpha
\end{equation*}
for the Reeb vector field $X_\alpha$ of $\alpha$, and finally defining
$J$ as the unique $\R$--invariant almost complex structure on
$\R\times M$ that has this form at $\{0\} \times M$.  Almost complex
structures of this type will be essential for the arguments of
\secref{sec:punctured}.  For the remainder of this section, we will
drop the $\R$--invariance condition but say that an almost complex
structure on $\R\times M$ is \textbf{compatible with~$\alpha$} if it
takes the above form on $\{0\} \times M$; in this case it is tamed by
$d(e^t\alpha)$ on any sufficiently small neighborhood of $\{0\} \times
M$.  It is sometimes useful to know that an adapted~$J$ on any weak
filling can be chosen to match any given~$J$ of this form near the
boundary.

\begin{proposition}\label{emdedding symplectization into weak collar}
  Let $(M,\xi)$ be a contact $3$--manifold with weak filling
  $(W,\omega)$.  Choose any contact form $\alpha$ for $\xi$ and an
  almost complex structure $J$ on $\R\times M$ compatible with
  $\alpha$.  Then for sufficiently small $\epsilon > 0$, the canonical
  identification of $\{0\} \times M$ with $\p W$ can be extended to a
  diffeomorphism from $(-\epsilon,0] \times M$ to a collar
  neighborhood of $\p W$ such that the push-forward of~$J$ is tamed
  by~$\omega$.

  In particular, this almost complex structure can then be extended to
  a global almost complex structure on $W$ that is tamed by $\omega$,
  and is thus adapted to the filling.
\end{proposition}
\begin{proof}
  Writing $J_\xi := \restricted{J}{\xi}$, construct an auxiliary
  complex structure $J_{\mathrm{aux}}$ on $\restricted{TW}{M}$ as the
  direct sum of $J_\xi$ on the symplectic bundle
  $\bigl(\restricted{\xi}{\{0\}\times M}, \omega\bigr)$ with a
  compatible complex structure on its $\omega$--symplectic complement
  $\bigl(\restricted{\xi^{\perp \omega}}{\{0\}\times M},
  \omega\bigr)$.  Clearly this complex structure is tamed by
  $\restricted{\omega}{M}$.

  Define an outward pointing vector field along the boundary by
  setting
  \begin{equation*}
    Y =  - J_{\mathrm{aux}}\cdot X_\alpha \;.
  \end{equation*}
  Extend $Y$ to a smooth vector field on a small neighborhood of $M$
  in $W$, and use its flow to define an embedding of a subset of the
  symplectization
  \begin{equation*}
    \Psi:\, (-\epsilon, 0] \times M \to W,\, \bigl(t, p\bigr)
    \mapsto \Phi_Y^t(p)
  \end{equation*}
  for sufficiently small $\epsilon > 0$.  The restriction of $\Psi$ to
  $\{0\}\times M$ is the identity on $M$, and the push-forward of $J$
  under this map coincides with $J_{\mathrm{aux}}$ along $M$, because
  $\Psi_* \p_t = Y$.  It follows that the push-forward of~$J$ is tamed
  by $\omega$ on a sufficiently small neighborhood of $M= \p W$, and
  we can then extend it to~$W$ as an almost complex structure tamed
  by~$\omega$.
\end{proof}

\subsubsection{Generation of the Bishop family}
\label{sec:Bishop family}

We shall now choose an almost complex structure $J_0$ on the
symplectization of $M$ that allows us to write down the germ of a
Bishop family in $\R\times M$ which generates a component
of~$\moduli(J_0)$.  At the same time, $J_0$ will prevent other
holomorphic curves in the same component of~$\moduli(J_0)$ from
approaching the singular boundaries of the half-twisted annuli
$\annulus^-_\theta$.  We can then apply Proposition~\ref{emdedding
  symplectization into weak collar} to identify a neighborhood of
$\{0\}\times M$ in the symplectization with a boundary collar of $W$,
so that $W$ contains the Bishop family.

The singular boundaries of $\annulus^-_\theta$ define closed leaves of
the characteristic foliation on a torus
\begin{equation*}
  T := \bigcup_{\theta \in \SS^1} \p_S\annulus^-_\theta \subset M \;,
\end{equation*}
which is therefore a pre-Lagrangian torus.  We then obtain the
following by a standard Moser-type argument.
\begin{lemma}\label{lemma:preLagCoords}
  For sufficiently small $\epsilon > 0$, a tubular neighborhood
  $\nbhd(T) \subset M$ of~$T$ can be identified with $\T^2 \times
  (-\epsilon,\epsilon)$ with coordinates $(\phi,\theta ; r)$ such
  that:
  \begin{itemize}
  \item $T = \T^2 \times \{0\}$,
  \item $\xi = \ker \left[ \cos(2\pi r)\, d\theta + \sin(2\pi r)\,
      d\phi \right]$,
  \item $\annulus \cap \nbhd(T) = \{ \theta = 0 \}$, and
    $\annulus^-_{\theta_0} \cap \nbhd(T) = \{ \theta=\theta_0,\ r \in
    (-\epsilon,0] \}$ for all $\theta_0 \in \SS^1$.
  \end{itemize}
\end{lemma}

Using the coordinates given by the lemma, we can reflect the
half-twisted annuli $\annulus^-_{\theta_0}$ across~$T$ within this
neighborhood to define the surfaces
\begin{equation*}
  \annulus^+_{\theta_0} := \bigl\{ \theta = \theta_0,\ 
  r \in [0,\epsilon) \bigr\} \subset M \;.
\end{equation*}
Each of these surfaces looks like a collar neighborhood of the
singular boundary in a half-twisted annulus.  Now choose for $\xi$ a
contact form $\alpha$ on~$M$ that restricts on $\nbhd(T)$ to
\begin{equation}
  \label{eqn:standardT3}
  \restricted{\alpha}{\nbhd(T)} = \cos (2\pi r)\,d\theta +
  \sin (2\pi r)\, d\phi \;.
\end{equation}
The main idea of the construction is to identify the set $\nbhd(T)$
with an open subset of the unit cotangent bundle $\T^3 =
\SS\bigl(T^*\T^2\bigr)$ of $\T^2$, with its canonical contact
form~$\acan$.  We will then use an integrable complex structure on
$T^*\T^2$ to find explicit families of holomorphic curves that give
rise to holomorphic annuli in $\R\times M$.

The cotangent bundle of $\T^2 = \R^2 / \Z^2$ can be identified
naturally with
\begin{equation*}
  \C^2 / i\Z^2 = \R^2 \oplus i(\R^2 / \Z^2)
\end{equation*}
such that the canonical $1$--form takes the form $\lcan = p_1\,dq_1 +
p_2\,dq_2$ in coordinates $[z_1,z_2] = \bigl[p_1 + iq_1, p_2 +i
q_2\bigr]$.  The unit cotangent bundle $\SS\bigl(T^*\T^2\bigr) =
\bigl\{[p_1 + iq_1, p_2 + iq_2] \in T^*\T^2\bigm| \, \abs{p_1}^2 +
\abs{p_2}^2 = 1 \bigr\}$ can then be parametrized by the map
\begin{equation*}
  \T^3 = \T^2 \times \SS^1 \ni (\phi,\theta;r) \mapsto
  \bigl[\sin 2\pi r + i\phi, \cos 2\pi r + i \theta\bigr] \in T^*\T^2 \;,
\end{equation*}
and the pull-back of $\lcan$ to $\T^3$ gives
\begin{equation*}
  \acan := \restricted{\lcan}{T\SS(T^*\T^2)} = \cos (2\pi r)\,d\theta
  + \sin (2\pi r)\, d\phi \;.
\end{equation*}
The Liouville vector field dual to $\lcan$ is $p_1\, \p_{p_1} + p_2\,
\p_{p_2}$, and we can use its flow to identify $T^*\T^2 \setminus
\T^2$ with the symplectization of $\SS\bigl(T^*\T^2\bigr)$:
\begin{equation*}
  \Phi :\, (\R\times \SS\bigl(T^*\T^2\bigr), d(e^t \acan))
  \to (T^*\T^2 \setminus \T^2, d\lcan),\, (t;p + iq) \mapsto e^t p + iq \;.
\end{equation*}
Then it is easy to check that the restriction of the complex structure
$\Phi^*i$ to $\{0\} \times \T^3$ preserves $\ker\acan$ and maps $\p_t$
to the Reeb vector field of~$\acan$, hence $\Phi^*i$ is compatible
with~$\acan$.  Now for the neighborhood $\nbhd(T) \cong \T^2 \times
(-\epsilon,\epsilon)$, denote by
\begin{equation*}
  \Psi :\, (-\epsilon,0] \times \nbhd(T) \hookrightarrow \R\times \T^3
\end{equation*}
the natural embedding determined by the coordinates $(\phi,\theta;r)$.
Proposition~\ref{emdedding symplectization into weak collar} then
implies:

\begin{lemma}\label{lemma:J0}
  There exists an almost complex structure $J_0$ adapted to the
  filling $(W,\omega)$ of $(M,\xi)$, and a collar neighborhood
  $\nbhd(\p W) \cong (-\epsilon,0] \times M$ of $\p W$ such that on
  $(-\epsilon,0] \times \nbhd(T) \subset W$, $J_0 = \Psi^*\Phi^*i$.
\end{lemma}

Consider the family of complex lines $L_\zeta :=
\bigl\{(z_1,z_2)\bigm|\, z_2 = \zeta\bigr\}$ in $\C^2$.  The
projection of these curves into $T^*\T^2 \cong \C^2 /i\Z^2$ are
holomorphic cylinders, whose intersections with the unit disk bundle
$\Disk(T^*\T^2) = \bigl\{ p + iq \in \C^2 / i\Z^2 \ \bigm|\ \abs{p}^2
\le 1 \bigr\}$ define holomorphic annuli.  In particular, for
sufficiently small $\delta > 0$ and any
\begin{equation*}
  (c,\tau) \in (0,\delta] \times \SS^1 \;,
\end{equation*}
the intersection $L_{(1-c) + i\tau} \cap \Disk(T^*\T^2)$ is a
holomorphic annulus in $\Phi \circ \Psi\bigl((-\epsilon,0] \times
\nbhd(T)\bigr)$, which therefore can be identified with a
$J_0$--holomorphic annulus
\begin{equation*}
  u_{(c,\tau)} :\, A_{r_c} \to W
\end{equation*}
with image in the neighborhood $(-\epsilon,0] \times \nbhd(T)$, where
the modulus $r_c > 0$ depends on~$c$ and approaches zero as $c \to 0$.
It is easy to check that the two boundary components of $u_{(c,\tau)}$
map into the interiors of the surfaces $\annulus^+_\tau$ and
$\annulus^-_\tau$ respectively in~$\p W$.  Observe that all of these
annuli are obviously embedded, and they foliate a neighborhood of~$T$
in~$W$.  We summarize the construction as follows.

\begin{proposition}\label{prop:BishopFamily}
  For the almost complex structure $J_0$ given by
  Lemma~\ref{lemma:J0}, there exists a smooth family of properly
  embedded $J_0$--holomorphic annuli
  \begin{equation*}
    \bigl\{ u_{(c,\tau)} :\, A_{r_c} \to W \bigr\}_{(c,\tau)
      \in (0,\delta] \times \SS^1}
  \end{equation*}
  which foliate a neighborhood of~$T$ in $W \setminus T$ and satisfy
  the boundary conditions
  \begin{equation*}
    u_{(c,\tau)}\left(\p_{r_c}^+\right) \subset \interior{\annulus^+_\tau},
    \qquad
    u_{(c,\tau)}\left(\p_{r_c}^-\right) \subset \interior{\annulus^-_\tau}.
  \end{equation*}
  In particular the curves $u_{(c,0)}$ for $c \in (0,\delta]$ all
  belong to the moduli space $\moduli(J_0)$.
\end{proposition}

\begin{figure}
  \begin{center}
    \includegraphics[width=0.35\textwidth,
    keepaspectratio]{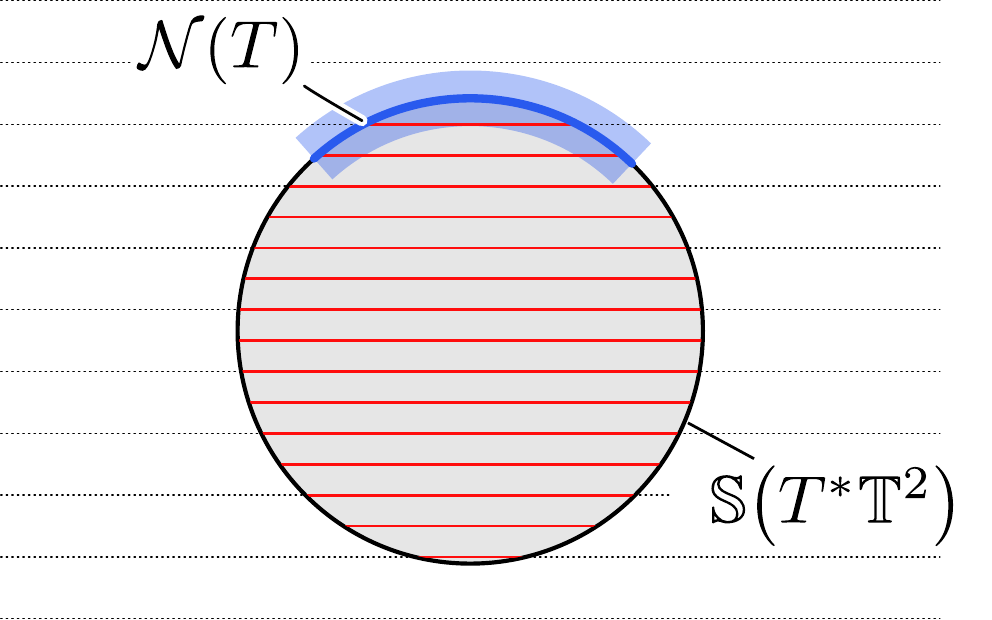}
  \end{center}
  \vspace{-10pt}
  \caption{The unit disk bundle in $T^*\T^2$ is foliated by a family
    of holomorphic annuli obtained from the complex planes $L_\zeta$.
    The neighborhood $\nbhd(T)$ can be identified with a subset of the
    unit disk bundle $\SS\bigl( T^*\T^2\bigr)$.}
\end{figure}

Denote the neighborhood foliated by the curves $u_{(c,\tau)}$ by
\begin{equation*}
  \uU = \bigcup_{(c,\tau) \in (0,\delta]\times\SS^1}  u_{(c,\tau)}(A_{r_c}) \;,
\end{equation*}
and define the following special class of almost complex structures,
\begin{equation*}
  \complex_\uU(\omega,\xi) = \bigl\{ \text{almost complex structures~$J$
    adapted to the filling $(W,\omega)$ such that $J \equiv J_0$
    on~$\overline{\uU}$} \bigr\} \;.
\end{equation*}
The annuli $u_{(c,\tau)}$ are thus $J$--holomorphic for any $J \in
\complex_\uU(\omega,\xi)$, and the space $\moduli(J)$ is therefore
nonempty.  In this case, denote by
\begin{equation*}
  \moduli_0(J) \subset \moduli(J)
\end{equation*}
the connected component of $\moduli(J)$ that contains the curves
$u_{(c,0)}$.

\begin{lemma}\label{lemma:embeddedBoundary}
  Every curve $u : A_r \to W$ in $\moduli_0(J)$ is proper, and its
  restriction to $\p A_r$ is embedded.
\end{lemma}
\begin{proof}
  Properness follows immediately from Lemma~\ref{lemma:Jconvex}, and
  due to our assumptions on the characteristic foliation of a
  half-twisted annulus, embeddedness at the boundary also follows from
  the lemma after observing that the homotopy class of $u|_{\p_r^\pm}$
  is the same as for the curves $u_{(c,0)}$, whose boundaries
  intersect every characteristic leaf once.
\end{proof}

\begin{proposition}\label{prop:intersection}
  For $J \in \complex_\uU(\omega,\xi)$, suppose $u \in \moduli_0(J)$
  is not one of the curves $u_{(c,0)}$.  Then $u$ does not intersect
  the interior of~$\uU$.
\end{proposition}
\begin{proof}
  The proof is based on an intersection argument.  Each of the curves
  $u_{(c,\tau)}$ foliating $\uU$ can be capped off to a cycle
  $\widehat u_{(c,\tau)}$ that represents the trivial homology class
  in $H_2(W)$.  We shall proceed in a similar way to obtain a cycle
  $\widehat{u}$ for $u$, arranged such that intersections between the
  cycles $\widehat{u}$ and $\widehat u_{(c,\tau)}$ can only occur when
  the actual holomorphic curves $u$ and $u_{(c,\tau)}$ intersect.
  Then if $u$ is not any of the curves $u_{(c,0)}$ but intersects the
  interior of~$\uU$, it also is not a multiple cover of any
  $u_{(c,0)}$ due to Lemma~\ref{lemma:embeddedBoundary}, and therefore
  must have an isolated positive intersection with some curve
  $u_{(c,\tau)}$.  It follows that $[\widehat u_{c_0}] \bullet
  [\widehat{u}] > 0$, but since $[\widehat u_{c_0}] = 0 \in H_2(W)$,
  this is a contradiction.

  We construct the desired caps as follows.  Suppose $u(\p^-_r)
  \subset \annulus^-_{\theta_0}$.  We may assume without loss of
  generality that $u$ and $u_{(c,\tau)}$ intersect each other in the
  \emph{interior}, and since this intersection will not disappear
  under small perturbations, we can adjust~$\tau$ so that it equals
  neither~$0$ nor~$\theta_0$.  A cap for $u_{(c,\tau)}$ can then be
  constructed by filling in the space in $\annulus^-_\tau \cup
  \annulus^+_\tau$ between the two boundary components of
  $u_{(c,\tau)}$; clearly the resulting homology class $[\widehat
  u_{(c,\tau)}$] is trivial.

  The cap for $u$ will be a piecewise smooth surface in $\p W$
  constructed out of three smooth pieces:
  \begin{itemize}
  \item A subset of $\annulus^+$ filling the space between the
    singular boundary $\p_S\annulus^+$ and $u(\p^+_r)$,
  \item A subset of $\annulus^-_{\theta_0}$ filling the space between
    the singular boundary $\p_S\annulus^-_{\theta_0}$ and $u(\p^-_r)$,
  \item An annulus in~$T = \{ r = 0\}$ defined by letting $\theta$
    vary over a path in $\SS^1$ that connects~$0$ to~$\theta_0$ by
    moving in a direction such that it does not hit~$\tau$.
  \end{itemize}
  By construction, the two caps are disjoint, and since both are
  contained in $\p W$, neither intersects the interior of either
  curve.
\end{proof}

\subsubsection{Local structure of the moduli space}
\label{subsubsec:index}

We now show that $\moduli_0(J)$ can be given a nice local structure
for generic data.

\begin{proposition}
  For generic $J \in \complex_\uU(\omega,\xi)$, the moduli space
  $\moduli_0(J)$ is a smooth $1$--dimensional manifold.
\end{proposition}
\begin{proof}
  Since $\moduli_0(J)$ is connected by assumption, the dimension can
  be derived by computing the Fredholm index of the associated
  linearized Cauchy-Riemann operator for any of the curves $u_{(c,0)}
  \in \moduli_0(J)$.  By Lemma~\ref{lemma:embeddedBoundary}, every
  curve $u \in \moduli_0(J)$ is somewhere injective, thus standard
  arguments as in \cite{McDuffSalamonJHolo} imply that for generic $J
  \in \complex_\uU(\omega,\xi)$, the subset of curves in
  $\moduli_0(J)$ that are not completely contained in $\overline{\uU}$
  is a smooth manifold of the correct dimension.
  Proposition~\ref{prop:intersection} implies that the remaining
  curves all belong to the family $u_{(c,0)}$, and for these we will
  have to examine the Cauchy-Riemann operator more closely since $J$
  cannot be assumed to be generic in~$\overline{\uU}$.

  Abbreviate $u = u_{(c,0)} :\, A_r \to W$ for any $c \in (0,\delta]$.
  Since $u$ is embedded, a neighborhood of~$u$ in $\moduli_0(J)$ can
  be described via the \emph{normal Cauchy-Riemann operator}
  (cf.~\cite{ChrisTransversality}),
  \begin{equation}
    \label{eqn:DuN}
    \mathbf{D}_u^N :\, W^{1,p}_{\ell,\zeta}(N_u) \to
    L^p\bigl(\overline{\Hom}_\C(TA_r,N_u)\bigr)\;,
  \end{equation}
  where $p > 2$, $N_u \to A_r$ is the complex normal bundle of~$u$,
  $\mathbf{D}_u^N$ is the normal part of the restriction of the usual
  linearized Cauchy-Riemann operator $D\dbar_J(u)$ (which acts on
  sections of $u^*TW$) to sections of~$N_u$, and the subscripts~$\ell$
  and~$\zeta$ represent a boundary condition to be described below.
  We must define the normal bundle $N_u$ so that at the boundary its
  intersection with $T\annulus$ has real dimension one, thus defining
  a totally real subbundle
  \begin{equation*}
    \ell = N_u|_{\p A_r} \cap (u|_{\p A_r})^*T\annulus
    \subset N_u|_{\p A_r} \;.
  \end{equation*}
  To be concrete, note that in the coordinates $(\phi,\theta;r)$ on
  $\nbhd(T)$, the image of~$u$ can be parametrized by a map of the
  form
  \begin{equation*}
    v:\, [-r_0,r_0] \times \SS^1 \to (-\epsilon,0] \times \nbhd(T),\,
    (\sigma,\tau) \mapsto (a(\sigma); \tau , 0 ; \sigma)
  \end{equation*}
  for some $r_0 > 0$, where $a(\sigma)$ is a smooth, convex and even
  function.  Choose a vector field along~$v$ of the form
  \begin{equation*}
    \nu(\sigma,\tau) = \nu_1(\sigma) \, \p_r + \nu_2(\sigma) \, \p_t
  \end{equation*}
  which is everywhere transverse to the path $\sigma\mapsto
  (a(\sigma),\sigma)$ in the $tr$--plane, and require
  \begin{equation*}
    \nu(\pm r_0,\tau) = \mp \p_r \;.
  \end{equation*}
  Then the vector fields $\nu$ and $i\nu$ along~$v$ span a complex
  line bundle that is everywhere transverse to~$v$, and its
  intersection with $T\annulus$ at the boundary is spanned by $\p_r$.
  We define this line bundle to be the normal bundle $N_u$ along~$u$,
  which comes with a global trivialization defined by the vector
  field~$\nu$, for which we see immediately that both components of
  the real subbundle $\ell$ along~$\p A_r$ have vanishing Maslov
  index.  To define the proper linearized boundary condition, we still
  must take account of the fact that the image of $\p^-_r$ for nearby
  curves in the moduli space may lie in different half-annuli
  $\annulus^-_\theta$: this means there is a smooth section $\zeta \in
  \Gamma(N_u|_{\p^-_r})$ which is everywhere transverse to~$\ell$,
  such that the domain for $\mathbf{D}_u^N$ takes the form
  \begin{equation*}
    \begin{split}
      W^{1,p}_{\ell,\zeta}(N_u) := \bigl\{ \eta \in W^{1,p}(N_u) \
      \bigm|\
      &\text{$\eta(z)\in \ell_z$ for all $z \in \p^+_r$},\\
      &\text{$\eta(z) + c\, \zeta(z) \in \ell_z$ for all $z \in
        \p^-_r$ and any constant $c \in \R$} \bigr\} \; .
    \end{split}
  \end{equation*}
  Leaving out the section $\zeta$, we obtain the standard totally real
  boundary condition
  \begin{equation*}
    W^{1,p}_\ell(N_u) := \{ \eta \in W^{1,p}(N_u) \ |\ 
    \text{$\eta(z)\in \ell_z$ for all $z \in \p A_r$} \} \;,
  \end{equation*}
  and the Riemann-Roch formula implies that the restriction of
  $\mathbf{D}^N_u$ to this smaller space has Fredholm index~$0$.
  Since the smaller space has codimension one in
  $W^{1,p}_{\ell,\zeta}(N_u)$, the index of~$\mathbf{D}^N_u$ on the
  latter is~$1$, which proves the dimension formula for
  $\moduli_0(J)$.  Moreover, since $N_u$ has complex rank one, there
  are certain \emph{automatic transversality} theorems that apply: in
  particular, Theorem~4.5.36 in \cite{ChrisThesis} implies that
  \eqref{eqn:DuN} is always surjective, and $\moduli_0(J)$ is
  therefore a smooth manifold of the correct dimension, even in the
  region where $J$ is not generic.
\end{proof}

\subsubsection{Energy bounds}
\label{sec:energy_bound}

Assume now that $\omega$ is exact on the anchor, i.e.~there exists a
$1$--form $\beta$ on the region $\bigcup_{\theta \in \SS^1}
\annulus^-_\theta$ with $d\beta = \omega$.  The aim of this section is
to find a uniform bound on the $\omega$--energy
\begin{equation*}
  E_\omega (u) = \int_{A_r} u^* \omega
\end{equation*}
for all curves
\begin{equation*}
  u:\, \bigl(A_r, \partial_r^- \cup \partial_r^+ \bigr)
  \to (W,  \annulus^-_\theta \cup \annulus^+)
\end{equation*}
in the connected moduli space $\moduli_0(J)$ generated by the Bishop
family.

Given such a curve~$u \in \moduli_0(J)$, there exists a smooth
$1$--parameter family of maps
\begin{equation*}
  \{  u_t:\, A_{r} \to W \}_{t\in [\epsilon,1]} \;,
\end{equation*}
such that $u_\epsilon$ is a reparametrization one of the explicitly
constructed curves $u_{(c,0)}$ that foliate~$\uU$, and $u_1 = u$.  The
map $\bar{u} : [\epsilon,1] \times A_r \to W : (t,z) \mapsto u_t(z)$
then represents a $3$--chain, and applying Stokes' theorem to the
integral of $d(\bar{u}^*\omega)=0$ over $[\epsilon,1] \times A_r$
gives
\begin{equation*}
  E_\omega(u) =  E_\omega(u_\epsilon)
  - \int_{[\epsilon,1]\times \p A_{r}} \bar{u}^*\omega \;.
\end{equation*}
The image $\bar{u}\bigl([\epsilon,1]\times \p A_{r}\bigr)$ has two
components $\bar{u}\bigl([\epsilon,1]\times \p_{r}^+\bigr)$ and
$\bar{u}\bigl([\epsilon,1]\times \p_{r}^-\bigr)$.  The first lies in a
single half-twisted annulus $\annulus^+$, and thus the absolute value
of $\int_{[\epsilon,1]\times \p_{r}^+} \bar{u}^*\omega$ can be bounded
by $\int_{\annulus^+} \abs{\omega}$.  For the second component, the
image $\bar{u}\bigl([\epsilon,1] \times \p_{r}^-\bigr)$ lies in the
anchor $\bigcup_{\theta \in \SS^1} \annulus^-_\theta$, so we can write
\begin{equation*}
  E_\omega(u) \le  E_\omega(u_\epsilon) + \int_{\annulus^+} \abs{\omega}
  + \int_{\p_\epsilon^-} u_\epsilon^*\beta - \int_{\p_r^-} u^*\beta \;.
\end{equation*}

\begin{figure}[htbp]
  \centering
  \includegraphics[width=0.3\textwidth,
  keepaspectratio]{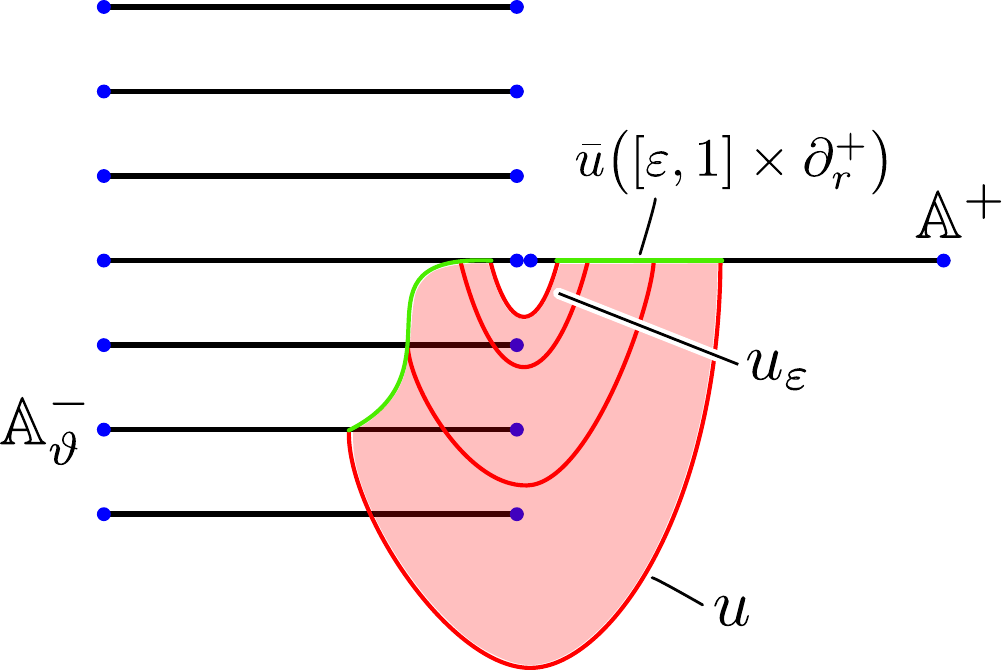}
  \caption{The holomorphic annulus $u:\, \bigl(A_r, \partial_r^-
    \cup \partial_r^+ \bigr) \to (W, \annulus_\theta^- \cup
    \annulus^+)$ is part of a $1$--parameter family $u_t$ of curves
    that start at an annulus $u_\epsilon$ that lies in the Bishop
    family.}\label{fig: energy argument}
\end{figure}

It remains only to find a uniform bound on the last term in this sum,
$\int_{\p_r^-} u^*\beta$.  Observe that $u(\partial_r^-)$ and the
singular boundary $\p_S\annulus^-_\theta$ enclose an annulus within
$\annulus^-_\theta$, thus
\begin{equation*}
  \abs{\int_{\partial_r^+} u^*\beta} \le   \int_{\p_S \annulus^-_\theta}
  \abs{\beta} +  \int_{\annulus^-_\theta}\abs{\omega} \;.
\end{equation*}
This last sum is uniformly bounded since the surfaces
$\annulus^-_\theta$ for $\theta \in \SS^1$ form a compact family.

\subsubsection{Gromov compactness for the holomorphic annuli}

The main technical ingredient still needed for the proof of
Theorem~\ref{thm:BishopGirouxTorsion} is the following application of
Gromov compactness.

\begin{proposition}\label{prop:compactness}
  Suppose $J$ is generic in $\complex_\uU(\omega,\xi)$, $\omega$ is
  exact on the anchor, and
  \begin{equation*}
    u_k:\, \bigl(A_{r_k}, \partial_{r_k}^- \cup \partial_{r_k}^+ \bigr) \to
    (W, \annulus^-_{\theta_k} \cup \annulus^+)
  \end{equation*}
  is a sequence of curves in $\moduli_0(J)$ with images not contained
  in~$\uU$.  Then there exist $r > 0$, $\theta \in \SS^1$ and a
  sequence $\tau_k \in \SS^1$ such that after passing to a
  subsequence, $r_k \to r$, $\theta_k \to \theta$ and the maps
  \begin{equation*}
    z \mapsto u_k(e^{2\pi i\tau_k} z)
  \end{equation*}
  are $C^\infty$--convergent to a $J$--holomorphic annulus $u:\, A_r
  \to W$ satisfying $u(\p^-_r) \subset \annulus^-_\theta$ and
  $u(\p^+_r) \subset \annulus^+$.
\end{proposition}

The energies $\int_{A_{r_k}} u_k^*\omega$ are uniformly bounded due to
the exactness assumption, and the proof is then essentially the same
as in the disk case, cf.~\cite{Eliashberg_HoloDiscs} or
\cite{Zehmisch_Diplomarbeit}.  A priori, $u_k$ could converge to a
nodal holomorphic annulus, with nodes on both the boundary and the
interior.  Boundary nodes are impossible however for topological
reasons, as each boundary component of $u_k$ must pass exactly once
through each leaf in an $\SS^1$--family of characteristic leaves, and
any boundary component in a nodal annulus will also pass \emph{at
  least} once through each of these leaves.  Having excluded boundary
nodes, $u_k$ could converge to a bubble tree consisting of holomorphic
spheres and either an annulus or a pair of disks, all connected to
each other by interior nodes.  This however is a codimension~$2$
phenomenon, and thus cannot happen for generic~$J$ since
$\moduli_0(J)$ is $1$--dimensional.  Here we make use of two important
facts:
\begin{enumerate}
\item Any component of the limit that has nonempty boundary must be
  somewhere injective, as it will be embedded at the boundary by the
  same argument as in Lemma~\ref{lemma:embeddedBoundary}.  Such
  components therefore have nonnegative index.
\item $(W,\omega)$ is semipositive (as is always the case in
  dimension~$4$), hence holomorphic spheres of negative index cannot
  bubble off.
\end{enumerate}
With this, the proof of Proposition~\ref{prop:compactness} is
complete.

\subsubsection{Proof of Theorem~\ref{thm:BishopGirouxTorsion}}

Assume $(W,\omega)$ is a weak filling of $(M,\xi)$ and the latter has
positive Giroux torsion.  As shown in Example~\ref{ex: overtwisted
  annuli in torsion}, $(M,\xi)$ contains an anchored overtwisted
annulus.  For this setting, we defined in \secref{moduliSpace} a
moduli space of $J$--holomorphic annuli $\moduli(J)$ with a
$1$--parameter family of totally real boundary conditions.  In
\secref{sec:Bishop family}, we found a special almost complex
structure $J_0$ which admits a Bishop family of holomorphic annuli,
and thus generates a nonempty connected component $\moduli_0(J_0)
\subset \moduli(J_0)$.  This space remains nonempty after perturbing
$J_0$ generically outside the region foliated by the Bishop family,
thus producing a new almost complex structure~$J$ and nonempty moduli
space $\moduli_0(J)$.  We then showed in \secref{subsubsec:index} that
$\moduli_0(J)$ is a smooth $1$--dimensional manifold, which is
therefore diffeomorphic to an open interval, one end of which
corresponds to the collapse of the Bishop annuli into the singular
circle at the center of the overtwisted annulus.  In particular, this
implies that $\moduli_0(J)$ is not compact, and the key is then to
understand its behavior at the other end.  The assumption that
$\omega$ is exact on the anchor provides a uniform energy bound, with
the consequence that if all curves in~$u$ remain a uniform positive
distance away from the Legendrian boundaries of $\annulus^+$ and
$\annulus^-_\theta$, Proposition~\ref{prop:compactness} implies
$\moduli_0(J)$ is compact.  But since the latter is already known to
be false, this implies that $\moduli_0(J)$ contains a sequence of
curves drawing closer to the Legendrian boundary, and applying
Proposition~\ref{prop:compactness} again, a subsequence converges to a
$J$--holomorphic annulus that touches the Legendrian boundary of
$\annulus^+$ or $\annulus^-_\theta$ tangentially.  That is impossible
by Lemma~\ref{lemma:Jconvex}, and we have a contradiction.  Together
with the following remark, this completes the proof of
Theorem~\ref{thm:BishopGirouxTorsion}.

\begin{remark}\label{remark:hypersurface}
  If $(M,\xi) \subset (W,\omega)$ is a separating hypersurface of weak
  contact type, then half of $(W,\omega)$ is a weak filling of
  $(M,\xi)$ and the above argument provides a contradiction.  To
  finish the proof of the theorem, it thus remains to show that
  $(M,\xi)$ under the given assumptions can never occur as a
  \emph{nonseparating} hypersurface of weak contact type in any closed
  symplectic $4$--manifold $(W,\omega)$.  This follows from almost the
  same argument, due to the following trick introduced in
  \cite{AlbersBramhamWendl}.  If $M$ does not separate~$W$, then we
  can cut $W$ open along~$M$ to produce a connected symplectic
  cobordism $(W_0,\omega_0)$ between $(M,\xi)$ and itself, and then
  attach an infinite chain of copies of this cobordism to obtain a
  \emph{noncompact} symplectic manifold $(W_\infty,\omega_\infty)$
  with weakly contact boundary $(M,\xi)$.  Though noncompact,
  $(W_\infty,\omega_\infty)$ is \emph{geometrically bounded} in a
  certain sense, and an argument in \cite{AlbersBramhamWendl} uses the
  monotonicity lemma to show that for a natural class of adapted
  almost complex structures on~$W_\infty$, any connected moduli space
  of $J$--holomorphic curves with boundary on $\p W_\infty$ and
  uniformly bounded energy also satisfies a uniform $C^0$--bound.  In
  light of this, the above argument for the compact filling also works
  in the ``noncompact filling'' furnished by
  $(W_\infty,\omega_\infty)$, thus proving that $(M,\xi)$ cannot occur
  as a nonseparating weakly contact hypersurface.

  We will use this same trick again in the proof of
  Theorem~\ref{theorem: planarTorsion}.  In relation to
  Theorem~\ref{thm:planar}, it also implies that in any closed
  symplectic $4$--manifold, a weakly contact hypersurface that is
  planar must always be separating.  This is closely related to
  Etnyre's theorem \cite{Etnyre_planar} that planar contact manifolds
  never admit weak semifillings with disconnected boundary, which also
  can be shown using holomorphic curves, by a minor variation on the
  proof of Theorem~\ref{thm:planar}.
\end{remark}

\begin{remark}
  It should be possible to generalize the Bishop family idea still
  further by considering ``overtwisted planar surfaces'' with
  arbitrarily many boundary components (Figure~\ref{fig:OTplanar}).
  The disk or annulus would then be replaced by a $k$--holed sphere
  $\Sigma$ for some integer $k \ge 1$, with Legendrian boundary, of
  which $k-1$ of the boundary components are ``anchored'' by
  $\SS^1$--families of half-twisted annuli.  The characteristic
  foliation on~$\Sigma$ must in general have $k-2$ hyperbolic singular
  points.  One would then find Bishop families of annuli near the
  anchored boundary components, which eventually must collide with
  each other and could be glued at the hyperbolic singularities to
  produce more complicated $1$--dimensional families of rational
  holomorphic curves with multiple boundary components, leading in the
  end to a more general filling obstruction.

  One situation where such an object definitely exists is in the
  presence of planar torsion (see \secref{subsec:review}), though we
  will not pursue this approach here, as that setting lends itself
  especially well to the punctured holomorphic curve techniques
  explained in the next section.

  \begin{figure}[htbp]
    \centering
    \includegraphics[width=0.3\textwidth,
    keepaspectratio]{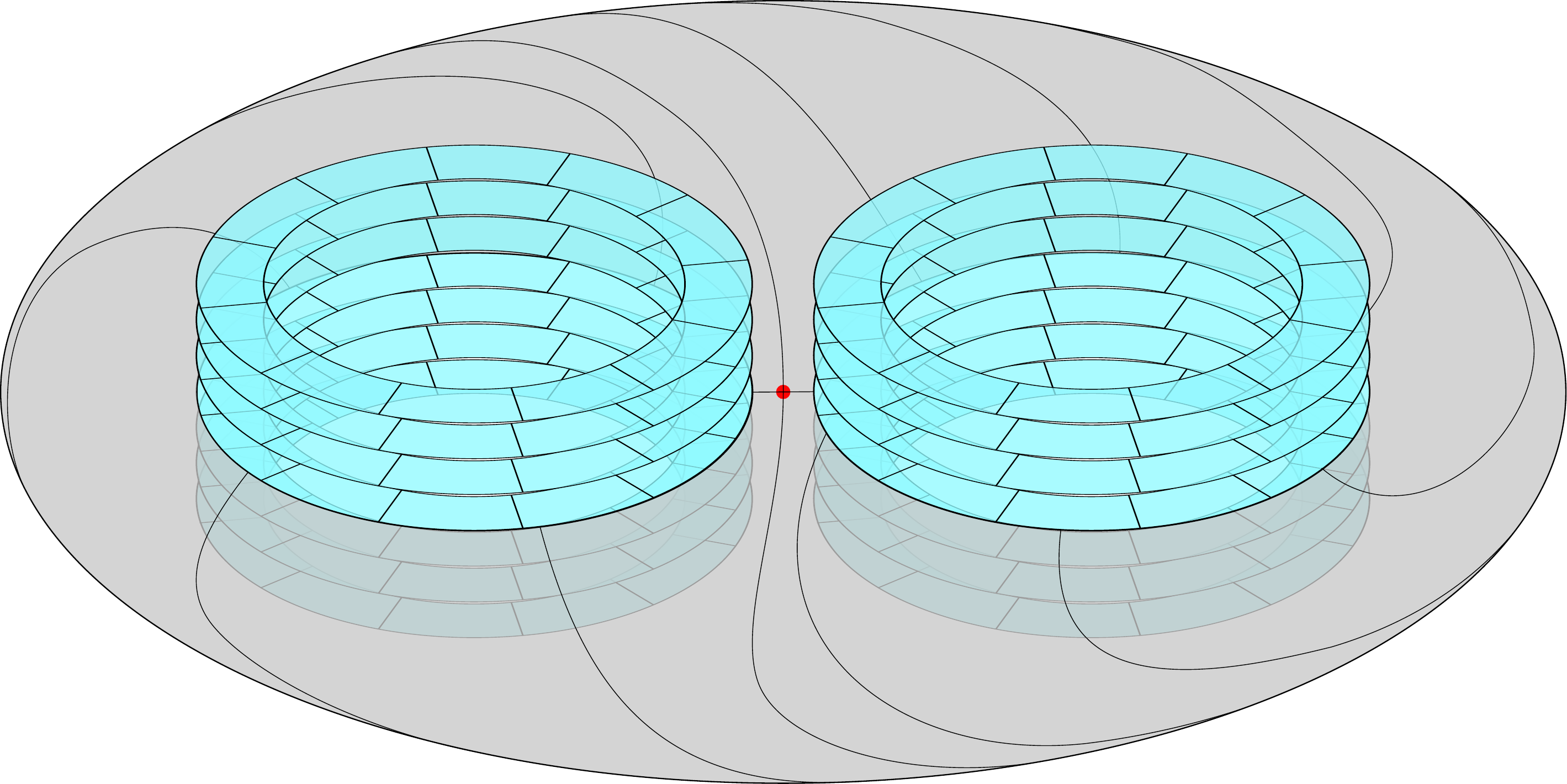}
    \caption{An overtwisted planar surface anchored at two boundary
      components.}\label{fig:OTplanar}
  \end{figure}
\end{remark}

\section{Punctured pseudoholomorphic curves and weak
  fillings}\label{sec:punctured}

We begin this section by showing that up to symplectic deformation,
every weak filling can be enlarged by symplectically attaching a
cylindrical end in which the theory of finite energy punctured
$J$--holomorphic curves is well behaved.  This fact is standard in the
case where the symplectic form is exact near the boundary: indeed,
Eliashberg \cite{EliashbergContactProperties} observed that if
$(W,\omega)$ is a weak filling of $(M,\xi)$ and $H^2_\dR(M) = 0$, then
one can always deform $\omega$ in a collar neighborhood of $\p W$ to
produce a strong filling of $(M,\xi)$, which can then be attached
smoothly to a half-symplectization of the form $\bigl([0,\infty)
\times M, d(e^t\alpha)\bigr)$.  For obvious cohomological reasons,
this is not possible whenever $[\restricted{\omega}{M}] \ne 0 \in
H^2_\dR(M)$.  The solution is to work in the more general context of
stable Hamiltonian structures, in which~$M$ carries a closed maximal
rank $2$--form that is not required to be exact.  We will recall in
\secref{sec: weak to stable} the important properties of stable
hypersurfaces and stable Hamiltonian structures, proving in particular
(Proposition~\ref{prop:cohomology}) that there exist stable
Hamiltonian structures representing every de Rham cohomology class.
We will then use this in \secref{sec: collar neighborhood weak
  boundary} to prove Theorem~\ref{theorem: stableHypersurface}, that
weak boundaries can always be deformed to stable hypersurfaces.  A
quick review of the definition and essential facts about planar
torsion will then be given in \secref{subsec:review}, leading in
\secref{subsec:proof} to the proofs of Theorems~\ref{thm:planar}
and~\ref{theorem: planarTorsion}.

\subsection{Stable hypersurfaces and stable Hamiltonian structures}
\label{sec: weak to stable}

Let us recall some important definitions.  The first originates in
\cite{HoferZehnder}.

\begin{defn}\label{def: stableHypersurface}
  Given a symplectic manifold $(W,\omega)$, a hypersurface $M$ is
  called \textbf{stable} if it is transverse to a vector field $Y$
  defined near $M$ whose flow $\Phi^t_Y$ for small $\abs{t}$ preserves
  characteristic line fields, i.e.~if $M_t := \Phi^t_Y(M)$ and $\ell_t
  \subset TM_t$ is the kernel of $\restricted{\omega}{TM_t}$, then
  $(\Phi^t_Y)_* \ell_0 = \ell_t$.
\end{defn}

As an important special case, if $(W,\omega)$ is a strong filling of
$(M,\xi)$, then $\p W$ is stable, as it is transverse to an outward
pointing Liouville vector field which dilates $\omega$ and therefore
preserves characteristic line fields.  In this case we say the
boundary of~$W$ is \textbf{convex}; if $\p W$ is instead transverse to
an \emph{inward} pointing Liouville vector field, we say it is
\textbf{concave}.

Stable hypersurfaces were initially introduced in order to study
dynamical questions, but it was later recognized that they also yield
suitable settings for the theory of punctured $J$--holomorphic curves.
In this context, the following more intrinsic notion was introduced in
\cite{BourgeoisCompactness}.

\begin{defn}
  A \textbf{stable Hamiltonian structure} on an oriented $3$--manifold
  $M$ is a pair
  \begin{equation*}
    \mathcal{H} = (\lambda,\Omega)
  \end{equation*}
  consisting of a $1$--form $\lambda$ and $2$--form $\Omega$ such that
  \begin{enumerate}
  \item $d\Omega = 0$,
  \item $\lambda \wedge \Omega > 0$,
  \item $\ker\Omega \subset \ker(d\lambda)$.
  \end{enumerate}
\end{defn}

The second condition implies that $\Omega$ has maximal rank and is
nondegenerate on the distribution
\begin{equation*}
  \xi := \ker\lambda \;,
\end{equation*}
so that $(\xi,\Omega)$ is a symplectic vector bundle.  There is then a
positively transverse vector field~$X$ uniquely determined by the
conditions
\begin{equation*}
  \Omega(X,\cdot) = 0, \qquad \lambda(X) = 1\;,
\end{equation*}
and the flow of~$X$ preserves both~$\xi$ and $\Omega$.  Conversely, a
triple $(X,\xi,\Omega)$ satisfying these properties uniquely
determines $(\lambda,\Omega)$, and thus can be taken as an alternative
definition of a stable Hamiltonian structure.

If $M \subset (W,\omega)$ is a stable hypersurface and $Y$ is the
transverse vector field of Definition~\ref{def: stableHypersurface},
then we can orient~$M$ in accordance with the coorientation determined
by~$Y$ and assign to it a stable Hamiltonian structure
$(\lambda,\Omega)$ defined as follows:
\begin{equation}
  \label{eqn:SHS}
  \lambda := \restricted{\bigl(\iota_Y \omega\bigr)}{TM},
  \qquad\text{ and }\qquad \Omega := \restricted{\omega}{TM} \;.
\end{equation}
Now $\Omega$ is obviously closed and nondegenerate on $\xi :=
\ker\lambda$, and the stability condition implies that for any vector
$X$ in the characteristic line field on $M$,
\begin{equation*}
  \restricted{\bigl(\lie{Y} \omega\bigr)(X,\cdot)}{\xi} = 0 \;.
\end{equation*}
From this it is an easy exercise to verify that the pair
$(\lambda,\Omega)$ satisfies the conditions of a stable Hamiltonian
structure.

Given a $3$--manifold $M$ with stable Hamiltonian structure
$(\lambda,\Omega)$, the $2$--form
\begin{equation}\label{eqn:SHSomega}
  \omega := \Omega + d(t\lambda)
\end{equation}
on $(-\epsilon,\epsilon) \times M$ is symplectic for sufficiently
small $\epsilon > 0$.  Conversely, and more generally (cf.\ Lemma~2.3
in \cite{CieliebakMohnkeCompactness}):

\begin{lemma}\label{normalform for weak collars}
  Let $(W,\omega)$ be a symplectic $4$--manifold whose interior
  contains a closed oriented hypersurface $M \subset W$, and let
  $\lambda$ be a nonvanishing $1$--form on~$M$ that defines a
  cooriented (and thus also oriented) $2$--plane distribution $\xi$.
  Assume $\restricted{\omega}{\xi} > 0$.  Then writing $\Omega =
  \restricted{\omega}{TM}$, there exists an embedding
  \begin{equation*}
    \Phi:\, (-\epsilon,\epsilon) \times M \hookrightarrow W
  \end{equation*}
  for sufficiently small $\epsilon > 0$, such that $\Phi(0,\cdot)$ is
  the inclusion and
  \begin{equation*}
    \Phi^*\omega = \Omega + d(t\lambda) \;.
  \end{equation*}
\end{lemma}
\begin{proof}
  Since $\omega$ is nondegenerate on~$\xi$, there is a unique vector
  field $X_\omega$ on~$M$ determined by the conditions
  $\omega(X_\omega,\cdot) \equiv 0$ and $\lambda(X_\omega) \equiv 1$.
  Choose a smooth section $Y$ of $\restricted{TW}{M}$ such that $Y$
  also lies in the $\omega$--complement of $\xi$ and $\omega(Y,
  X_\omega) \equiv 1$.  Extend this arbitrarily as a nowhere zero
  vector field on some neighborhood of $M$.  Then $Y$ is transverse
  to~$M$, and $\restricted{(\iota_Y\omega)}{TM} = \lambda$.

  Using the flow $\Phi_Y^t$ of~$Y$, we can define for sufficiently
  small $\epsilon > 0$ an embedding
  \begin{equation*}
    \Phi:\, (-\epsilon,\epsilon) \times M \to W, \,
    (t,p) \mapsto \Phi_Y^t(p)\;,
  \end{equation*}
  and compare $\omega_0 := \Phi^*\omega$ with the model $\omega_1 :=
  d(t\,\lambda) + \Omega$ on $(-\epsilon,\epsilon) \times M$,
  shrinking~$\epsilon$ if necessary so that $\omega_1$ is symplectic.
  Then $\omega_1$ and $\omega_0$ are symplectic forms that match
  identically along $\{0\} \times M$, and the usual Moser deformation
  argument provides an isotopy between them on a neighborhood of
  $\{0\} \times M$.
\end{proof}

This result has an obvious analog for the case $\p W = M$.  Given
this, if $(W,\omega)$ is any symplectic manifold with stable boundary
$\p W = M$ and $\mathcal{H} = (\lambda, \Omega)$ is an induced stable
Hamiltonian structure, then one can glue a cylindrical end $[0,\infty)
\times M$ symplectically to the boundary as follows.  Choose $\epsilon
> 0$ sufficiently small so that
\begin{equation}\label{eqn:makeItSymplectic}
  \left(\Omega +
    t\,d\lambda \right)|_\xi > 0 \quad \text{ for all $\abs{t} \le
    \epsilon$},
\end{equation}
and let $\mathcal{T}$ denote the set of smooth functions
\begin{equation*}
  \phi:\, [0,\infty) \to [0,\epsilon)
\end{equation*}
which satisfy $\phi(t) = t$ for $t$ near~$0$ and $\phi' > 0$
everywhere.  Then if a neighborhood of $\p W$ is identified with
$(-\epsilon,0] \times M$ as above, we can define the completed
manifold
\begin{equation*}
  W^\infty := W \cup \bigl( [0,\infty) \times M \bigr)
\end{equation*}
by the obvious gluing, and assign to it a $2$--form
\begin{equation}\label{eqn:completion}
  \omega_\phi := \begin{cases}
    \omega & \text{ in $W$},\\
    \Omega + d(\phi\lambda) & \text{ in $[0,\infty) \times M$}
  \end{cases}
\end{equation}
which is symplectic for any $\phi \in \mathcal{T}$ due to
\eqref{eqn:makeItSymplectic}.  There is also a natural class
$\complex(\omega, \mathcal{H})$ of almost complex structures on
$W^\infty$, where we define $J$ to be in $\complex(\omega,
\mathcal{H})$ if
\begin{enumerate}
\item $J$ is compatible with $\omega$ on~$W$,
\item $J$ is $\R$--invariant on $[0,\infty) \times M$, maps $\p_t$ to
  $X$ and restricts to a complex structure on $\xi$ compatible with
  $\restricted{\Omega}{\xi}$.
\end{enumerate}
Then any $J \in \complex(\omega, \mathcal{H})$ is compatible with any
$\omega_\phi$ for $\phi \in \mathcal{T}$.  Observe that whenever
$\lambda$ is a contact form, the conditions characterizing $J \in
\complex(\omega, \mathcal{H})$ on the cylindrical end depend on
$\lambda$, but \emph{not on $\Omega$}, as $\restricted{J}{\xi}$ is
compatible with $\restricted{\Omega}{\xi}$ if and only if it is
compatible with $\restricted{d\lambda}{\xi}$.  In this case we simply
say that $J$ is \textbf{compatible} with~$\lambda$ on the cylindrical
end.

For $J \in \complex(\omega, \mathcal{H})$, we define the
\textbf{energy} of a $J$--holomorphic curve $u:\, \dot{\Sigma} \to
W^\infty$ by
\begin{equation*}
  E(u) = \sup_{\phi \in \mathcal{T}} \int u^*\omega_\phi \;.
\end{equation*}
Then $E(u) \ge 0$, with equality if and only if $u$ is constant.  It
is straightforward to show that this notion of energy is equivalent to
the one defined in \cite{BourgeoisCompactness}, in the sense that
uniform bounds on either imply uniform bounds on the other.  Thus if
$\dot{\Sigma}$ is a punctured Riemann surface, finite energy
$J$--holomorphic curves have \emph{asymptotically cylindrical}
behavior at nonremovable punctures, i.e.~they approach closed orbits
of the vector field~$X$ at $\{+\infty\} \times M$.

The most popular example of a stable Hamiltonian structure is
$(\lambda,\Omega) = (\alpha,d\alpha)$, where $\alpha$ is a contact
form; this is the case that arises naturally on the boundary of a
strong filling.  One can then obtain other stable Hamiltonian
structures in the form
\begin{equation}\label{eqn:SHScontact}
  (\lambda,\Omega) = (\alpha,F\,d\alpha) \;,
\end{equation}
for any function $F:\, M \to (0,\infty)$ such that $dF \wedge d\alpha
= 0$.  In fact, since $\ker(d\alpha)$ is a vector bundle of rank~$1$
whenever $\xi = \ker\alpha$ is contact, \emph{every} stable
Hamiltonian structure in this case has the form of
\eqref{eqn:SHScontact}, and the vector field $X$ is the usual Reeb
vector field $X_\alpha$.  In this context it will be useful to know
that one can choose $F$ so that $F\,d\alpha$ may lie in any desired
cohomology class.  In order to formulate a sufficiently general
version of this statement, we will need the following definition.

\begin{defn}\label{defn:symmetric}
  Suppose $K \subset (M,\xi)$ is a transverse knot.  We will say that
  a contact form $\alpha$ for~$\xi$ is in \textbf{standard symmetric
    form near~$K$} if a neighborhood $\nbhd(K) \subset M$ of $K$ can
  be identified with a solid torus $\SS^1 \times \Disk \ni (\theta ;
  \rho,\phi)$, thus defining positively oriented cylindrical
  coordinates in which $K = \{ \rho = 0 \}$ and $\alpha$ takes the
  form
  \begin{equation*}
    \alpha = f(\rho)\, d\theta + g(\rho)\, d\phi
  \end{equation*}
  for some smooth functions $f , g :\, [0,1] \to \R$ with $f(0) > 0$
  and $g(0) = 0$.
\end{defn}

Recall that by the contact neighborhood theorem, there always exists a
contact form in standard symmetric form near any knot transverse to
the contact structure.  The condition that $\alpha$ is a positive
contact form in these coordinates then amounts to the condition
$f(\rho) g'(\rho) - f'(\rho) g(\rho) > 0$ for $\rho > 0$, and $g''(0)
> 0$.  An oriented knot is called \textbf{positively transverse} if
its orientation matches the coorientation of the contact structure; in
this case its orientation must always match the orientation of the
$\theta$--coordinate in the above definition.

\begin{remark}\label{remark:nondegenerate}
  Recall that a contact form $\alpha$ is called \textbf{nondegenerate}
  whenever its Reeb vector field $X_\alpha$ admits only nondegenerate
  periodic orbits.  The transverse knot $K \subset M$ is always the
  image of a periodic orbit if $\alpha$ is in standard symmetric form
  near~$K$.  Then after multiplying $\alpha$ by a smooth function that
  depends only on~$\rho$, one can always arrange without loss of
  generality that $K$ and all its multiple covers are
  \emph{nondegenerate} orbits and are the only periodic orbits in a
  small neighborhood of~$K$.  In this way we can always find
  nondegenerate contact forms that are in standard symmetric form
  near~$K$.
\end{remark}

\begin{proposition}\label{prop:cohomology}
  Suppose $(M,\xi)$ is a contact $3$--manifold,
  \begin{equation*}
    K = K_1 \cup \dotsb \cup K_n \subset  M
  \end{equation*}
  is an oriented positively transverse link, $N_K \subset M$ is a
  neighborhood of $K$ and $\alpha$ is a contact form for~$\xi$ that is
  in standard symmetric form near~$K$.  Then for any set of positive
  real numbers $c_1,\dotsc,c_n > 0$, there exists a smooth function
  $F:\, M \to (0,\infty)$ such that the following conditions are
  satisfied:
  \begin{enumerate}
  \item $(\alpha, F\,d\alpha)$ is a stable Hamiltonian structure.
  \item $F \equiv 1$ on $M\setminus N_K$ and $F$ is a positive
    constant on a smaller neighborhood of~$K$.
  \item $[ F\,d\alpha] \in H^2_\dR(M)$ is Poincaré dual to $c_1\,
    [K_1] + \dotsb + c_n\, [K_n] \in H_1(M ; \R)$.
  \end{enumerate}
\end{proposition}

\begin{remark}
  Since every oriented link has a $C^0$--small perturbation that makes
  it positively transverse (see for example \cite{Geiges_book}),
  \emph{every} homology class in $H_1(M ; \R)$ can be represented by a
  finite linear combination
    \begin{equation*}
      c_1\, [K_1] + \dotsb + c_n\, [K_n]
    \end{equation*}
    where $c_1,\dotsc,c_n > 0$ and $K_1 \cup \dotsb \cup K_n$ is a
    positively transverse link.
\end{remark}

\begin{remark}
  A few days after the first version of this paper was made public,
  Cieliebak and Volkov unveiled a comprehensive study of stable
  Hamiltonian structures \cite{CieliebakVolkovSHS} which includes an
  existence result closely related to
  Proposition~\ref{prop:cohomology}, and valid also in higher
  dimensions.
\end{remark}

\begin{proof}[Proof of Proposition~\ref{prop:cohomology}]
  We will have $[F\,d\alpha] = \PD\bigl(c_1[K_1] + \dotsb +
  c_n[K_n]\bigr)$ if and only if
  \begin{equation*}
    \int_S F\,d\alpha = \sum_{i=1}^n c_i\, [K_i] \bullet [S]
  \end{equation*}
  for every closed oriented surface $S \subset M$.  Then a
  function~$F$ with the desired properties can be constructed as
  follows.  By assumption, each component $K_i \subset K$ comes with a
  tubular neighborhood $\nbhd(K_i) \subset N_K$ that is identified
  with $\SS^1 \times \Disk \ni (\theta ; \rho,\phi)$, on which
  $\alpha$ has the form
  \begin{equation*}
    \alpha = f_i(\rho)\, d\theta + g_i(\rho) \, d\phi
  \end{equation*}
  for some smooth functions $f_i, g_i :\, [0,1] \to \R$ with $f_i(0) >
  0$ and $g_i(0) = 0$.  Denote the union of all these coordinate
  neighborhoods by $\nbhd(K)$.  Now choose $h:\, M \to (0,\infty)$ to
  be any smooth function with the following properties:
  \begin{enumerate}
  \item The support of~$h$ is in the interior of $\nbhd(K)$.
  \item On each neighborhood $\nbhd(K_i)$, $h$ depends only on the
    $\rho$--coordinate, and restricts to a function $h_i(\rho)$ that
    is constant for $\rho$ near~$0$ and satisfies
    \begin{equation*}
      2\pi \int_0^1 h_i(\rho)\, g_i'(\rho)\,d\rho = c_i \;.
    \end{equation*}
  \end{enumerate}
  Now for any closed oriented surface $S \subset M$, we can deform $S$
  so that its intersection with $\nbhd(K)$ is a finite union of disks
  of the form $\{\theta_0\} \times \Disk \subset \SS^1 \times \Disk$
  for each $x = (\theta_0,0,0) \in K_i \cap S$, each oriented
  according to the intersection index $\sigma(x) = \pm 1$.  Thus if we
  set $F = 1 + h$, then
  \begin{equation*}
    \begin{split}
      \int_S F\,d\alpha &=  \int_S d\alpha +  \int_S h\,d\alpha \\
      &= \sum_{i=1}^n \sum_{x \in K_i \cap S} \sigma(x)
      \int_{\Disk} h_i(\rho)\, g_i'(\rho)\,d\rho \wedge d\phi \\
      &= \sum_{i=1}^n c_i\, [K_i] \bullet [S] \;,
    \end{split}
  \end{equation*}
  as desired.
\end{proof}

\subsection{Collar neighborhoods of weak boundaries}\label{sec: collar
  neighborhood weak boundary}

The application of punctured holomorphic curve methods to weak
fillings is made possible by the following result.

\begin{thm}\label{theorem: stableHypersurface}
  Suppose $(W,\omega)$ is a symplectic $4$--manifold with weakly
  contact boundary $(M,\xi)$, $K = K_1 \cup \dotsb \cup K_n \subset M$
  is a positively transverse link with positive numbers
  $c_1,\dotsc,c_n > 0$ such that the homology class
  \begin{equation*}
    c_1\,[K_1] + \dotsb + c_n\, [K_n] \in H_1(M; \R)
  \end{equation*}
  is Poincaré dual to $[\restricted{\omega}{TM}] \in H^2_\dR(M)$,
  $\nbhd(K)$ is a tubular neighborhood of~$K$, $\lambda$ is a contact
  form for $\xi$ that is in standard symmetric form near~$K$
  (cf.~Definition~\ref{defn:symmetric}), and $\nbhd(M) \subset W$ is a
  collar neighborhood of~$\p W$.  Then there exists a symplectic form
  $\widehat{\omega}$ on~$W$ such that
  \begin{enumerate}
  \item $\widehat{\omega} = \omega$ on $W \setminus \nbhd(M)$,
  \item $M$ is a stable hypersurface in $(W,\widehat{\omega})$, with
    an induced stable Hamiltonian structure of the form $(C\,\lambda,
    F\,d\lambda)$ for some constant $C > 0$ and smooth function $F:\,
    M \to (0,\infty)$ that is constant near~$K$ and outside of
    $\nbhd(K)$.
  \end{enumerate}
\end{thm}

In light of Proposition~\ref{prop:cohomology}, the result will be an
easy consequence of the lemmas proved below, which construct various
types of symplectic forms on collar neighborhoods, compatible with
given distributions on the boundary.  For later applications
(particularly in \secref{sec:handles}), it will be convenient to
assume that the distribution $\xi = \ker\lambda$ is not necessarily
contact; we shall instead usually assume it is a \emph{confoliation},
which means
\begin{equation*}
  \lambda \wedge d\lambda \ge 0 \;.
\end{equation*}
Observe that if $\Omega$ is the restriction of a symplectic form
$\omega$ on $(-\epsilon,0] \times M$ to the boundary, and $\lambda$ is
a nonvanishing $1$--form on~$M$ with $\xi = \ker\lambda$, then
$\restricted{\omega}{\xi} > 0$ if and only if
\begin{equation*}
  \lambda \wedge \Omega > 0 \;.
\end{equation*}
Conversely, whenever this inequality is satisfied for a $1$--form
$\lambda$ and $2$--form $\Omega$ on~$M$, one can define a symplectic
form on $(-\epsilon,0] \times M$ for sufficiently small $\epsilon > 0$
by the formula
\begin{equation*}
  d(t\,\lambda) + \Omega \;,
\end{equation*}
where $t$ denotes the coordinate on the interval $(-\epsilon,0]$.
Lemma~\ref{normalform for weak collars} shows that $\omega$ can always
be assumed to be of this form in the right choice of coordinates.  The
following lemma then provides a symplectic interpolation between any
two cohomologous symplectic structures of this form for a fixed
confoliation~$\xi$, as long as we are willing to rescale the
$1$--form~$\lambda$.

\begin{lemma}\label{lemma:interpolation}
  Suppose $M$ is a closed oriented $3$--manifold, and fix the
  following data:
  \begin{itemize}
  \item $\uU, \uU' \subset M$ are open subsets with $\overline{\uU}
    \subset \uU'$,
  \item $\xi \subset TM$ is a cooriented confoliation, defined as the
    kernel of a nonvanishing $1$--form $\lambda$ such that $\lambda
    \wedge d\lambda \ge 0$,
  \item $\Omega_0$ and $\Omega_1$ are closed, cohomologous $2$--forms
    that are both positive on~$\xi$ and satisfy
    \begin{equation*}
      \Omega_1 = \Omega_0 + d\eta
    \end{equation*}
    for some $1$--form $\eta$ with compact support in~$\uU$.
  \end{itemize}
  Then for any $\epsilon > 0$ sufficiently small, $[-\epsilon,0]
  \times M$ admits a symplectic form $\omega$ which satisfies
  $\restricted{\omega}{\xi} > 0$ on $\{0\} \times M$ and the following
  additional properties:
  \begin{enumerate}
  \item $\omega = d(t\lambda) + \Omega_0$ in a neighborhood of
    $\{-\epsilon\} \times M$ and outside of $[-\epsilon,0] \times
    \uU'$,
  \item $\omega = d(\varphi\,\lambda) + \Omega_1$ in a neighborhood of
    $\{0\} \times M$, where $\varphi :\, [-\epsilon,0] \times M \to
    [-\epsilon,\infty)$ is a smooth function that depends only on~$t$
    in $[-\epsilon,0] \times \uU$ and satisfies $\p_t\varphi > 0$
    everywhere.
  \end{enumerate}
\end{lemma}
\begin{proof}
  Assume $\epsilon > 0$ is small enough so that $\lambda \wedge
  (\Omega_1 - \epsilon \,d\lambda)$ and $\lambda \wedge (\Omega_0 -
  \epsilon \,d\lambda)$ are both positive volume forms.  Choose smooth
  functions $\varphi:\,[-\epsilon,0] \times M \to [-\epsilon,\infty)$
  and $f:\,[-\epsilon,0] \to [0,1]$ such that $f(t) =0$ for~$t$
  near~$-\epsilon$ and $f(t) = 1$ for~$t$ near~$0$, while
  $\varphi(t,p) = t$ whenever~$t$ is near $-\epsilon$ or $p \in
  M\setminus \uU'$, and $\p_t\varphi > 0$ everywhere.  The latter
  gives rise to a smooth family of functions
  \begin{equation*}
    \varphi_t = \varphi(t,\cdot) :\, M \to \R \;,
  \end{equation*}
  for which we shall also assume that $d\varphi_t$ vanishes outside of
  $\uU' \setminus \overline{\uU}$ for all $t \in [-\epsilon,0]$.  We
  must then show that under these conditions, $\varphi$ can be chosen
  so that the closed $2$--form
  \begin{equation*}
    \omega := d\bigl(\varphi\, \lambda\bigr) + \Omega_0
    + d\bigl(f\,\eta\bigr)
  \end{equation*}
  is nondegenerate, where $f$ is lifted in the obvious way to a
  function on $[-\epsilon,0] \times M$.  We compute,
  \begin{equation*}
    \begin{split}
      \omega \wedge \omega & = 2 \p_t\varphi\,dt \wedge \lambda \wedge
      \left[ (1-f)\, \Omega_0 + f\, \Omega_1 +
        \varphi_t\, d\lambda \right] \\
      & \qquad + 2 f'\,dt\wedge \eta \wedge \bigl[ (1-f)\, \Omega_0 +
      f\,\Omega_1 + \varphi_t\, d\lambda \bigr] + 2 f' \,dt \wedge
      \eta \wedge d\varphi_t \wedge \lambda \; ,
    \end{split}
  \end{equation*}
  and observe that the first of the three terms is a positive volume
  form, while the second vanishes outside of $[-\epsilon,0] \times
  \uU$ due to the compact support of~$\eta$, and the third vanishes
  everywhere since the supports of $d\varphi_t$ and~$\eta$ are
  disjoint.  Thus if $\varphi$ is chosen with $\p_t\varphi$
  sufficiently large on $[-\epsilon,0]\times \uU$, the first term
  dominates the second and we have $\omega \wedge \omega > 0$
  everywhere.  The condition $\restricted{\omega}{\xi} > 0$ on $\{0\}
  \times M$ is now immediate from the construction.
\end{proof}

Combining Proposition~\ref{prop:cohomology} with this lemma in the
special case $\uU = M$, Theorem~\ref{theorem: stableHypersurface} now
follows from the observation that if $(\lambda,\Omega)$ is a stable
Hamiltonian structure such that $\lambda$ is contact, and $\varphi$ is
a strictly increasing smooth positive function on some interval
in~$\R$, then the level sets $\{T\} \times M$ are all stable
hypersurfaces with respect to the symplectic form $d(\varphi\,
\lambda) + \Omega$, inducing the stable Hamiltonian structure
$(\varphi'(T)\,\lambda , \varphi(T)\, d\lambda + \Omega)$ on such a
hypersurface.

For the handle attaching argument in \secref{sec:handles}, we will
also need a variation on Lemma~\ref{lemma:interpolation} that changes
$\lambda$ instead of~$\omega$.

\begin{lemma}\label{lemma:interpolation2}
  Suppose $M$ is a closed oriented $3$--manifold, and fix the
  following data:
  \begin{itemize}
  \item $\uU, \uU' \subset M$ are open subsets with $\overline{\uU}
    \subset \uU'$,
  \item $\{\xi_\tau\}_{\tau \in [0,1]}$ is a $1$--parameter family of
    confoliations, defined via a smooth $1$--parameter family of
    nonvanishing $1$--forms $\lambda_\tau$ with $\lambda_\tau \wedge
    d\lambda_\tau \ge 0$, all of which are identical outside of~$\uU$,
  \item $\Omega$ is a closed $2$--form that is positive on~$\xi_\tau$
    for all $\tau \in [0,1]$.
  \end{itemize}
  Then for any $\epsilon > 0$ sufficiently small, $[-\epsilon,0]
  \times M$ admits a symplectic form $\omega$ which satisfies
  $\restricted{\omega}{\xi_1} > 0$ on $\{0\} \times M$ and the
  following additional properties:
  \begin{enumerate}
  \item $\omega = d(t\,\lambda_0) + \Omega$ in a neighborhood of
    $\{-\epsilon\} \times M$ and outside of $[-\epsilon,0] \times
    \uU'$,
  \item $\omega = d(\varphi\,\lambda_1) + \Omega$ in a neighborhood of
    $\{0\} \times M$, where $\varphi :\, [-\epsilon,0] \times M \to
    [-\epsilon,\infty)$ is a smooth function that depends only on~$t$
    in $[-\epsilon,0] \times \uU$ and satisfies $\p_t\varphi > 0$
    everywhere.
  \end{enumerate}
\end{lemma}
\begin{proof}
  Assume $\epsilon > 0$ is small enough so that $\lambda_\tau \wedge
  (\Omega - \epsilon\, d\lambda_\tau) > 0$ for all $\tau \in [0,1]$.
  Pick a smooth function
  \begin{equation*}
    [-\epsilon,0] \to [0,1] :\, t \mapsto \tau
  \end{equation*}
  such that $\tau = 0$ for all $t$ near~$-\epsilon$ and $\tau = 1$ for
  all~$t$ near~$0$, and use this to define a $1$--form $\Lambda$ on
  $[-\epsilon,0] \times M$ by
  \begin{equation*}
    \Lambda_{(t,m)} = \left(\lambda_{\tau}\right)_m
  \end{equation*}
  for all $(t,m) \in [-\epsilon,0] \times M$.  Next, choose a smooth
  function $\varphi:\,[-\epsilon,0] \times M \to [-\epsilon,\infty)$
  such that $\varphi(t,m) = t$ whenever~$t$ is near~$-\epsilon$ or $m
  \in M\setminus \uU'$, and $\p_t\varphi > 0$ everywhere.  Denote by
  \begin{equation*}
    \varphi_t = \varphi(t,\cdot) : M \to \R \;,
  \end{equation*}
  the resulting smooth family of functions, and assume also that
  $d\varphi_t$ vanishes outside of $\uU' \setminus \overline{\uU}$ for
  all $t \in [-\epsilon,0]$.  Now set
  \begin{equation*}
    \omega = d\bigl(\varphi\, \Lambda\bigr) + \Omega
  \end{equation*}
  and compute:
  \begin{equation*}
    \omega \wedge \omega = 2 \p_t\varphi \, dt \wedge \lambda_\tau \wedge
    \left(\Omega + \varphi_t\, d\lambda_\tau \right)
    + \left( \varphi_t \, d\Lambda \right)^2 + 2 \varphi_t\,
    d\Lambda \wedge \Omega + 2 \varphi_t\, d\varphi_t \wedge \lambda_\tau
    \wedge d\Lambda \;.
  \end{equation*}
  The first term is a positive volume form and can be made to dominate
  the second and third if $\p_t\varphi$ is large enough; note that the
  second and third terms also vanish completely outside of
  $[-\epsilon,0] \times \uU$ since $\lambda_\tau$ is then independent
  of~$\tau$, so that $\Lambda$ reduces to a $1$--form on~$M$ and both
  terms are thus $4$--forms on a $3$--manifold.  For the same reason,
  the last term vanishes everywhere.
\end{proof}

\subsection{Review of planar torsion}
\label{subsec:review}

In this section we recall the important definitions and properties of
planar torsion; we shall give only the main ideas here, referring to
\cite{ChrisOpenBook2} for further details.

Recall that an \textbf{open book decomposition} of a closed oriented
$3$--manifold $M$ is a fibration $\pi:\, M \setminus B \to \SS^1$,
where the \textbf{binding} $B \subset M$ is an oriented link, and the
fibers are oriented surfaces with embedded closures whose oriented
boundary is~$B$.  The fibers are connected if and only if~$M$ is
connected, and we call the connected components of the fibers
\textbf{pages}.  We wish to consider two topological operations that
can be performed on an open book:
\begin{enumerate}
\item \emph{Blowing up} a binding circle $\gamma \subset B$: this
  means replacing $\gamma$ by the unit circle bundle in its normal
  bundle, or equivalently, removing a small neighborhood of~$\gamma$
  so that $M$ becomes a manifold $\widehat{M}$ with $2$--torus
  boundary.  Defining $\widehat{B} = B \setminus \gamma$, the
  fibration $\pi:\, M \setminus B \to \SS^1$ now induces a fibration
  \begin{equation*}
    \hat{\pi}:\, \widehat{M} \setminus \widehat{B} \to \SS^1 \;.
  \end{equation*}
  The structure associated with this fibration is called a
  \textbf{blown up open book} with binding~$\widehat{B}$.  Observe
  that $\p\widehat{M}$ also carries a distinguished $1$--dimensional
  homology class, arising from the meridian on the tubular
  neighborhood of~$\gamma$.
\item \emph{The binding sum}: consider two distinct binding circles
  $\gamma_1, \gamma_2 \subset B$, which come with distinguished
  trivializations of their normal bundles $\nu\gamma_1,\nu\gamma_2$
  determined by the open book.  Any orientation preserving
  diffeomorphism $\gamma_1 \to \gamma_2$ is then covered by a unique
  (up to homotopy) orientation reversing isomorphism
  \begin{equation*}
    \Phi :\, \nu\gamma_1 \to \nu\gamma_2
  \end{equation*}
  which is constant with respect to the distinguished trivializations.
  Blowing up both $\gamma_1$ and $\gamma_2$, we obtain a manifold
  $\widehat{M}$ with two torus boundary components $\p_1\widehat{M}$
  and $\p_2\widehat{M}$, and $\Phi$ determines a unique (up to
  isotopy) orientation reversing diffeomorphism
  \begin{equation*}
    \widehat{\Phi} :\, \p_1\widehat{M} \to \p_2\widehat{M} \;,
  \end{equation*}
  which we may assume restricts to orientation preserving
  diffeomorphisms between boundary components of fibers
  of~$\hat{\pi}$.  Gluing $\p_1\widehat{M}$ and $\p_2\widehat{M}$
  together via~$\widehat{\Phi}$ then gives a new closed manifold
  $\check{M}$, containing a distinguished torus $\interface \subset
  \check{M}$, called the \textbf{interface}, which also carries
  distinguished $1$--dimensional homology classes (unique up to sign)
  determined by the meridians.  Due to the orientation reversal, the
  fibration is not well defined on the interface, but it determines a
  fibration
  \begin{equation*}
    \check{\pi} :\, \check{M} \setminus (\check{B} \cup \interface)
    \to \SS^1 \;,
  \end{equation*}
  where $\check{B} := B \setminus (\gamma_1 \cup \gamma_2)$.  The
  associated structure is called a \textbf{summed open book} with
  binding~$\check{B}$ and interface~$\interface$.  If $M_1$ and $M_2$
  are two distinct manifolds with open books, one can attach them by
  choosing some collection of binding circles in~$M_1$, pairing each
  with a distinct binding circle in~$M_2$ and constructing the binding
  sum for each pair.  We use the shorthand notation
  \begin{equation*}
    M_1 \boxplus M_2
  \end{equation*}
  for any manifold and summed open book constructed from two open
  books in this way.
\end{enumerate}

Clearly both operations can also be performed on binding components of
blown up or summed open books, so iterating them finitely many times
we can produce a more complicated manifold (possibly with boundary),
carrying a more general decomposition known as a \textbf{blown up
  summed open book}.  If $M$ carries such a structure, then it comes
with a fibration
\begin{equation*}
  \pi :\, M \setminus (B \cup \interface) \to \SS^1 \;,
\end{equation*}
where the \textbf{binding}~$B$ is an oriented link and the
\textbf{interface}~$\interface$ is a disjoint union of tori.  The
connected components of fibers of~$\pi$ are again called
\textbf{pages}, and their closures are generally immersed surfaces, as
they occasionally may have multiple boundary components that coincide
as oriented circles in the interface.  We call a blown up summed open
book \textbf{irreducible} if the fibers $\pi^{-1}(*)$ are all
connected, and \textbf{planar} if they also have genus zero.

Generalizing the standard definition of a contact structure supported
by an open book, we say that a contact form $\alpha$ on $M$ with
induced Reeb vector field~$X_\alpha$ is a \textbf{Giroux form} if it
satisfies the following conditions:
\begin{enumerate}
\item $X_\alpha$ is positively transverse to the interiors of all
  pages,
\item $X_\alpha$ is positively tangent to the boundaries of the
  closures of all pages,
\item The characteristic foliation induced on $\interface \cup \p M$
  by $\ker\alpha$ has closed leaves representing the distinguished
  homology classes determined by meridians.
\end{enumerate}
It follows that the interface and boundary are always foliated by
closed orbits of the Reeb vector field for any Giroux form.  We say
that a contact structure $\xi$ is \textbf{supported} by the summed
open book whenever it is the kernel of a Giroux form.

\begin{example}\label{ex:S1invariant}
  Suppose $\Sigma$ is a compact, connected and oriented surface,
  possibly with boundary, and $\xi$ is a positive, cooriented and
  $\SS^1$--invariant contact structure on $\SS^1 \times \Sigma$, such
  that the curves $\SS^1 \times \{z\}$ are Legendrian for all $z \in
  \p\Sigma$.  We can then divide $\Sigma$ into the following subsets:
  \begin{align*}
    \Sigma_+ &= \{ z \in \Sigma \ |\ \text{$\SS^1 \times \{z\}$ is
      positively transverse} \} \;, \\
    \Sigma_- &= \{ z \in \Sigma \ |\ \text{$\SS^1 \times \{z\}$ is
      negatively transverse} \} \;, \\
    \Gamma &= \{ z \in \Sigma \ |\ \text{$\SS^1 \times \{z\}$ is
      Legendrian} \} \;.
  \end{align*}
  By assumption, $\p\Sigma \subset \Gamma$.  The Lutz construction
  \cite{Lutz_CircleActions} produces such a contact structure for any
  given multicurve $\Gamma$ that contains $\p\Sigma$ and divides
  $\Sigma$ into two separate pieces $\Sigma_+$ and $\Sigma_-$.  In
  fact, one can find a contact form $\alpha$ for $\xi$ such that for
  every $t \in \SS^1$, the Reeb vector field $X_\alpha$ is positively
  transverse to $\{t\} \times \Sigma_+$, negatively transverse to
  $\{t\} \times \Sigma_-$ and tangent to $\{t\} \times \Gamma$.  This
  is thus a Giroux form for a blown up summed open book, whose pages
  are the connected components of $\{t\} \times
  (\Sigma\setminus\Gamma)$, with trivial monodromy.  The interface is
  the union of all the tori $\SS^1 \times \gamma$ for connected
  components $\gamma \subset \Gamma$ in the interior of~$\Sigma$, and
  the binding is empty.
\end{example}

A blown up summed open book is called \textbf{symmetric} if its
boundary and binding are both empty, and it is obtained as a binding
sum of two connected pieces $M_+ \boxplus M_-$, with open books whose
pages are diffeomorphic to each other.  The two simplest examples of
contact structures supported by symmetric summed open books are the
standard contact structures on $\SS^1 \times \SS^2$ and $\T^3$: the
former can be obtained as a binding sum of two open books with
disk-like pages, and the latter as a binding sum of two open books
with cylindrical pages and trivial monodromy.

\begin{defn}\label{defn:planarTorsion}
  A \textbf{planar torsion domain} is any contact $3$-manifold
  $(M,\xi)$, possibly with boundary, together with a supporting blown
  up summed open book that can be obtained as a binding sum of two
  separate nonempty pieces,
  \begin{equation*}
    M = M_0 \boxplus M_1 \;,
  \end{equation*}
  where $M_0$ carries an irreducible planar summed open book without
  boundary, and $M_1$ carries an arbitrary blown up summed open book
  (possibly disconnected), such that the induced blown up summed open
  book on~$M$ is not symmetric.  The interior of~$M$ then contains a
  compact submanifold with nonempty boundary,
  \begin{equation*}
    M^P \subset M \;,
  \end{equation*}
  called the \textbf{planar piece}, which is obtained from $M_0$ by
  blowing up all of its summed binding components.  The closure of $M
  \setminus M^P$ is called the \textbf{padding}.

  We say that a contact $3$--manifold $(M,\xi)$ has \textbf{planar
    torsion} whenever it admits a contact embedding of some planar
  torsion domain.
\end{defn}

Note that the interface of the blown up summed open book on a planar
torsion domain contains the (nonempty) boundary of the planar piece,
and may also have additional components in its interior.

\begin{defn}\label{defn:planarSeparating}
  For any closed $2$--form $\Omega$ on a closed contact $3$--manifold
  $(M,\xi)$, we say that $(M,\xi)$ has \textbf{$\Omega$--separating
    planar torsion} if it contains a planar torsion domain such that
  $\int_L \Omega = 0$ for every interface torus~$L$ in the planar
  piece.  If each of these tori is nullhomologous in $H_2(M;\R)$, then
  we say $(M,\xi)$ has \textbf{fully separating planar torsion}.
\end{defn}

\begin{remark}
  The fully separating condition can only be satisfied when the planar
  piece $M^P \subset M$ has no interface tori in its interior and each
  of its boundary components separates~$M$.  This follows from the
  observation that an interface torus in an irreducible blown up
  summed open book is \emph{always} homologically nontrivial.
\end{remark}

\begin{example}\label{ex:OTGiroux}
  As shown in \cite{ChrisOpenBook2}, any open neighborhood of a Lutz
  twist contains a fully separating planar torsion domain whose planar
  piece has disk-like pages, and in fact planar torsion of this type
  (called \emph{planar $0$--torsion}) is equivalent to
  overtwistedness.  Similarly, a neighborhood of a Giroux torsion
  domain always contains a planar torsion domain whose planar piece
  has cylindrical pages (called \emph{planar $1$--torsion}).
\end{example}

\begin{example}\label{ex:S1invariant2}
  The $\SS^1$--invariant contact manifold $(\SS^1 \times \Sigma,\xi)$
  of Example~\ref{ex:S1invariant} is a planar torsion domain whenever
  $\Sigma\setminus \Gamma$ contains a connected component of genus
  zero whose closure is disjoint from $\p\Sigma$, but which is not
  diffeomorphic to both $\Sigma_+$ and $\Sigma_-$.  The fully
  separating condition is satisfied whenever every boundary component
  of the genus zero piece separates~$\Sigma$.
\end{example}

The following is a combination of two of the main results in
\cite{ChrisOpenBook2}.

\begin{thm}[\cite{ChrisOpenBook2}]\label{thm:holOpenbook}
  If $(M,\xi)$ is a closed contact $3$--manifold with planar torsion
  then it is not strongly fillable.  Moreover, if $M^P \subset M$
  denotes the planar piece of a planar torsion domain in~$M$ and
  $\pi:\, M^P \setminus (B \cup \interface) \to \SS^1$ is the
  associated fibration with binding~$B$ and interface~$\interface$,
  then for any $\epsilon > 0$, $(M,\xi)$ admits a Morse-Bott contact
  form~$\alpha$ and a generic $\R$--invariant almost complex structure
  $J$ on $\R\times M$, compatible with~$\alpha$, such that:
  \begin{itemize}
  \item $\alpha$ is in standard symmetric form (see
    Definition~\ref{defn:symmetric}) near~$B$, and the components
    of~$B$ are nondegenerate elliptic Reeb orbits of Conley-Zehnder
    index~$1$ (with respect to the trivialization determined by the
    open book) and period less than~$\epsilon$.
  \item The interface and boundary tori $\interface \cup \p M \subset
    M^P$ are Morse-Bott submanifolds foliated by Reeb orbits of period
    less than~$\epsilon$.
  \item All Reeb orbits in~$M$ outside of $B \cup \interface \cup \p
    M^P$ have period at least~$1$.
  \item The interior of each planar page $\pi^{-1}(\tau)$ is the
    projection to~$M$ of an embedded finite energy punctured
    $J$--holomorphic curve
    \begin{equation*}
      u_\tau :\, \dot{\Sigma} \to \R \times M \;,
    \end{equation*}
    with only positive ends and Fredholm index~$2$.
  \end{itemize}
\end{thm}

\subsection{Proofs of Theorems~\ref{thm:planar} and~\ref{theorem:
    planarTorsion}}
\label{subsec:proof}

The important feature that Theorems~\ref{thm:planar} and~\ref{theorem:
  planarTorsion} have in common is that they involve weak fillings of
contact manifolds that admit regular families of index~$2$ punctured
holomorphic spheres.  For Theorem~\ref{thm:planar}, the idea will be
to stabilize the boundary so that the pages of a given planar open
book can be lifted to holomorphic curves in the cylindrical end---we
can then repeat precisely the argument used for strong fillings in
\cite{ChrisGirouxTorsion}, as the resulting moduli space spreads into
the filling to form the fibers of a symplectic Lefschetz fibration.
The idea for Theorem~\ref{theorem: planarTorsion} is similar, except
that instead of a Lefschetz fibration, we will get a contradiction.
First however we must take care to stabilize the boundary in such a
way that the desired holomorphic curves in the cylindrical end will
actually exist, and this is not trivial since by Theorem~\ref{theorem:
  stableHypersurface}, we can only choose the contact form freely
outside of a neighborhood of a certain transverse link.

\begin{lemma}\label{lemma:mappingDuTore}
  Suppose $\Sigma$ is a compact oriented surface with nonempty
  boundary, $\varphi :\, \Sigma \to \Sigma$ is a diffeomorphism with
  support away from the boundary, and $\Sigma_\varphi$ denotes the
  mapping torus of~$\varphi$, i.e.~the manifold $(\R \times \Sigma) /
  \sim$ where $(t+1,z) \sim (t,\varphi(z))$ for all $t \in \R$, $z \in
  \Sigma$.  Then for any given connected component $L \subset
  \p\Sigma_\varphi$, every homology class $h \in H_1(\Sigma_\varphi)$
  can be represented as a sum of cycles
  \begin{equation*}
    h = h_\Sigma + h_L \;,
  \end{equation*}
  where $h_\Sigma$ lies in a fiber of the natural fibration
  $\Sigma_\varphi \to \SS^1$, and $h_L$ lies in~$L$.
\end{lemma}
\begin{proof}
  The fibration $\Sigma_\varphi \to \SS^1$ gives rise to an exact
  sequence
  \begin{equation*}
    H_1(\Sigma) \stackrel{\varphi_* - 1}{\longrightarrow} H_1(\Sigma) 
    \stackrel{\iota_*}{\longrightarrow} H_1(\Sigma_\varphi) 
    \stackrel{\Phi}{\longrightarrow} H_0(\Sigma) \cong \Z \;,
  \end{equation*}
  where $\iota :\, \Sigma \to \Sigma_\varphi$ is the inclusion and
  $\Phi$ computes the intersection number of any $1$--cycle in the
  interior of $\Sigma_\varphi$ with a fiber.  Thus if we choose any
  reference cycle $h_0 \in H_1(\Sigma_\varphi)$ that passes once
  transversely through each fiber, the exact sequence implies that any
  $h \in H_1(\Sigma_\varphi)$ decomposes as a sum of the form
  \begin{equation*}
    h = \iota_*(h_\Sigma) + c\, h_0
  \end{equation*}
  for $h_\Sigma \in H_1(\Sigma)$ and $c \in \Z$.  The lemma follows
  since $h_0$ can be represented by a loop in any given connected
  component of~$\p \Sigma_\varphi$.
\end{proof}

Assume $(W,\omega)$ is a weak filling of $(M,\xi)$, and the latter
either is planar or contains a planar torsion domain with planar piece
$M^P \subset M$, whose binding and interface are denoted by
$B^P,\interface^P \subset M^P$ respectively.  In the planar case it
makes sense also to define $M^P = M$ and $\interface^P = \emptyset$,
so in both cases $M^P$ carries a planar blown up summed open book with
binding~$B^P$ and interface $\interface^P$.  After modifying $\omega$
via Theorem~\ref{theorem: stableHypersurface}, we can assume $\p W$ is
a stable hypersurface, with an induced stable Hamiltonian structure of
the form $\mathcal{H} = (\lambda, F\,d\lambda)$, where $\lambda$ is a
contact form for~$\xi$ that is in standard symmetric form near some
positively transverse link $K = K_1 \cup \dotsb \cup K_n$.  The latter
must be chosen so that
\begin{equation}
  \label{eqn:PD}
  \PD\bigl([\restricted{\omega}{TM}]\bigr) = \sum_{i=1}^n c_i\, [K_i]
\end{equation}
for some set of positive real numbers $c_1,\dotsc,c_n > 0$.

\begin{lemma}\label{lemma:outsidePlanar}
  If $\int_L \omega = 0$ for every connected component $L \subset
  \interface^P \cup \p M^P$, then one can choose the positively
  transverse link $K$ to be a disjoint union of three links
  \begin{equation*}
    K = K_B \cup K_P \cup K' \;,
  \end{equation*}
  where $K_B$ is a subcollection of the oriented components of $B^P$,
  $K_P$ lies in a single page in $M^P$ and $K' \subset M \setminus
  M^P$.
\end{lemma}
\begin{proof}
  Note that in the planar case, $M^P = M$ and the condition on the
  boundary and interface is vacuous: then applying
  Lemma~\ref{lemma:mappingDuTore} to the mapping torus of the
  monodromy of the open book, we see that for any oriented binding
  component $\gamma \subset B^P$, any $h \in H_1(M;\R)$ can be written
  as $h = c\,[\gamma] + h_P$ for some $c \in \R$ and $h_P$ is
  represented by a cycle in a page.  If $c < 0$, we can exploit the
  fact that the total binding is the boundary of a page and thus
  rewrite $c\,[\gamma]$ as a positive linear combination of the other
  oriented binding components.

  For the case of a planar torsion domain, we have $\p M^P \ne
  \emptyset$ and must show first that $h =
  \PD\bigl([\restricted{\omega}{TM}]\bigr)$ under the given
  assumptions can be represented by a cycle that does not intersect
  $\interface^P \cup \p M^P$.  The above argument then completes the
  proof.

  To find a representative cycle disjoint from $\interface^P \cup \p
  M^P$, suppose $K = K_1 \cup \dotsb \cup K_n$ is any oriented link
  with $c_1\, [K_1] + \dotsb + c_n\, [K_n]$ Poincaré dual to
  $[\restricted{\omega}{TM}]$ for some real numbers $c_1,\dotsc,c_n
  \ne 0$.  Then for each connected component $L \subset \interface^P
  \cup \p M^P$, Poincaré duality implies
  \begin{equation*}
    \sum_i c_i\, [K_i] \bullet [L] = \int_L \omega = 0 \;.
  \end{equation*}
  We can assume $K$ and~$L$ have only transverse intersections $x \in
  K \cap L$.  Now for each component $K_i$, we can replace $K_i$ by a
  homologous link for which all intersections of $K_i$ with~$L$ have
  the same sign: indeed, if $x , y \in K_i \cap L$ are two
  intersections of opposite sign, we can eliminate both of them by
  splicing $K_i$ with a path between~$x$ and~$y$ along~$L$.  Having
  done this, we can also split $K_i$ into multiple parallel components
  so that each intersects~$L$ either not at all or exactly once.  Then
  by switching orientations of $K_i$ and signs of $c_i$, we can
  arrange for this intersection to be positive.  Let us therefore
  assume that each component $K_i$ has at most one intersection
  with~$L$, which is transverse and positive, so
  \begin{equation*}
    \sum_{\{i ; K_i \cap L \ne \emptyset \} } c_i = 0 \;.
  \end{equation*}
  Now if any intersection $x \in K \cap L$ exists, there must be
  another $y \in K \cap L$ for which the real coefficient has the
  opposite sign; for concreteness let us assume $x \in K_1$, $y \in
  K_2$, $c_1 > 0$ and $c_2 < 0$.  We can then eliminate one of these
  intersections via the following two steps: first, replace $K_2$ by a
  disjoint union of two knots $K_2'$ and $K_2''$, where $K_2' := K_2$
  and $K_2''$ is a parallel copy of it, and set $c_2' := - c_1$,
  $c_2'' := c_2 + c_1$.  This introduces one additional intersection
  $y'' \in K_2'' \cap L$.  But now since $c_2' = -c_1$, we can
  eliminate $x$ and~$y$ by splicing in a path between them along~$L$
  to connect $K_1$ and $K_2'$.  The result of this operation is a new
  link $\tilde{K} = \tilde{K}_1 \cup \dotsb \cup
  \tilde{K}_{\tilde{n}}$ with real numbers
  $\tilde{c}_1,\dotsc,\tilde{c}_{\tilde{n}} \ne 0$ such that
  \begin{equation*}
    \sum_{i=1}^{\tilde{n}} \tilde{c}_i [\tilde{K}_i] =
    \sum_{i=1}^n c_i [K_i]
  \end{equation*}
  and $\tilde{K} \cap L$ contains one point fewer than $K \cap L$.
  One can then repeat this process until the intersection of~$K$ with
  $\interface^P \cup \p M^P$ is empty.  By switching orientations of
  the components $K_i$ again, we can then assume the real coefficients
  $c_1,\dotsc,c_n$ are all positive.
\end{proof}

The lemma has the following consequence: for any fixed page $\Sigma
\subset M^P$, we can now freely choose the contact form $\lambda$ on
some open set~$\uU$,
\begin{equation*}
  \Sigma \cup B^P \cup \interface^P \cup \p M^P \subset \uU \subset M^P \;,
\end{equation*}
to be the one provided by Theorem~\ref{thm:holOpenbook}, for which
there exists a generic almost complex structure $J$ compatible
with~$\mathcal{H}$ such that the pages in $\uU$ lift to embedded
$J$--holomorphic curves of index~$2$ in the symplectization.  Enlarge
$W$ to $W^\infty$ by attaching a cylindrical end, and extend the
compatible $J$ from the end to a generic almost complex structure $J
\in \complex(\omega,\mathcal{H})$ on~$W^\infty$.  After pushing up by
$\R$--translation, the $J$--holomorphic pages in $\R \times \uU$ may
be assumed to live in $[c,\infty) \times M$ for arbitrarily large $c >
0$ and thus can also be regarded as $J$--holomorphic curves in
$W^\infty$.  Since the asymptotic orbits of these curves have much
smaller periods than all other Reeb orbits in~$M$, the connected
$2$--dimensional moduli space $\moduli$ of $J$--holomorphic curves in
$W^\infty$ that contains these curves satisfies a compactness theorem
proved in \cite{ChrisOpenBook2}: namely, $\moduli$ is compact except
for codimension~$2$ nodal degenerations and curves that ``escape'' to
$+\infty$ (and thus converge to curves in $\R\times M$).  Moreover,
the curves in $\moduli$ foliate $W^\infty$ except at a finite set of
nodal singularities, which are transverse intersections of two leaves.
A similar statement holds for the curves in $\R\times M$ that form the
``boundary'' of $\moduli$: observe that for any $m \in M \setminus
(B^P \cup \interface^P \cup \p M^P)$, one can find a sequence $t_k \to
\infty$ such that each of the points $(t_k,m)$ is in the image of a
unique curve $u_k \in \moduli$, and the latter sequence must converge
to a curve in $\R\times M$ whose projection to~$M$ passes through~$m$.
By positivity of intersections using \cite{Siefring_intersection}, any
two of these curves in $\R\times M$ are either identical or disjoint,
and their projections to~$M$ are all embedded, thus forming a
foliation of $M \setminus (B^P \cup \interface^P \cup \p M^P)$ by
holomorphic curves whose asymptotic orbits all lie in the same
Morse-Bott families.  At this point the two proofs diverge in separate
directions.

\begin{proof}[Proof of Theorem~\ref{thm:planar}]
  Following the proof of Theorem~1 in \cite{ChrisGirouxTorsion}, the
  curves in the compactification of the moduli space $\moduli$ form
  the fibers of a Lefschetz fibration
  \begin{equation*}
    \Pi :\, W^\infty \to \Disk \;,
  \end{equation*}
  and the vanishing cycles in this fibration are all homologically
  nontrivial if~$W$ is minimal.  It then follows from Eliashberg's
  topological characterization of Stein manifolds
  \cite{EliashbergSteinManifolds} that $(W,\omega)$ is deformation
  equivalent to a symplectic blow-up of a Stein domain.
\end{proof}

\begin{proof}[Proof of Theorem~\ref{theorem: planarTorsion}]
  Since the planar piece of a planar torsion domain has nonempty
  boundary $\p M^P$ by assumption, one can pick any component $L
  \subset \p M^P$ and define an \emph{asymptotic evaluation map} as in
  \cite{ChrisGirouxTorsion}, which defines an embedding of $\moduli$
  into a certain line bundle over the $\SS^1$--family of orbits
  in~$L$.  It follows that the compactified moduli space
  $\overline{\moduli}$ is diffeomorphic to an annulus, and its curves
  are the fibers of a Lefschetz fibration
  \begin{equation*}
    \Pi :\, W^\infty \to [0,1] \times \SS^1 \;,
  \end{equation*}
  whose boundary is a \emph{symmetric} summed open book.  As shown in
  \cite{ChrisFiberSums} using ideas due to Gompf, such a Lefschetz
  fibration always admits a symplectic structure, unique up to
  symplectic deformation, which produces a \emph{strong} filling of
  the contact manifold supported by the symmetric summed open book.
  But $(M,\xi)$ is not strongly fillable due to
  Theorem~\ref{thm:holOpenbook}, so we have a contradiction.

  It remains to exclude the possibility that $(M,\xi)$ could embed
  into a closed symplectic $4$--manifold $(W,\omega)$ as a
  nonseparating weakly contact hypersurface.  This is ruled out by
  almost the same argument, using the ``infinite chain'' trick of
  \cite{AlbersBramhamWendl}: as explained in
  Remark~\ref{remark:hypersurface}, we can cut $W$ open along~$M$ and
  use it to construct a noncompact but geometrically bounded
  symplectic manifold $(W_\infty,\omega_\infty)$ with weakly contact
  boundary $(M,\xi)$, then attach a cylindrical end and consider the
  above moduli space of holomorphic curves in $W_\infty$.  The
  monotonicity lemma gives a $C^0$--bound for these curves, but the
  same arguments that we used above also imply that they must foliate
  $W_\infty$, which is already a contradiction since $W_\infty$ is
  noncompact by construction.
\end{proof}

\subsection{Contact homology and twisted coefficients}
\label{subsec:SFT}

In this section we will justify Theorem~\ref{thm:SFT} by using the
deformation result Theorem~\ref{theorem: stableHypersurface} to show
that any weak filling $(W,\omega)$ of $(M,\xi)$ gives rise to an
algebra homomorphism from contact homology with suitably twisted
coefficients to a certain Novikov completion of the group ring
${\Q}\big[H_2(M;\R) / \ker[\restricted{\omega}{TM}]\big]$.  Thus if
$\1=0$ in twisted contact homology, the same must be true in the
Novikov ring and we obtain a contradiction.  Since our main goal is to
illustrate the role of twisted coefficients in SFT rather than provide
a rigorous proof, we shall follow the usual custom of ignoring
transversality problems---let us merely point out at this juncture
that abstract perturbations are required (e.g.~within the scheme under
development by Hofer-Wysocki-Zehnder, cf.~\cite{Hofer_polyfoldSurvey})
in order to make the following discussion fully rigorous.

We first briefly review the definition of contact homology, due to
Eliashberg \cite{EliashbergContactInvariants} and
Eliashberg-Givental-Hofer \cite{SymplecticFieldTheory}.  In order to
allow maximal flexibility in the choice of coefficients and avoid
certain complications of bookkeeping (e.g.~torsion in $H_1(M)$), we
will set up the theory with only a $\Z_2$--grading instead of the
usual $\Z$--grading---this choice makes no difference to the vanishing
of the homology and its consequences.  Assume $(M,\xi)$ is a closed
$(2n-1)$--dimensional manifold with a positive and cooriented contact
structure, and $\alpha$ is a contact form for~$\xi$ such that all
closed orbits of the Reeb vector field $X_\alpha$ are nondegenerate.
Each closed Reeb orbit $\gamma$ then has a canonically defined mod~$2$
Conley-Zehnder index, $\CZ(\gamma) \in \Z_2$, which defines the even
or odd \textbf{parity} of the orbit.  An orbit is called \textbf{bad}
if it is the double cover of an orbit with different parity than its
own; all other orbits are called \textbf{good}.  For any linear
subspace $\Ring \subset H_2(M ; \R)$, the group ring $\Q[H_2(M;\R) /
\Ring]$ consists of all finite sums of the form $\sum_{i=1}^N c_i
e^{A_i}$ with $c_i \in \Q$ and $A_i \in H_2(M ; \R) / \Ring$, where
multiplication is defined so that $e^A e^B = e^{A + B}$.  Now let
\begin{equation*}
  \CC{M , \alpha ;\, \Q[H_2(M;\R) / \Ring]}
\end{equation*}
denote the free $\Z_2$--graded supercommutative algebra with unit
generated by the elements of $\Q[H_2(M ; \R) / \Ring]$, which we
define to have even degree, together with the symbols $q_\gamma$ for
every good Reeb orbit~$\gamma$, to which we assign the degree
\begin{equation*}
  \abs{q_\gamma} = n - 3 + \CZ(\gamma) \in \Z_2 \;.
\end{equation*}
Note that orbits with the same image but different periods
(i.e.~distinct covers of the same orbit) give rise to \emph{distinct}
generators in this definition.

To define a differential on $\CC{M,\alpha;\, \Q[H_2(M;\R) / \Ring]}$,
we must make a few more choices.  First, let $C_1,\dotsc,C_N$ denote a
basis of cycles generating $H_1(M ; \R)$, and for each good orbit
$\gamma$, choose a real singular $2$--chain $F_\gamma$ in~$M$ such
that $\p F_\gamma = \gamma - \sum_{i=1}^N d_i C_i$ for a (unique) set
of coefficients $d_i \in \R$.  Choose also an $\R$--invariant almost
complex structure $J$ on $\R\times M$ which is compatible
with~$\alpha$.  Then any punctured finite energy $J$--holomorphic
curve $u : \dot{\Sigma} \to \R\times M$ represents a $2$--dimensional
\emph{relative} homology class, which can be completed uniquely to an
absolute homology class $[u] \in H_2(M ; \R)$ by adding the
appropriate combination of spanning $2$--chains~$F_\gamma$.  Given $A
\in H_2(M ; \R) / \Ring$ and a collection of good Reeb orbits
$\gamma^+, \gamma^-_1,\dotsc,\gamma^-_k$ for some $k \ge 0$, we denote
by
\begin{equation*}
  \moduli^A(\gamma^+ ; \gamma^-_1,\dotsc,\gamma^-_k)
\end{equation*}
the moduli space of unparametrized finite energy punctured
$J$--holomorphic spheres in homology classes representing $A \in
H_2(M;\R) / \Ring$, with one positive cylindrical end approaching
$\gamma^+$, and $k$ ordered negative cylindrical ends approaching
$\gamma^-_1,\dotsc,\gamma^-_k$ respectively.\footnote{Since various
  conflicting conventions appear throughout the literature, we should
  emphasize that our moduli spaces are defined with \emph{ordered}
  punctures and \emph{no asymptotic markers}.  The combinatorial
  factors in \eqref{eqn:differential} and \eqref{eqn:cobordismMap} are
  written with this in mind.}  The components of this moduli space can
be oriented coherently \cite{BourgeoisMohnke}, and we call a curve in
$\moduli^A(\gamma^+ ; \gamma^-_1,\dotsc,\gamma^-_k)$ \textbf{rigid} if
it lives in a connected component of the moduli space that has virtual
dimension~$1$.  The rigid curves in $\moduli^A(\gamma^+ ;
\gamma^-_1,\dotsc,\gamma^-_k)$ up to $\R$--translation can then be
counted algebraically, producing a rational number
\begin{equation*}
  \# \left(\frac{\moduli^A(\gamma^+ ; \gamma^-_1,\dotsc,\gamma^-_k)}{\R}\right)
  \in \Q \;.
\end{equation*}
(Note that since we are allowing the homology class to vary in an
equivalence class within $H_2(M;\R)$, $\moduli^A(\gamma^+ ;
\gamma^-_1,\dotsc,\gamma^-_k)$ may in general contain a mixture of
rigid and non-rigid curves; we ignore the latter in the count.)  We
then define the differential on generators $q_\gamma$ by
\begin{equation}\label{eqn:differential}
  \p q_\gamma = \sum_{k=0}^\infty \sum_{(\gamma_1,\dotsc,\gamma_k)}
  \sum_{A \in H_2(M;\R) / \Ring}
  \frac{\kappa_\gamma}{k!} \cdot
  \#\left( \frac{\moduli^A(\gamma ; \gamma_1,\dotsc,\gamma_k)}{\R} \right)
  e^A q_{\gamma_1} \dotsm q_{\gamma_k},
\end{equation}
where the second sum is over all ordered $k$--tuples
$(\gamma_1,\dotsc,\gamma_k)$ of good orbits, and $\kappa_\gamma \in
\N$ denotes the covering multiplicity of~$\gamma$.  It follows from
the main compactness theorem of Symplectic Field Theory
\cite{BourgeoisCompactness} that this sum is finite, and moreover that
the resulting map
\begin{equation*}
  \p :\, \CC{M,\alpha ;\, \Q[H_2(M;\R) / \Ring]} \to \CC{M,\alpha ;\,
    \Q[H_2(M;\R) / \Ring]} \;,
\end{equation*}
extended uniquely to the complex as a $\Q[H_2(M;\R) / \Ring]$--linear
derivation of odd degree, satisfies $\p^2 = 0$.  The homology of this
complex,
\begin{equation*}
  \HC{M,\xi ;\, \Q[H_2(M;\R) / \Ring]} := 
  H_*\left(\CC{M,\alpha;\, \Q[H_2(M;\R) / \Ring]},\p\right)
\end{equation*}
is a $\Z_2$--graded algebra with unit which is an invariant of the
contact structure~$\xi$, called the \textbf{contact homology} of
$(M,\xi)$ with coefficients in $\Q[H_2(M;\R) / \Ring]$.  We say that
this homology \textbf{vanishes} if it contains only one element; this
is equivalent to the relation $\1=0$, which is true if and only there
exists an element $Q \in \CC{M,\alpha;\, \Q[H_2(M;\R) / \Ring]}$ such
that $\p Q = \1$.  In general, this means there exists a rigid
$J$--holomorphic plane that cannot be ``cancelled'' in an appropriate
sense by other rigid curves with the same positive asymptotic orbit.

Suppose now that $n=2$ and $(W,\omega)$ is a weak filling of
$(M,\xi)$.  By Theorem~\ref{theorem: stableHypersurface}, we can
deform $\omega$ to make the boundary stable, inducing a stable
Hamiltonian structure $\mathcal{H} = (\alpha,\Omega)$ on~$M$ such that
$\alpha$ is a nondegenerate contact form for~$\xi$, and $\Omega$ is a
closed maximal rank $2$--form with
\begin{equation*}
  [\Omega] = [\restricted{\omega}{TM}] \in H^2_\dR(M) \;.
\end{equation*}
We can therefore extend $W$ by attaching a cylindrical end $[0,\infty)
\times M)$ with a symplectic structure of the form $d(\varphi(t)
\alpha) + \Omega$ for some small but increasing function~$\varphi$.
Denote the extended manifold by $W^\infty$, and choose a generic
compatible almost complex structure $J \in
\complex(\omega,\mathcal{H})$ on $W^\infty$.

The following observation is now crucial: since $\Omega$ and $d\alpha$
are conformally equivalent as symplectic structures on~$\xi$, the
compatibility condition for~$J$ on the cylindrical end $[0,\infty)
\times M$ depends only on~$\alpha$, not on~$\Omega$.  Thus~$J$
determines an almost complex structure on the symplectization
$\R\times M$ of precisely the type that is used to define the
differential on $\CC{M,\alpha;\, \Q[H_2(M;\R) / \Ring]}$, and the
breaking of $J$--holomorphic curves in $W^\infty$ into multi-level
curves will generally produce curves that are counted in the
computation of $\HC{M,\xi;\, \Q[H_2(M;\R) / \Ring]}$.  The only
difference between this and the case of a \emph{strong} filling is the
definition of energy, which does involve $\Omega$, but this makes no
difference for the count of curves in $\R\times M$.

Relatedly, one can now define another version of contact homology with
coefficients that depend on the filling: defining a complex
$\CC{M,\alpha ;\, \Q[H_2(W;\R) / \ker\omega]}$ the same way as above
but replacing $\Q[H_2(M;\R) / \Ring]$ with $\Q[H_2(W;\R) /
\ker\omega]$, \eqref{eqn:differential} yields a differential
\begin{equation*}
  \p_W :\, \CC{M,\alpha ;\, \Q[H_2(W;\R) / \ker\omega]}
  \to \CC{M,\alpha ;\, \Q[H_2(W;\R) / \ker\omega]}
\end{equation*}
by interpreting the term $e^A$ as an element of $\Q[H_2(W;\R) /
\ker\omega]$ through the canonical map $H_2(M;\R) \to H_2(W;\R)$
induced by the inclusion $M \hookrightarrow W$.  We denote the
homology of this complex by
\begin{equation*}
  \HC{M,\xi ;\, \Q[H_2(W;\R) / \ker\omega]} =
  H_*\left(\CC{M,\alpha;\, \Q[H_2(W;\R) / \ker\omega]},\p_W\right) \;,
\end{equation*}
and observe that since the canonical map $H_2(M;\R) \to H_2(W;\R)$
takes $\ker\Omega$ into $\ker\omega$, there is also a natural algebra
homomorphism
\begin{equation*}
  \HC{M,\xi ;\, \Q[H_2(M;\R) / \ker\Omega]} \to 
  \HC{M,\xi ;\, \Q[H_2(W;\R) / \ker\omega]} \;.
\end{equation*}
The right hand side therefore vanishes whenever the left hand side
does.

With this understood, we shall now count rigid $J$--holomorphic curves
in $W^\infty$ to define an algebra homomorphism from $\HC{M,\xi ;\,
  \Q[H_2(W;\R) / \ker\omega]}$ into a certain Novikov completion of
$\Q[H_2(W;\R) / \ker\omega]$.  Choose a basis of $1$--cycles
$Z_1,\dotsc,Z_m$ for the image of $H_1(M;\R)$ in $H_1(W;\R)$, and for
each of the basis cycles $C_i$ in~$M$, choose a real $2$--chain $G_i$
in~$W$ such that $\p G_i = C_i - \sum_{j=1}^m d_j Z_j$ for some
(unique) coefficients $d_j \in \R$.  Then for any finite energy
punctured $J$--holomorphic curve $u : \dot{\Sigma} \to W^\infty$ with
positive cylindrical ends approaching Reeb orbits in~$M$, these
choices allow us again to define an absolute homology class $[u] \in
H_2(W;\R)$ by adding the relative homology class to the appropriate
sum of the spanning $2$--chains $F_\gamma$ and~$G_i$.

For any Reeb orbit $\gamma$ in~$M$ and $A \in H_2(W;\R) / \ker\omega$,
denote by
\begin{equation*}
  \moduli^A(\gamma)
\end{equation*}
the moduli space of unparametrized finite energy $J$--holomorphic
planes in $W^\infty$ in homology classes representing~$A$, with a
positive end approaching the orbit~$\gamma$.  We call such a plane
\textbf{rigid} if its connected component of the moduli space has
virtual dimension~$0$.  Since the natural homomorphism $[\omega] :\,
H_2(W;\R) \to \R$ descends to $H_2(W;\R) / \ker\omega$, the
holomorphic curves in $\moduli^A(\gamma)$ satisfy a uniform energy
bound depending on~$A$ and~$\gamma$, thus the compactness theory
implies that $\moduli^A(\gamma)$ contains finitely many rigid curves.
These can again be counted algebraically (ignoring the non-rigid
curves) to define a rational number $\# \moduli^A(\gamma) \in \Q$.
Now for any good Reeb orbit $\gamma$ in~$M$, define the formal sum
\begin{equation}\label{eqn:cobordismMap}
  \Phi_W(q_\gamma) = \sum_{A \in H_2(W) / \ker\omega} 
  \kappa_\gamma \cdot \#\left( \moduli^A(\gamma) \right) e^A \;.
\end{equation}
This sum is not generally finite unless $\omega$ is exact, but it does
belong to the Novikov ring $\Lambda_\omega$, which we define to be the
completion of $\Q[H_2(W;\R) / \ker \omega]$ obtained by including
infinite formal sums
\begin{equation*}
  \left\{ \sum_{i=1}^\infty c_i e^{A_i} \ \Big|\ c_i \in \Q\setminus\{0\},
    \ A_i \in H_2(W;\R) / \ker \omega,
    \ \langle [\omega] , A_i \rangle \to +\infty \right\} \;.
\end{equation*}
One can extend $\Phi_W$ uniquely as an algebra homomorphism
\begin{equation*}
  \Phi_W :\, \CC{M,\alpha ;\, \Q[H_2(W;\R) / \ker\omega]} \to \Lambda_\omega \;,
\end{equation*}
which we claim descends to the homology $\HC{M,\xi ;\, \Q[H_2(W;\R) /
  \ker\omega]}$.  This follows by considering the boundary of the
union of all $1$--dimensional connected components of
$\moduli^A(\gamma)$: indeed, this boundary is precisely the set of all
broken rigid curves, consisting of an upper level in $\R\times M$ that
has a positive end approaching~$\gamma$ and an arbitrary number of
negative ends, which are capped off by a lower level formed by a
disjoint union of planes in~$W^\infty$.  Counting these broken rigid
curves yields the identity
\begin{equation*}
  \Phi_W \circ \p_W = 0 \;,
\end{equation*}
implying that $\Phi_W$ descends to an algebra homomorphism
\begin{equation*}
  \Phi_W :\, \HC{M,\xi ;\, \Q[H_2(W;\R) / \ker\omega]} \to \Lambda_\omega \;.
\end{equation*}
Theorem~\ref{thm:SFT} follows immediately, because we now have a
sequence of algebra homomorphisms
\begin{equation*}
  \HC{M,\xi ;\, \Q[H_2(M;\R) / \ker\Omega]} \to
  \HC{M,\xi ;\, \Q[H_2(W;\R) / \ker\omega]} \to \Lambda_\omega \;,
\end{equation*}
for which $\1 \ne 0$ on the right hand side.

\section{Toroidal symplectic $1$--handles}
\label{sec:handles}

In this section we introduce a symplectic handle attachment technique
that can be used to construct weak fillings of contact manifolds.  To
apply the method in general, we need the following ingredients:
\begin{itemize}
\item A weakly fillable contact manifold $(M,\xi)$, possibly
  disconnected,
\item Two disjoint homologically nontrivial pre-Lagrangian tori $T_+,
  T_- \subset (M,\xi)$ with characteristic foliations that are linear
  and rational,
\item Choices of $1$--cycles $K_\pm \subset T_\pm$ that intersect each
  leaf once,
\item A (possibly disconnected) weak filling $(W,\omega)$ of $(M,\xi)$
  such that $\omega$ restricts to an area form on the tori $T_\pm$ and
  (with appropriate choices of orientations) $\int_{T_+} \omega =
  \int_{T_-} \omega$.
\end{itemize}
Note that examples of this setup are easy to find: for instance if
$(W_\pm,\omega_\pm)$ are a pair of \emph{strong} fillings of contact
manifolds $(M_\pm,\xi_\pm)$ which contain pre-Lagrangian tori $T_\pm
\subset M_\pm$ with $[T_\pm] \ne 0 \in H_2(W_\pm ; \R)$, one may
assume after a perturbation that the characteristic foliations on
$T_\pm$ are rational.  Furthermore one can deform the symplectic
structures $\omega_\pm$ so that they vanish on $T_\pm$, and find
closed $2$--forms $\sigma_\pm$ on $W_\pm$ such that
$\sigma_\pm|_{T_\pm} > 0$ and $\int_{T_\pm} \sigma_\pm = 1$.  Then for
any $\epsilon > 0$ sufficiently small,
\begin{equation*}
  (W_+, \omega_+ + \epsilon\, \sigma_+) \sqcup
  (W_-, \omega_- + \epsilon\, \sigma_-)
\end{equation*}
is a weak filling of $(M,\xi) := (M_+,\xi_+) \sqcup (M_-,\xi_-)$ with
the desired properties.  We will use a construction of this sort in
the proof of Theorem~\ref{thm:weakFillings}.

Given this data, we will show that a new symplectic manifold with
weakly contact boundary can be produced by attaching to $W$ a
\emph{toroidal $1$--handle}
\begin{equation*}
  \T^2 \times [0,1] \times [0,1]
\end{equation*}
along $T_+ \sqcup T_-$.  The effect of this on the contact manifold
can be described as a contact topological operation called
\emph{splicing}, which essentially cuts $(M,\xi)$ open along $T_+$ and
$T_-$ and then reattaches it along a homeomorphism that swaps
corresponding boundary components.  The result of this operation
depends on the isotopy class of the map used when identifying the
boundary tori, but a choice can be specified uniquely by requiring
that this map take the generators of $H_1(T_-, \Z)$ represented by the
cycle $K_-$ and a leaf of the characteristic foliation to the
corresponding generators of $H_1(T_+,\Z)$.

We shall describe this topological operation in
\secref{subsec:splicing}, and prove a general result on toroidal
symplectic handle attaching in \secref{subsec:attaching}, leading in
\secref{subsec:fillingsConstruction} to the proof of
Theorem~\ref{thm:weakFillings}.  As an easy by-product of the setting
we use for handle attaching, we will also see why fillability is
preserved under Lutz twists along symplectic pre-Lagrangian tori.

\subsection{Pre-Lagrangian tori, splicing and Lutz twists}
\label{subsec:splicing}

Assume $(M,\xi)$ is a contact $3$--manifold, let $T\subset M$ be an
embedded and oriented pre-Lagrangian torus with rational linear
characteristic foliation, and choose a $1$--cycle $K \subset T$ that
intersects each characteristic leaf once.  We can find a
contactomorphism between a neighborhood of $T$ and the local model
\begin{equation*}
  \bigl(\T^2 \times (-\epsilon,\epsilon),\,
  \ker (d\theta + r \,d\phi)\bigr) \;,
\end{equation*}
where we use the coordinates $(\phi,\theta;r)$ on the thickened torus
$\T^2 \times (-\epsilon,\epsilon)$, such that $T$ is identified with
$\T^2 \times \{0\}$ with its natural orientation, and the
$\theta$--cycles are homologous to~$K$ up to sign.  This
identification is uniquely defined up to isotopy.  We shall refer to
the coordinates $(\phi,\theta;r)$ chosen in this way as
\textbf{standard coordinates} near $(T,K)$.

Now suppose $(T_+,K_+)$ and $(T_-,K_-)$ are two pairs as described
above, with $T_+ \cap T_- = \emptyset$, and choose disjoint
neighborhoods $\nbhd(T_\pm)$ together with standard coordinates
$(\phi,\theta;r)$.  The coordinates divide each of the neighborhoods
$\nbhd(T_\pm)$ into two halves:
\begin{equation*}
  \nbhd^+(T_\pm) := \bigl\{ r \in [0,\epsilon) \bigr\} \subset \nbhd(T_\pm)
  \quad\text{ and }\quad
  \nbhd^-(T_\pm) := \bigl\{ r \in (-\epsilon,0] \bigr\} \subset \nbhd(T_\pm)\;.
\end{equation*}
We can then construct a new contact manifold $(M',\xi')$ by the
following steps (see Figure~\ref{fig: splicing}):
\begin{enumerate}
\item Cut $M$ open along $T_+$ and $T_-$, producing a manifold with
  four pre-Lagrangian torus boundary components $\p \nbhd^+(T_+)$, $\p
  \nbhd^-(T_+)$, $\p\nbhd^+(T_-)$ and $\p\nbhd^-(T_-)$.
\item Attach $\nbhd^-(T_-)$ to $\nbhd^+(T_+)$ and $\nbhd^-(T_+)$ to
  $\nbhd^+(T_-)$ so that the standard coordinates glue together
  smoothly.
\end{enumerate}
The resulting contact manifold $(M',\xi')$ is uniquely defined up to
contactomorphism, and it also contains a distinguished pair of
pre-Lagrangian tori $T'_\pm$, namely
\begin{equation*}
  T'_+ := \nbhd^+(T_+) \cap \nbhd^-(T_-) \subset M'
  \quad\text{ and }\quad
  T'_- := \nbhd^-(T_+) \cap \nbhd^+(T_-) \subset M' \;.
\end{equation*}

\begin{figure}[htbp]
  \centering
  \includegraphics[width=0.9\textwidth, keepaspectratio]{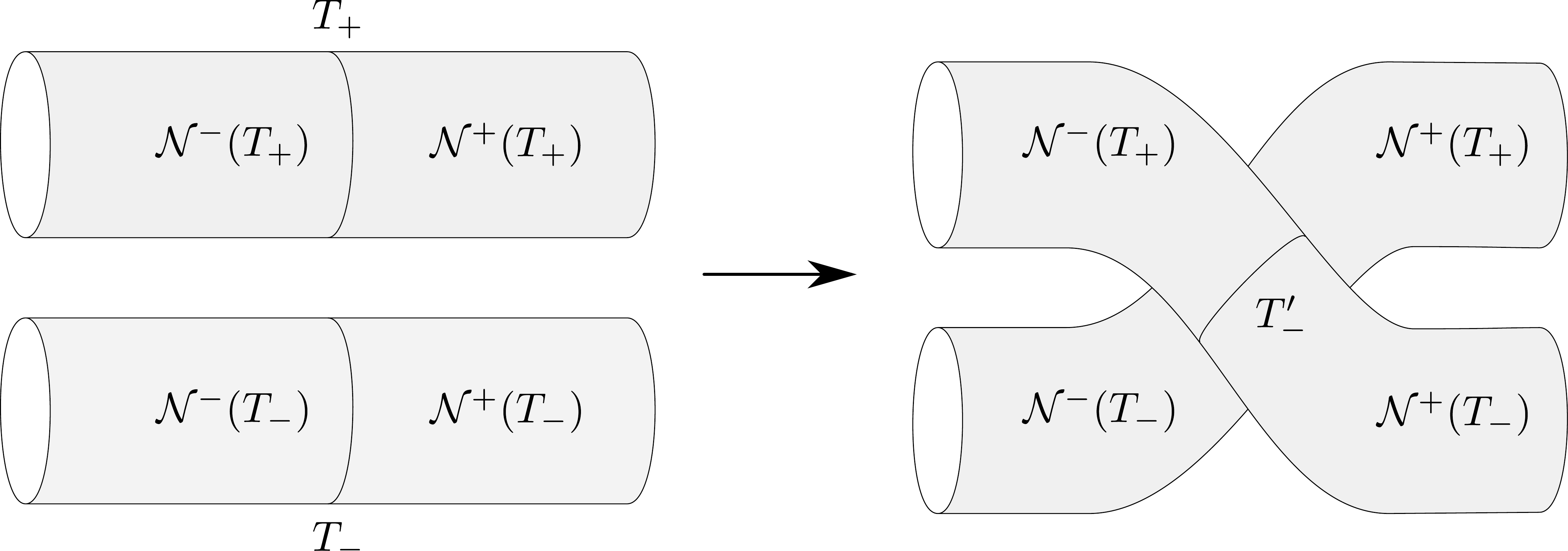}
  \caption{Splicing along tori.}\label{fig: splicing}
\end{figure}

\begin{defn}\label{defn:splicing}
  We will say that $(M',\xi')$ constructed above is the contact
  manifold obtained from $(M,\xi)$ by \textbf{splicing} along
  $(T_+,K_+)$ and $(T_-,K_-)$.
\end{defn}

\begin{example}\label{ex:LutzTwist}
  Consider the tight contact torus $(\T^3,\xi_n)$ for $n \in \N$,
  where
  \begin{equation*}
    \xi_n = \ker\bigl[ \cos(2\pi n \rho)\, d\theta
    + \sin(2\pi n \rho)\, d\phi \bigr]
  \end{equation*}
  in coordinates $(\phi,\theta,\rho) \in \T^3$.  Then $T_0 := \{ \rho
  = 0 \}$ is a pre-Lagrangian torus, to which we assign the natural
  orientation induced by the coordinates $(\phi,\theta)$.  If
  $(M,\xi)$ is another connected contact $3$--manifold with an
  oriented pre-Lagrangian torus $T \subset M$, then splicing $(M,\xi)
  \sqcup (\T^3,\xi_n)$ along $T$ and $T_0$ produces a new connected
  contact manifold, namely the one obtained from $(M,\xi)$ by
  performing~$n$ \textbf{Lutz twists} along~$T$.  If $T \subset M$ is
  compressible then the resulting contact manifold is overtwisted; by
  contrast, Lutz twists along incompressible tori can be used to
  construct tight contact manifolds with arbitrarily large Giroux
  torsion.  Note that in this example the choice of the transverse
  cycles on $T$ and $T_0$ does not influence the resulting manifold.
\end{example}

\begin{remark}\label{remark:circle_action_and_splicing}
  Note that if $(M,\xi)$ is a contact $3$--manifold with an
  $\SS^1$--action such that the oriented pre-Lagrangian tori $T_+,T_-
  \subset M$ consist of Legendrian $\SS^1$--orbits, then the splicing
  operation can be assumed compatible with the circle action, in the
  sense that the manifold $(M',\xi')$ obtained by splicing is then
  also an $\SS^1$--manifold, with the tori $T_\pm'$ consisting of
  Legendrian orbits.

  If sections $\sigma_\pm$ of the $\SS^1$--action are given in a
  neighborhood of the tori $T_+,T_-$ in $(M,\xi)$, then we can obtain
  any desired intersection number $e_+$ between $\sigma_- \cap
  \nbhd^-(T_-)$ and $\sigma_+ \cap \nbhd^+(T_+)$ in $T_+'$ by letting
  the cycle $K_-$ be the intersection $\sigma_- \cap T_-$, and
  choosing a cycle $K_+$ that has intersection number $e_+$ with
  $\sigma_+$.  The intersection number $e_-$ between $\sigma_+ \cap
  \nbhd^-(T_+)$ and $\sigma_- \cap \nbhd^+(T_-)$ in $T_-'$ will always
  be equal to $-e_+$.

  Note in particular that we can arrange for the sections $\sigma_\pm$
  to glue smoothly after splicing by choosing both cycles $K_\pm
  \subset T_\pm$ to be the intersections $\sigma_\pm \cap T_\pm$.
\end{remark}

\subsection{Attaching handles}
\label{subsec:attaching}

Given $\delta > 0$, we will use the term \textbf{toroidal $1$--handle}
to refer to the smooth manifold with boundary and corners,
\begin{equation*}
  \handle_\delta = \T^2 \times [-\delta,\delta] \times [-\delta,\delta] \;.
\end{equation*}
Let $(\phi,\theta;r,r')$ denote the natural coordinates on
$\handle_\delta$, and label the smooth pieces of its boundary
$\p\handle_\delta = \p^N \handle_\delta \cup \p^S\handle_\delta \cup
\p^W\handle_\delta \cup \p^E\handle_\delta$ as follows:
\begin{equation*}
  \p^N\handle_\delta = \{ r' = +\delta \}, \qquad
  \p^S\handle_\delta = \{ r' = -\delta \}, \qquad
  \p^W\handle_\delta = \{ r = -\delta \}, \text{ and }
  \p^E\handle_\delta = \{ r = +\delta \} \;.
\end{equation*}
Observe that if we assign the natural boundary orientations to each of
these pieces, then the induced coordinates $(\phi,\theta;r)$ are
negatively oriented on $\p^N\handle_\delta$ but positively oriented on
$\p^S\handle_\delta$; similarly, the coordinates $(\phi,\theta;r')$
are negatively oriented on $\p^W\handle_\delta$, and positively
oriented on $\p^E\handle_\delta$.

Suppose $(M,\xi)$ is a contact manifold, $W = (-\epsilon,0] \times M$
is a collar neighborhood with $\p W = M$, and $T_+, T_- \subset M$ are
oriented pre-Lagrangian tori with transverse $1$--cycles $K_\pm
\subset T_\pm$ and standard coordinates $(\phi,\theta;r)$ on a pair of
disjoint neighborhoods
\begin{equation*}
  \T^2 \times (-\epsilon,\epsilon) \cong \nbhd(T_\pm) \subset M \;.
\end{equation*}
Choosing $\delta$ with $0 < \delta < \epsilon$, we can \textbf{attach}
$\handle_\delta$ to~$W$ along $(T_+,K_+)$ and $(T_-,K_-)$ via the
orientation reversing embeddings
\begin{equation*}
  \begin{split}
    \Phi :\,& \p^N\handle_\delta \hookrightarrow \nbhd(T_+), \,
    (\phi,\theta;r,\delta)
    \mapsto (\phi,\theta;r) \\
    \Phi :\,& \p^S\handle_\delta \hookrightarrow \nbhd(T_-), \,
    (\phi,\theta;r,-\delta) \mapsto (\phi,\theta;-r) \;.
  \end{split}
\end{equation*}
Then if $W' = W \cup_\Phi \handle_\delta$, after smoothing the
corners, the new boundary $M' = \p W'$ is diffeomorphic to the
manifold obtained from~$M$ by splicing along~($T_+,K_+)$
and~$(T_-,K_-)$, where the distinguished tori $T_\pm' \subset M'$ are
naturally identified with
\begin{equation*}
  T_\pm' = \T^2 \times \bigl\{(\pm \delta,0)\bigr\} \subset
  \p^W \handle_\delta \cup \p^E \handle_\delta \subset M' \;.
\end{equation*}
The main result of this section is that such an operation can also be
defined in the symplectic and contact categories.

\begin{thm}\label{thm:handles}
  Suppose $(W,\omega)$ is a symplectic manifold with weakly contact
  boundary $(M,\xi)$, and $T_+, T_- \subset M$ are disjoint, oriented
  pre-Lagrangian tori with rational linear characteristic foliations
  and transverse $1$--cycles $K_\pm \subset T_\pm$, such that $T_\pm$
  are also symplectic with respect to~$\omega$, with
  \begin{equation*}
    \int_{T_+} \omega = \int_{T_-} \omega > 0\;.
  \end{equation*}
  Then after a symplectic deformation of $\omega$ near $T_+ \cup T_-$,
  $\omega$ extends to a symplectic form $\omega'$ on the manifold
  \begin{equation*}
    W' = W \cup \handle_\delta
  \end{equation*}
  obtained by attaching a toroidal $1$--handle $\handle_\delta$ to~$W$
  along~$(T_+,K_+)$ and $(T_-,K_-)$, so that $(W',\omega')$ then has
  weakly contact boundary $(M',\xi')$, where the latter is obtained
  from $(M,\xi)$ by splicing along $(T_+,K_+)$ and~$(T_-,K_-)$.
\end{thm}

As we saw in Example~\ref{ex:LutzTwist}, Lutz twists along a
pre-Lagrangian torus $T \subset (M,\xi)$ can always be realized by
splicing $(M,\xi)$ together with a tight contact $3$--torus, and due
to the construction of Giroux \cite{Giroux_plusOuMoins}, the latter
admits weak fillings for which the pre-Lagrangian tori $\{ \rho =
\text{const} \}$ are symplectic.  Thus whenever $(M,\xi)$ has weak
filling $(W,\omega)$ and $T \subset M = \p W$ is a torus that is both
pre-Lagrangian in $(M,\xi)$ and symplectic in $(W,\omega)$, the above
theorem can be used to construct weak fillings of every contact
manifold obtained by performing finitely many Lutz twists along~$T$.
We will see however that the setup needed to prove the theorem yields
a much more concrete construction of such a filling:

\begin{thm}\label{thm:LutzTwists}
  Suppose $(W,\omega)$ is a symplectic manifold with weakly contact
  boundary $(M,\xi)$, and $T \subset M$ is a pre-Lagrangian torus
  which is also symplectic with respect to~$\omega$.  Then for any $n
  \in \N$, $(W,\omega)$ can be deformed symplectically so that it is
  also positive on~$\xi_n$, where the latter is obtained from~$\xi$ by
  performing~$n$ Lutz twists along~$T$.
\end{thm}

To prove both of these results, we begin by constructing a suitable
symplectic deformation of a weak filling near any symplectic
pre-Lagrangian torus.  The local setup is as follows: let
\begin{equation*}
  M = \T^2 \times [-5\epsilon,5\epsilon]
\end{equation*}
with coordinates $(\phi,\theta;r)$ and contact structure $\xi =
\ker\lambda$, where
\begin{equation*}
  \lambda = d\theta + r\, d\phi \;.
\end{equation*}
Define also
\begin{equation*}
  W = (-\epsilon,0] \times M
\end{equation*}
with coordinates $(t;\phi,\theta;r)$, and identifying $M$ with $\p W =
\{0\} \times M$, assume $\Omega$ is a closed $2$--form on~$M$ such
that $\restricted{\Omega}{\xi} > 0$, and
\begin{equation*}
  \omega = d(t\, \lambda) + \Omega
\end{equation*}
is a symplectic form on~$W$.  Lemma~\ref{normalform for weak collars}
guarantees that $\omega$ can always be put in this form without loss
of generality.  Moreover, assume $\Omega$ is positive on the torus
\begin{equation*}
  T := \T^2 \times \{0\} \subset M \;.
\end{equation*}
By shrinking $\epsilon$ if necessary, we can then assume without loss
of generality that $\omega$ is positive on each of the tori $\{t\}
\times \T^2 \times \{r\}$ for $t \in (-\epsilon,0]$ and $r \in
[-5\epsilon,5\epsilon]$.  Define the constant
\begin{equation}\label{eqn:area}
  A = \int_T \omega > 0 \;.
\end{equation}

Let us now define a family of $1$--forms on~$M$,
\begin{equation*}
  \lambda_\sigma = d\theta + g_\sigma(r)\, d\phi
\end{equation*}
for $\sigma \in [0,1]$, where $g_\sigma :\, [-5\epsilon,5\epsilon]\to
\R$ is a smooth $1$--parameter family of odd functions such that:
\begin{enumerate}
\item $g_\sigma(r) = r$, and $g_\sigma(0) = 0$ for all $\sigma \in
  [0,1]$ when $\abs{r} \ge 4\epsilon$,
\item $g_\sigma' > 0$ for all $\sigma \in (0,1]$,
\item $g_1(r) = r$ for all~$r$,
\item $g_0(r) = 0$ for all $\abs{r} \le 3\epsilon$.
\end{enumerate}
Then $\lambda_1 = \lambda$, $\lambda_\sigma$ is a contact form for all
$\sigma \in (0,1]$ and $\lambda_0$ defines a confoliation, which is
integrable in the region $\bigl\{ \abs{r} \le 3\epsilon \bigr\}$.  Let
$\xi_\sigma = \ker\lambda_\sigma$.  By shrinking $\epsilon$ again if
necessary, we can assume without loss of generality that each
$\xi_\sigma$ is sufficiently $C^0$--close to $\xi$ so that
\begin{equation*}
  \restricted{\Omega}{\xi_\sigma} > 0
\end{equation*}
for all $\sigma \in [0,1]$.

Next, choose a smooth cutoff function
\begin{equation*}
  \beta :\, [-5\epsilon,5\epsilon] \to [0,1]
\end{equation*}
that has support in $[-3\epsilon,3\epsilon]$ and is identically~$1$ on
$[-2\epsilon,2\epsilon]$.  We use this to define a smooth
$2$--parameter family of $1$--forms for $(\sigma,\tau) \in [0,1]
\times [0,1]$,
\begin{equation}\label{eqn:2parameter}
  \lambda_\sigma^\tau = (1 - \tau) \beta(r) \, dr +
  \bigl[ 1 - (1 - \tau) \beta(r) \bigr]\, \lambda_\sigma,
\end{equation}
and distributions $\xi_\sigma^\tau = \ker \lambda_\sigma^\tau$.  The
following lemma implies that $\xi_\sigma^\tau$ is a contact structure
whenever both $\sigma$ and $\tau$ are positive.

\begin{lemma}\label{lemma:contactForms}
  Suppose $f(r)$ and $g(r)$ are any two smooth real valued functions
  on $[-5\epsilon,5\epsilon]$ such that the $1$--form
  \begin{equation*}
    \alpha = f(r)\, d\theta + g(r)\, d\phi
  \end{equation*}
  on $\T^2 \times [-5\epsilon,5\epsilon]$ is contact.  Then for any $t
  \in [0,1)$, the $1$--form
  \begin{equation*}
    \alpha_t := t \beta(r) \, dr + \bigl[ 1 - t \beta(r) \bigr]\, \alpha
  \end{equation*}
  is also contact.
\end{lemma}
\begin{proof}
  Noting that $dr \wedge d\alpha = 0$, we compute
  \begin{equation*}
    \begin{split}
      \alpha_t \wedge d\alpha_t &= \bigl[ t\beta\, dr + (1 - t\beta)\,
      \alpha \bigr]
      \wedge \bigl[ (1 - t\beta) \, d\alpha - t \beta' \, dr \wedge \alpha \bigr]\\
      &= (1 - t\beta)^2 \, \alpha \wedge d\alpha \ne 0 \; . \qedhere
    \end{split}
  \end{equation*}
\end{proof}

By Gray's stability theorem, each of the contact structures
$\xi_\sigma^\tau$ for $\sigma,\tau > 0$ are related to $\xi = \xi_1^1$
by isotopies with support in $\T^2 \times [-4\epsilon, 4\epsilon]$.
Thus after a compactly supported isotopy, we can view $\xi$ as a small
perturbation of the confoliation $\bar{\xi} := \xi_0^0 =
\ker\bar{\lambda}$, where we define
\begin{equation*}
  \bar{\lambda} = \lambda_0^0 = \beta(r)\, dr + \bigl[ 1 - \beta(r) \bigr]\,
  \left( d\theta + g_0(r)\, d\phi \right) \;.
\end{equation*}
This $1$--form is identical to $\lambda$ in $\bigl\{ \abs{r} \ge
4\epsilon \bigr\}$, but defines a foliation in $\bigl\{ \abs{r} \le
3\epsilon \bigr\}$ and takes the especially simple form
\begin{equation*}
  \bar{\lambda} = dr
  \quad \text{in $\T^2 \times [-2\epsilon,2\epsilon]$} \;.
\end{equation*}
The main technical ingredient we need is then the following
deformation result.

\begin{proposition}\label{prop:localModel}
  Given the local model of a symplectic pre-Lagrangian torus $T
  \subset M = \p W$ described above, for any sufficiently large
  constant $C > 0$ there exists a symplectic form $\bar{\omega}$ on
  $W$ with the following properties:
  \begin{enumerate}
  \item $\bar{\omega} = \omega$ outside some compact neighborhood
    of~$T$ in~$W$,
  \item $\restricted{\bar{\omega}}{\bar{\xi}} > 0$,
  \item $\bar{\omega} = A\, d\phi \wedge d\theta + C\, dt \wedge dr$
    on some neighborhood of~$T$.
  \end{enumerate}
\end{proposition}
\begin{proof}
  Note that $\Omega$ restricts on each of the $2$--plane fields
  $\xi_\sigma^\tau = \ker \lambda_\sigma^\tau$ to a positive form.
  This is clear, because $\lambda_\sigma^\tau$ as defined in
  \eqref{eqn:2parameter} is pointwise a convex combination of $dr$ and
  $\lambda_\sigma$, where $\lambda_\sigma \wedge \Omega$ and $dr
  \wedge \Omega$ are both positive, the latter due to the assumption
  that the tori $\T^2 \times \{r\}$ are all symplectic, the former
  because $\epsilon$ was chosen small enough to guarantee that
  $\restricted{\Omega}{\xi_\sigma} > 0$.  We therefore find a smooth
  homotopy from $\xi$ to $\bar{\xi}$, supported in a neighborhood
  of~$T$, through confoliations on which $\Omega$ is always positive.

  Next, let us replace $\Omega$ by a cohomologous closed $2$--form
  that takes a much simpler form near~$T$.  Indeed, since $\int_T
  \Omega = A = \int_T A\, d\phi \wedge d\theta$ and $T$ generates
  $H_2(M)$, there exists a $1$--form $\eta$ on $M$ such that
  \begin{equation*}
    A\, d\phi \wedge d\theta = \Omega + d\eta \;.
  \end{equation*}
  Choose a smooth cutoff function $F : [-5\epsilon,5\epsilon] \to
  [0,1]$ that has compact support in $[-2\epsilon,2\epsilon]$ and
  equals~$1$ on $[-\epsilon,\epsilon]$, and define the closed
  $2$--form
  \begin{equation*}
    \bar{\Omega} = \Omega + d( F(r)\, \eta)\;,
  \end{equation*}
  which equals $\Omega$ outside of $\bigl\{ \abs{r} \le 2\epsilon
  \bigr\}$ and $A\, d\phi \wedge d\theta$ in $\bigl\{ \abs{r} \le
  \epsilon \bigr\}$.  We claim
  \begin{equation*}
    \restricted{\bar{\Omega}}{\bar{\xi}} > 0 \;.
  \end{equation*}
  Indeed, outside of the region $\bigl\{ \abs{r} \le 2\epsilon
  \bigr\}$ this statement is nothing new, and otherwise $\bar{\lambda}
  = dr$, so we compute
  \begin{equation*}
    \bar{\lambda} \wedge \bar{\Omega} = dr \wedge \bigl[ (1 - F(r)) \, \Omega
    + A F(r) \, d\phi \wedge d\theta \bigr] > 0\;.
  \end{equation*}
  
  The result now follows by applying Lemmas~\ref{lemma:interpolation2}
  and~\ref{lemma:interpolation}, in that order.  This deforms $\omega$
  near~$T$ to a symplectic structure of the form
  \begin{equation*}
    \bar{\omega} = d(\varphi\, \bar{\lambda}) + \bar{\Omega} \;,
  \end{equation*}
  where $\varphi(t;\phi,\theta;r)$ depends only on $t$ near $T$ and
  satisfies $\p_t\varphi > 0$, so plugging in the local formulas
  $\bar{\lambda} = dr$ and $\bar{\Omega} = A\, d\phi \wedge d\theta$,
  the above becomes
  \begin{equation*}
    \bar{\omega} = \p_t\varphi\, dt \wedge dr + A\, d\phi \wedge d\theta \;.
  \end{equation*}
  One can also easily arrange for $\p_t\varphi$ to be constant
  near~$T$ so long as it is sufficiently large, and the result is thus
  proved.
\end{proof}

\begin{proof}[Proof of Theorem~\ref{thm:LutzTwists}]
  The following argument generalizes the construction of weak fillings
  on tight $3$--tori described by Giroux \cite{Giroux_plusOuMoins}.
  Consider the confoliation $\bar{\xi}$ and deformed symplectic
  structure $\bar{\omega}$ constructed in
  Proposition~\ref{prop:localModel}.  Then $\bar{\omega}$ is also
  positive on any contact structure $\xi'$ that is sufficiently
  $C^0$--close to~$\bar{\xi}$ as a distribution.  It suffices
  therefore to find, for any $n \in \N$, a contact structure $\xi_n$
  that is $C^0$--close to $\bar{\xi}$ and isotopic to the one obtained
  by performing~$n$ Lutz twists on~$\xi$ along~$T$.  This is easy: for
  $\sigma \in [0,1]$, define a smooth family of confoliation
  $1$--forms $\alpha_\sigma$ which match $\lambda_\sigma$ outside the
  coordinate neighborhood $\T^2 \times [-\epsilon,\epsilon]$, and in
  $\T^2 \times [-\epsilon,\epsilon]$ are contact and take the form
  \begin{equation*}
    F_\sigma(r)\, d\theta + G_\sigma(r)\, d\phi \;,
  \end{equation*}
  such that the curve $r \mapsto (F_0(r),G_0(r)) \in \R^2$ winds~$n$
  times counterclockwise about the origin for $r \in
  [-\epsilon,\epsilon]$.  Then $\alpha_\sigma$ is contact for every
  $\sigma \in (0,1]$ and defines a contact structure isotopic to the
  one we are interested in.  It follows now from
  Lemma~\ref{lemma:contactForms} that for all $\sigma \in (0,1]$ and
  $\tau \in (0,1]$,
  \begin{equation*}
    \alpha_\sigma^\tau := (1 - \tau) \beta(r) \, dr +
    \left[ 1 - (1 - \tau) \beta(r) \right] \alpha_\sigma
  \end{equation*}
  is a contact form, but as $\sigma \to 0$ and $\tau \to 0$ it
  converges to~$\bar{\lambda}$.
\end{proof}

\begin{proof}[Proof of Theorem~\ref{thm:handles}]
  We assume $(W,\omega)$ is a symplectic manifold with weakly contact
  boundary $(M,\xi)$, and $T_+ , T_- \subset M \subset W$ are oriented
  tori which are pre-Lagrangian in $(M,\xi)$ and symplectic in
  $(W,\omega)$, such that
  \begin{equation*}
    \int_{T_-} \omega = \int_{T_+} \omega = A > 0 \;.
  \end{equation*}
  Then for a sufficiently large constant~$C > 0$, we can use
  Proposition~\ref{prop:localModel} to deform $\omega$ near~$T_+$
  and~$T_-$ to a new symplectic structure $\bar{\omega}$, which takes
  the form
  \begin{equation*}
    \bar{\omega} = C\, dt \wedge dr + A\, d\phi \wedge d\theta
  \end{equation*}
  in local coordinates near $T_+$ and $T_-$, and satisfies
  $\restricted{\bar{\omega}}{\bar{\xi}} > 0$.  Here $\bar{\xi}$ is a
  confoliation with the following properties:
  \begin{itemize}
  \item $\bar{\xi} = \xi$ outside a small coordinate neighborhood $N
    \subset M$ of $T_+ \cup T_-$,
  \item $\bar{\xi}$ admits a $C^0$--small perturbation to a contact
    structure, which is isotopic to~$\xi$ by an isotopy supported
    in~$N$,
  \item $\bar{\xi} = \ker dr$ on an even smaller coordinate
    neighborhood of $T_+ \cup T_-$.
  \end{itemize}
  Choose $\delta > 0$ sufficiently small so that the coordinate
  neighborhoods $\T^2 \times [-\delta,\delta]$ of $T_-$ and $T_+$ are
  contained in the region where $\bar{\xi} = \ker dr$ and
  $\bar{\omega} = C\, dt \wedge dr + A\, d\phi \wedge d\theta$.  Then
  we define the following smooth model of a toroidal $1$--handle (see
  Figure~\ref{fig: sketch of level sets}):
  \begin{equation*}
    \handle_\delta = \Bigl\{ (\phi,\theta; r,r') \in \T^2 \times [-\delta,\delta]
    \times [-\delta,\delta] \, \Bigm|\ \abs{r} \le h(r') \Bigr\} \;,
  \end{equation*}
  where $h :\, [-\delta,\delta] \to (0,\delta]$ is a continuous, even
  and convex function that is smooth on $(-\delta,\delta)$ and has its
  derivative blowing up at $r' = \pm\delta$, such that its graph
  merges smoothly into the lines $r' = \pm\delta$.  Denote the smooth
  pieces of $\p\handle_\delta$ by
  \begin{equation*}
    \p^N \handle_\delta = \{ r' = +\delta \}, \, \p^S \handle_\delta =  \{ r' = -\delta \},\,
    \p^W\handle_\delta = \{ r = - h(r') \}, \, \text{ and }\, \p^E\handle_\delta = \{ r = + h(r') \} \;.
  \end{equation*}
  This model can be attached smoothly to~$W$ as in Figure~\ref{fig:
    sketch of level sets}, so that
  \begin{equation*}
    W' := W \cup \handle_\delta
  \end{equation*}
  has smooth boundary $M' := \p W'$.  The symplectic structure
  $\bar{\omega}$ then extends to $W'$ by defining
  \begin{equation*}
    \bar{\omega} = C\, dr' \wedge dr + A\, d\phi \wedge d\theta
  \end{equation*}
  on~$\handle_\delta$, which restricts positively to the smooth
  confoliation $\bar{\xi}'$ on~$M'$ defined by
  \begin{equation*}
    \bar{\xi}' = \begin{cases}
      \bar{\xi} & \text{ on $M \setminus (\p^N \handle_\delta \cup \p^S \handle_\delta)$,}\\
      T(\T^2 \times \{*\}) & \text{ on $\p^W\handle_\delta \cup \p^E \handle_\delta$.}
    \end{cases}
  \end{equation*}
  The latter admits a $C^0$--small perturbation to a contact form
  which is isotopic to the one obtained by splicing $(M,\xi)$ along
  $T_+$ and~$T_-$.
\end{proof}

\begin{wrapfigure}{r}{0.4\textwidth}
  \vspace{-10pt}
  \begin{center}
    \includegraphics[width=0.38\textwidth,
    keepaspectratio]{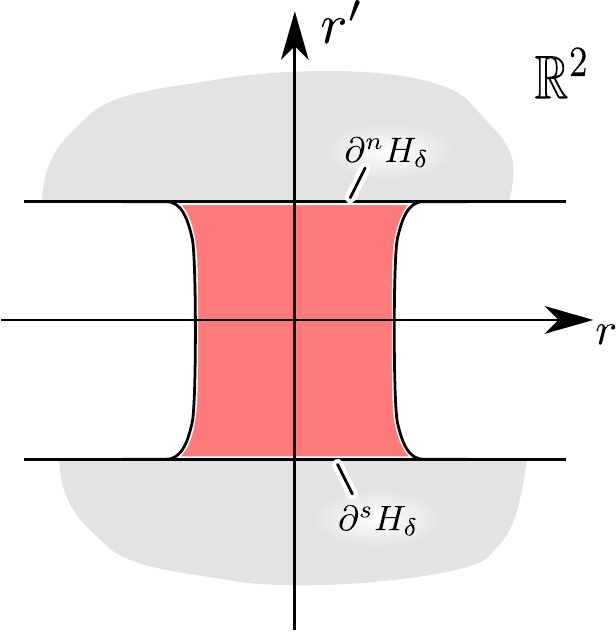}
  \end{center}
  \vspace{-10pt}
  \caption{The handle $\handle_\delta$ is attached in the ambient
    space $\T^2 \times \R^2$ to two model neighborhoods.}
  \label{fig: sketch of level sets}
\end{wrapfigure}

\subsection{Proof of Theorem~\ref{thm:weakFillings}}
\label{subsec:fillingsConstruction}

Assume $\Sigma = \Sigma_+ \cup_\Gamma \Sigma_-$ is a closed oriented
surface that is the union of two surfaces with boundary along a
multicurve $\Gamma \subset \Sigma$ whose connected components are all
nonseparating, and let $\bigl(P_{\Sigma,e}, \xi_{\Gamma,e}\bigr)$
denote the $\SS^1$--principal bundle $P_{\Sigma,e}$ over $\Sigma$ with
Euler number $e$ together with the $\SS^1$--invariant contact
structure $\xi_{\Gamma,e}$ that is everywhere transverse to the
$\SS^1$--fibers with exception of the tori that lie over the
multicurve $\Gamma$.  Under these assumptions, we will use the handle
attaching technique described in \secref{subsec:attaching} to
construct a weak filling of $\bigl(P_{\Sigma,e},
\xi_{\Gamma,e}\bigr)$.  The idea is to obtain $\bigl(P_{\Sigma,e},
\xi_{\Gamma,e}\bigr)$ by a sequence of splicing operations from a
simpler disconnected contact manifold for which a (disconnected)
strong filling is easy to construct by hand.  For this strong filling,
the components of~$\Gamma$ give rise to pre-Lagrangian tori, and the
significance of the nonseparating assumption will be that it allows us
to perturb the strong filling to a weak one for which these tori
become symplectic, and are thus suitable for handle attaching.

The building blocks are obtained in the following way.  Let $S$ be a
connected, oriented compact surface with non-empty boundary.  The
\emph{symmetric double} of $S$ is the closed surface
\begin{equation*}
  S^D := S \cup_{\p S} \overline{S} \;,
\end{equation*}
where $\overline{S}$ is a second copy of $S$ taken with reversed
orientation, and the two are glued along their boundaries via the
identity map.  The multicurve $\Gamma_S := \p S \subset S^D$
determines an $\SS^1$--invariant contact manifold $(\SS^1 \times S^D,
\xi_{\Gamma_S})$ in the standard way.

\begin{proposition}
  The contact manifold $(\SS^1 \times S^D, \xi_{\Gamma_S})$ obtained
  from a symmetric double has a strong symplectic filling homeomorphic
  to $[0,1] \times \SS^1 \times S$.
\end{proposition}
\begin{proof}
  Regard $S$ together with a positive volume form $\Omega_S$ as a
  symplectic manifold.  Choose a plurisubharmonic Morse function $f:\,
  S\to [0,C]$ whose critical values all lie in the interval
  $[0,\epsilon]$ with $\epsilon < C$, such that $f^{-1}(C) = \p S$.
  Take now the annulus $\R\times \SS^1$ with symplectic form $dx\wedge
  d\phi$, and with plurisubharmonic function $g(x,\phi) = x^2$.  The
  product manifold
  \begin{equation*}
    \bigl((\R\times \SS^1) \times S, \Omega + dx\wedge d\phi\bigr)
  \end{equation*}
  is a symplectic manifold with a plurisubharmonic function given by
  $F := f + x^2$.  The critical values of this function all lie in
  $[0,\epsilon]$, so that $N := F^{-1}(C)$ will be a smooth compact
  hypersurface.  In fact, it is easy to see that $N$ is diffeomorphic
  to $\SS^1 \times S^D$.  The standard circle action on the annulus
  $\R\times \SS^1$ splits off naturally, so that $N$ is the product of
  a circle with a closed surface.

  We can explicitly give two embeddings of the $3$--manifold
  $\SS^1\times S$ into $F^{-1}(C) \subset (\R\times \SS^1) \times S$
  as the graphs of the two maps $\SS^1 \times S \to (\R\times \SS^1)
  \times S, \, (\phi, p) \mapsto \bigl(\pm\sqrt{C - f(p)},\phi,
  p\bigr)$ distinguished by the different signs in front of the square
  root.  The boundary of $\SS^1\times S$ is mapped by both maps to the
  set $\{0\}\times \SS^1 \times \p S$ so that the two copies are glued
  along their boundary.

  The contact form is defined as $\alpha := - \restricted{d^cF}{TN} =
  - \restricted{dF\circ J}{TN}$.  It is $\SS^1$--invariant (for the
  standard complex structure on $\R\times \SS^1$), and the vector
  $\p_\phi$ is parallel to $N$, and never lies in the kernel of
  $\alpha$ with the exception of the points where $d(x^2)$ vanishes,
  which happens to be exactly along the boundary of $S$.  By
  \cite{Lutz_CircleActions}, this proves that the hypersurface $N$ is
  contactomorphic to $(\SS^1 \times S^D, \xi_{\Gamma_S})$.
\end{proof}

Now denote by
\begin{equation*}
  \Sigma_1,\dotsc,\Sigma_N
\end{equation*}
the closures of the connected components of $\Sigma \setminus \Gamma$,
whose boundaries $\Gamma_j := \p\Sigma_j$ are all disconnected due to
the assumption that components of~$\Gamma$ are nonseparating.  Then
for each $j=1,\dotsc,N$, construct the doubled manifold $\Sigma_j^D$,
and define the disconnected contact manifold
\begin{equation*}
  (M_0,\xi_0) = \bigsqcup_{j=1}^N (\SS^1 \times \Sigma_j^D,
  \xi_{\Gamma_j}) \;, 
\end{equation*}
which by the proposition above can be strongly filled.  Let
$(W_j,\omega_j)$ denote the resulting strong filling of $\SS^1 \times
\Sigma_j^D$.  For each connected component $\gamma \subset \Gamma_j$,
which is also a component of~$\Gamma$, the torus $\SS^1 \times \gamma
\subset \p W_j$ is a Lagrangian submanifold in $(W_j,\omega_j)$.

\begin{lemma}
  There exists a cohomology class $[\beta] \in
  H^2_\dR\bigl(P_{\Sigma,e}\bigr)$ such that $\int_T \beta \ne 0$ for
  every torus $T$ that lies over a connected component $\gamma \subset
  \Gamma$.
\end{lemma}
\begin{proof}
  By Poincaré duality, it suffices to find a homology class $A \in
  H_1\bigl(P_{\Sigma,e}; \R\bigr)$ whose intersection number $A
  \bullet [T_\gamma] \in \R$ is nonzero for every torus $T_\gamma$
  lying over a connected component $\gamma \subset \Gamma$.  For each
  component $\gamma\in \Gamma$, pick an oriented loop $C_\gamma$
  in~$P_{\Sigma,e}$ with $[C_\gamma] \bullet [T_\gamma] = 1$; this
  necessarily exists since $\gamma$ and hence also $T_\gamma$ is
  nonseparating.  Then we construct~$A$ by the following algorithm:
  starting with any connected component $\gamma_1 \subset \Gamma$, let
  $A_1 = [C_{\gamma_1}]$.  Then $A_1 \bullet [T_\gamma] \ne 0$ for
  some subcollection of the components $\gamma \subset \Gamma$,
  including~$\gamma_1$.  If there remains a component $\gamma_2
  \subset \Gamma$ such that $A_1 \bullet [T_{\gamma_2}] = 0$, then we
  set
  \begin{equation*}
    A_2 = A_1 + d_2\, [C_{\gamma_2}] \;,
  \end{equation*}
  where $d_2 > 0$ is chosen sufficiently small so that for every
  component $\gamma \subset \Gamma$ with $A_1 \bullet [T_\gamma]$
  nonzero, $A_2 \bullet [T_\gamma]$ is also nonzero.  The result is
  that $A_2 \bullet [T_\gamma]$ is nonzero for a strictly larger set
  of components than $A_1 \bullet [T_\gamma]$.  Thus after repeating
  this process finitely many times, we eventually find $A \in
  H_1\bigl(P_{\Sigma,e}; \R\bigr)$ with all intersection numbers $A
  \bullet [T_\gamma]$ nonzero.
\end{proof}

Using the cohomology class $[\beta]$ given by the lemma, orient every
torus $T_\gamma \subset P_{\Sigma,e}$ that projects onto a connected
component $\gamma \subset \Gamma$ in such a way that $\int_{T_\gamma}
\beta > 0$.  We find a closed $2$--form $\sigma$ representing
$[\beta]$ that is positive on each of the oriented pre-Lagrangian tori
$T_\gamma$.  Since every component $\Sigma_j$ has non-empty boundary,
it follows that the restriction $\restricted{P_{\Sigma,e}}{\Sigma_j}$
is trivial so that we can identify it with
\begin{equation*}
  \restricted{P_{\Sigma,e}}{\Sigma_j} \cong \SS^1\times \Sigma_j \;,
\end{equation*}
and we can then pull-back $\sigma$ to each component $\SS^1 \times
\Sigma_j$ to obtain a collection of $2$--forms $\sigma_j$ on the
fillings $W_j \cong [0,1] \times \SS^1 \times \Sigma_j$, all of which
are positive on the tori $\SS^1 \times \gamma \subset W_j$.  The same
is then true for the $2$--forms $\omega_j + \epsilon\,\sigma_j$, with
$\epsilon > 0$ chosen sufficiently small so that
\begin{equation*}
  (W_0,\omega_0) := \bigsqcup_{j=1}^N (W_j , \omega_j + \epsilon\, \sigma_j)
\end{equation*}
is a weak filling of $(M_0,\xi_0)$.

Observe now that each torus $T_\gamma$ for a connected component
$\gamma \subset \Gamma$ corresponds to \emph{two} pre-Lagrangian tori
in $(M_0,\xi_0)$, which are symplectic in $(W_0,\omega_0)$ and have
matching integrals of~$\omega_0$ by construction.  This allows us to
attach toroidal $1$--handles to $(W_0,\omega_0)$ along corresponding
pairs of tori via Theorem~\ref{thm:handles}, which by
Remark~\ref{remark:circle_action_and_splicing} can be done in a way
that is compatible with circle actions.  To prescribe the isotopy
class of the gluing maps, choose for all except one of the
tori~$T_\gamma \subset M_0$ the curves $\{*\}\times \p\Sigma_j$ as the
transverse cycle.  This way the splicing will glue the sections $\{*\}
\times \Sigma_j$ together smoothly along each of the pre-Lagrangian
tori.  If the transverse cycle is also chosen to be of the form
$\{*\}\times \p\Sigma_j$ on the last torus, then the section will in
fact glue to a global section, and the resulting manifold will be a
weak filling of two disjoint copies of the contact manifold $(\SS^1
\times \Sigma, \xi_\Gamma)$.  If we instead choose a different
transverse cycle on the last torus, we obtain a connected symplectic
manifold with weak contact boundary consisting of the disjoint union
of the circle bundles $\bigl(P_{\Sigma,e}, \xi_{\Gamma,e}\bigr)$ and
$\bigl(P_{\Sigma,-e}, \xi_{\Gamma,-e}\bigr)$ with the corresponding
contact structures.  We claim that the Euler number $e$ is given by
the intersection number of the two sections touching the last
pre-Lagrangian torus $T_0$, which is equal to the intersection number
of the chosen transverse cycle with the curve $\{*\}\times \p\Sigma_j$
.  The Euler number is obtained by chosing a section over a disk $D$,
a section over the complement of this disk, and computing the
intersection number of both sections in the torus that lies over the
boundary $\p D$.  Our construction yields so far a section of the
spliced manifold defined everywhere except at the last pre-Lagrangian
torus $T_0$.  We can push both parts of the section a bit away from
$T_0$, and connect them with a strip crossing this torus.  The new
section obtained this way is defined over the whole surface $\Sigma$
with the exception of a disk~$D$, and it is easy to see that the
intersection number between the section we have just constructed, and
a section over $D$ is equal to the intersection number of the two
initial sections in the pre-Lagrangian torus $T_0$.

Finally, capping the weak contact boundary $\bigl(P_{\Sigma,-e},
\xi_{\Gamma,-e}\bigr)$ using \cite{Eliashberg_capping,
  Etnyre_capping}, we obtain a weak filling of $\bigl(P_{\Sigma,e},
\xi_{\Gamma,e}\bigr)$, thus the proof of
Theorem~\ref{thm:weakFillings} is complete.

\bibliography{main}


\end{document}